%% file: arxiv.tex
\documentclass{article}

\PassOptionsToPackage{numbers, compress}{natbib}

\usepackage[final]{neurips_2023}
\input{macros.tex}

\usepackage{amsmath,amsfonts,amssymb,amsthm,array}
\usepackage{algorithm}
\usepackage{caption}
\usepackage{subcaption}
\usepackage{color}
\usepackage{graphicx}
\usepackage{enumitem}
\usepackage{nicefrac}

\usepackage{tikz}

\theoremstyle{plain}

\theoremstyle{remark}

\usepackage{mdframed} 
\usepackage{thmtools}

\usepackage[flushleft]{threeparttable} 

\usepackage{caption}
\usepackage{multirow}
\usepackage{colortbl}
\definecolor{bgcolor}{rgb}{0.8,1,1}
\definecolor{bgcolor2}{rgb}{0.8,1,0.8}
\definecolor{niceblue}{rgb}{0.0,0.19,0.56}

\usepackage{hyperref}
\hypersetup{colorlinks,linkcolor={blue},citecolor={niceblue},urlcolor={blue}}

\usepackage{natbib}
\usepackage{sidecap}

\definecolor{shadecolor}{gray}{0.9}
\declaretheoremstyle[
headfont=\normalfont\bfseries,
notefont=\mdseries, notebraces={(}{)},
bodyfont=\normalfont,
postheadspace=0.5em,
spaceabove=1pt,
mdframed={
  skipabove=8pt,
  skipbelow=8pt,
  hidealllines=true,
  backgroundcolor={shadecolor},
  innerleftmargin=4pt,
  innerrightmargin=4pt}
]{shaded}

\usepackage{xspace}

\usepackage[utf8]{inputenc} 
\usepackage[T1]{fontenc}    
\usepackage{hyperref}       
\usepackage{url}            
\usepackage{booktabs}       
\usepackage{amsfonts}       
\usepackage{nicefrac}       
\usepackage{microtype}      
\usepackage{xcolor}         

\title{Single-Call Stochastic Extragradient Methods \\for Structured Non-monotone Variational Inequalities: 
\\Improved Analysis under Weaker Conditions}

\author{%
Sayantan Choudhury \\
AMS \& MINDS \\
Johns Hopkins University
\And
Eduard Gorbunov \\
MBZUAI
\And Nicolas Loizou \\
AMS \& MINDS \\
Johns Hopkins University
}

\begin{document}

\maketitle
\begin{abstract}
Single-call stochastic extragradient methods, like stochastic past extragradient (\algname{SPEG}) and stochastic optimistic gradient (\algname{SOG}),  have gained a lot of interest in recent years and are one of the most efficient algorithms for solving large-scale min-max optimization and variational inequalities problems (VIP) appearing in various machine learning tasks. However, despite their undoubted popularity, current convergence analyses of \algname{SPEG} and \algname{SOG} require strong assumptions like bounded variance or growth conditions. In addition, several important questions regarding the convergence properties of these methods are still open, including mini-batching, efficient step-size selection, and convergence guarantees under different sampling strategies. In this work, we address these questions and provide convergence guarantees for two large classes of structured non-monotone VIPs: (i) quasi-strongly monotone problems (a generalization of strongly monotone problems) and (ii) weak Minty variational inequalities (a generalization of monotone and Minty VIPs). We introduce the expected residual condition, explain its benefits, and show how it allows us to obtain a strictly weaker bound than previously used growth conditions, expected co-coercivity, or bounded variance assumptions. Finally, our convergence analysis holds under the arbitrary sampling paradigm, which includes importance sampling and various mini-batching strategies as special cases.
\end{abstract}
\vspace{-2mm}
\section{Introduction}\label{Introduction}
Differentiable game formulations where several parameterized models/players compete to minimize their respective objective functions have recently gained much attention from the machine learning community. Some landmark advances in machine learning that are framed as games (or in their simplified form as min-max optimization problems) are Generative Adversarial Networks (GANs) \citep{goodfellow2014generative, arjovsky2017wasserstein}, adversarial training of neural networks \citep{madry2018towards, wang2021adversarial}, reinforcement learning \citep{brown2020combining, sokota2022unified}, and distributionally robust learning \citep{namkoong2016stochastic,yu2022fast}.

In this work, we consider a more abstract formulation of the problem and focus on solving the following unconstrained stochastic variational inequality problem (VIP): 
\begin{equation}\label{eq: Variational Inequality Definition}
\text{Find } x^* \in \mathbb{R}^d: \text{such that } F(x^*) = \frac{1}{n} \sum_{i = 1}^n F_i(x^*) = 0 
\end{equation}
where each $F_i: \mathbb{R}^d \to \mathbb{R}^d $ is a Lipschitz continuous operator. Problem~\eqref{eq: Variational Inequality Definition} generalizes the solution of several types of \emph{stochastic smooth games}~\citep{facchinei2007games, loizou2021stochastic, gorbunov2022stochastic, beznosikov2022smooth}. The simplest example is the unconstrained min-max optimization problem (also called a \emph{zero-sum} game):
\begin{equation}
\label{eq: MinMax}
\min_{x_1 \in\R^{d_1}} \max_{x_2 \in\R^{d_2}}\frac{1}{n} \sum_{i=1}^n g_i(x_1,x_2) \, ,
\end{equation}
where each component function $g_i:\R^{d_1} \times \R^{d_2}\rightarrow \R$ is assumed to be smooth. In this scenario, operator $F_i$ of \eqref{eq: Variational Inequality Definition} represents the appropriate concatenation of the block-gradients of $g_i$: $F_i(x) := (\nabla_{x_1} g_i(x_1,x_2); -\nabla_{x_2} g_i(x_1,x_2) )$, where $x := (x_1; x_2)$. Solving~\eqref{eq: Variational Inequality Definition} then amounts to finding a stationary point $x^* = (x_1^*; x_2^*)$ for~\eqref{eq: MinMax}, which under a convex-concavity assumption for $g_i$, implies that it is a global solution for the min-max problem. 

However, in modern machine learning applications, game-theoretical formulations that are special cases of problem~\eqref{eq: Variational Inequality Definition} are rarely monotone. That is, the min-max optimization problem~\eqref{eq: MinMax} does not satisfy the popular and well-studied convex-concave setting. For this reason, the ML community started focusing on non-monotone problems with extra structural properties.\footnote{The computation of approximate first-order locally optimal solutions for general non-monotone problems (without extra structure) is intractable. See \cite{daskalakis2021complexity} and \cite{diakonikolas2021efficient} for more details.} In this work, we focus on such settings (structured non-monotone operators) for which we are able to provide tight convergence guarantees and avoid the standard issues (like cycling and divergence of the methods) appearing in the more general non-monotone regime. In particular, we focus on understanding and efficiently analyze the performance of single-call extragradient methods for solving (i) $\mu$-quasi-strongly monotone VIPs~\citep{loizou2021stochastic,beznosikov2022stochastic} and (ii) weak Minty variational inequalities~\citep{diakonikolas2021efficient, lee2021fast}.

\paragraph{Classes of structured non-monotone VIPs.} Throughout this work we assume that operator $F$ in \eqref{eq: Variational Inequality Definition} is $L$- Lipschitz i.e. $\forall x, y \in \R^d$ operator $F$ satisfy $\|F(x) - F(y)\| \leq L \|x - y\|$.

As we have already mentioned, in this work, we deal with two classes of structured non-monotone problems: the $\mu$-quasi strongly monotone VIPs and the weak Minty variational inequalities. 
\begin{definition}
$F$ is said to be $\mu$-quasi strongly monotone if there is $\mu >0$ such that:
\begin{equation}\label{eq: Strong Monotonicity}
\forall x \in \mathbb{R}^d \qquad    \la F(x), x - x^* \ra \geq \mu \|x - x^*\|^2.
\end{equation}
\end{definition}
Condition~\eqref{eq: Strong Monotonicity} is a relaxation of $\mu$-strong monotonicity, and it includes several non-monotone games as special cases~\cite{loizou2021stochastic}. Inequality \eqref{eq: Strong Monotonicity} can be seen as an extension of the popular quasi-strong convexity assumption from optimization literature \citep{necoara2019linear, gower2019sgd} to the VIPs~\cite{loizou2021stochastic}. 
In the literature of variational inequality problems, quasi strongly monotone problems are also known as strong coherent VIPs~\citep{song2020optimistic} or VIPs satisfying the strong stability condition~\citep{mertikopoulos2019learning}, or strong Minty variational inequality~\citep{diakonikolas2021efficient}. 

One of the weakest possible assumptions on the structure of non-monotone VIPs is the weak Minty variational inequality~\citep{diakonikolas2021efficient}. 

\begin{definition}
We say weak Minty Variational Inequality (MVI) holds for $F$ if for some $\rho > 0$ :
\begin{equation}\label{eq: weak MVI}
    \forall x\in \mathbb{R}^d \qquad \la F(x), x - x^*\ra \geq - \rho \|F(x)\|^2.
\end{equation}
\end{definition}
To the best of our knowledge, the weak Minty variational inequality~\eqref{eq: weak MVI} as an assumption was first introduced in~\cite{diakonikolas2021efficient}.  The more popular and extensively studied Minty variational inequality~\citep{dang2015convergence, liu2019towards, liu2021first, mertikopoulos2018optimistic} is a particular case of \eqref{eq: weak MVI} with $\rho = 0$. In addition, the weak MVI condition is implied by the negative
comonotonicity~\cite{bauschke2021generalized} or, equivalently, the positive cohypomonotonicity~\cite{combettes2004proximal}. 
Finally, when we focus on min-max optimization problems~\eqref{eq: MinMax}, weak MVI condition (with $\rho = 0$) is satisfied for several non-convex non-concave families of min-max objectives, including quasi-convex quasi-concave or star convex- star concave~\citep{gorbunov2022stochastic}. Extragradient-type methods for solving VIPs satisfying the weak MVI have been proposed in \cite{diakonikolas2021efficient, pethick2022escaping} and  \cite{bohm2022solving}.

\subsection{Main Contributions}
\begin{table*}[t]
    \centering
    \scriptsize
    \caption{\scriptsize 
    Summary of known and new convergence results for versions of \algname{SEG} and \algname{SPEG} with constant step-sizes applied to solve quasi-strongly monotone variational inequalities and variational inequalities with operators satisfying Weak Minty condition. Columns: ``Setup'' = quasi-strongly monotone or Weak MVI; ``No UBV?'' = is the result derived without bounded variance assumption?; ``Single-call'' = does the method require one oracle call per iteration?; ``Convergence rate'' = rate of convergence neglecting numerical factors. Notation: $K$ = number of iterations; $L_{\max} = \max_{i\in [n]} L_i$, where $L_i$ is a Lipschitz constant of $F_i$; $\overline{\mu} = \frac{1}{n}\sum_{i=1}^n\mu_i$, where $\mu_i$ is quasi-strong monotonicity constant of $F_i$ (see details in \citep{gorbunov2022stochastic}); $\sigma_{\text{US}*}^2 = \frac{1}{n}\sum_{i=1}^n \|F_i(x^*)\|^2$; $\overline{L} = \frac{1}{n}\sum_{i=1}^n L_i$; $\sigma_{\text{IS}*}^2 = \frac{1}{n}\sum_{i=1}^n \frac{\overline{L}}{L_i}\|F_i(x^*)\|^2$; $L$ = Lipschitz constant of $F$; $\mu$ = quasi-strong monotonicity constant of $F$; $\delta, \sigma_*^2$ = parameters from \eqref{eq: variance bound}; $\rho$ = parameter from Weak Minty condition; $\tau$ = batchsize.}
    \label{tab:comparison_of_rates}
    \begin{threeparttable}
    \resizebox{\columnwidth}{!}{%
        \begin{tabular}{|c|c|c c c|}
        \hline
        Setup & Method & No UBV? & Single-call? & Convergence rate 
        \\
        \hline\hline
        \multirow{7}{2cm}{\centering Quasi-strong mon.} & \begin{tabular}{c}
            \algname{S-SEG-US}\\
            \citep{gorbunov2022stochastic}
        \end{tabular} & \cmark\tnote{{\color{blue}(1)}} & \xmark & $\frac{L_{\max}}{\overline{\mu}}\exp\left(- \frac{\overline{\mu}}{L_{\max}}K\right) + \frac{\sigma_{\text{US}*}^2}{\overline{\mu}^2 K}$\\
        & \begin{tabular}{c}
            \algname{S-SEG-IS}\\
            \citep{gorbunov2022stochastic}
        \end{tabular} & \cmark\tnote{{\color{blue}(1)}} & \xmark & $\frac{\overline{L}}{\overline{\mu}}\exp\left(- \frac{\overline{\mu}}{\overline{L}}K\right) + \frac{\sigma_{\text{IS}*}^2}{\overline{\mu}^2 K}$\\
        & \begin{tabular}{c}
            \algname{SPEG}\\
            \citep{hsieh2019convergence}
        \end{tabular} & \xmark\tnote{{\color{blue}(2)}} & \cmark & $\frac{L}{\mu}\exp\left(- \frac{\mu}{L}K\right) + \frac{\sigma_*^2}{\mu^2 K}$\tnote{{\color{blue}(3)}}\\
        &\cellcolor{bgcolor2}\begin{tabular}{c}
            \algname{SPEG}\\
            (This work)
        \end{tabular} & \cellcolor{bgcolor2}\cmark & \cellcolor{bgcolor2}\cmark & \cellcolor{bgcolor2} $\max\left\{\frac{L}{\mu}, \frac{\delta}{\mu^2}\right\}\exp\left(- \min\left\{\frac{\mu}{L}, \frac{\mu^2}{\delta}\right\}K\right) + \frac{\sigma_*^2}{\mu^2 K}$ \\

        \hline\hline
        \multirow{6}{2cm}{\centering Weak MVI\tnote{{\color{blue}(4)}}} & \begin{tabular}{c}
            \algname{SEG+}\\
            \citep{diakonikolas2021efficient}
        \end{tabular} & \xmark\tnote{{\color{blue}(2)}} & \xmark & $\frac{L^2\|x_0 - x^*\|^2}{K(1 - 8\sqrt{2}L\rho)} + \frac{\sigma_*^2}{\tau (1 - 8\sqrt{2}L\rho)}$ \tnote{{\color{blue}(5)}}\\
        & \begin{tabular}{c}
            \algname{OGDA+}\\
            \citep{bohm2022solving}
        \end{tabular} & \xmark\tnote{{\color{blue}(2)}} & \cmark & $\frac{\|x_0 - x^*\|^2}{Kac(a - \rho)} + \frac{\sigma_*^2}{\tau L^2 ac(a - \rho)}$ \tnote{{\color{blue}(6)}}\\
        & \cellcolor{bgcolor2}\begin{tabular}{c}
            \algname{SPEG}\\
            (This work)
        \end{tabular} & \cellcolor{bgcolor2}\cmark & \cellcolor{bgcolor2}\cmark &\cellcolor{bgcolor2} $\frac{\left(1 + \frac{48\omega\gamma\delta}{\tau(1-L\gamma)^2}\right)^K\|x_0 - x^*\|^2}{K \omega\gamma(1-L(\gamma+4\omega))} + \frac{\left(1 + \frac{1-L\gamma}{K}\left(1 + \frac{48\omega\gamma\delta}{\tau(1-L\gamma)^2}\right)^K\right)\sigma_*^2}{\tau(1-L\gamma)(1-L(\gamma+4\omega))}$ \tnote{{\color{blue}(7)}} \; \\
        \hline
    \end{tabular}%
    }
    \begin{tablenotes}
        {\scriptsize\item [{\color{blue}(1)}] Quasi-strong monotonicity of all $F_i$ is assumed.
        \item [{\color{blue}(2)}] It is assumed that \eqref{eq: variance bound} holds with $\delta = 0$.
        \item [{\color{blue}(3)}] \cite{hsieh2019convergence} do not derive this result but it can be obtained from their proof using standard choice of step-sizes.
        \item [{\color{blue}(4)}] All mentioned results in this case require large batchsizes $\tau = \cO(K)$ to get $\cO(\nicefrac{1}{K})$ rate.
        \item [{\color{blue}(5)}] The result is derived for $\rho < \nicefrac{1}{8\sqrt{2}L}$.
        \item [{\color{blue}(6)}] The result is derived for $\rho < \nicefrac{3}{8L}$. Here $a$ and $c$ are assumed to satisfy $aL \leq \frac{7 - \sqrt{1 + 48c^2}}{8(1+c)}$, $c > 0$ and $a > \rho$.
        \item [{\color{blue}(7)}] The result is derived for $\rho < \nicefrac{1}{2L}$. Here we assume that $\max\{2\rho, \nicefrac{1}{(2L)}\} < \gamma < \nicefrac{1}{L}$ and $0 < \omega < \min\{\gamma - 2\rho, \nicefrac{(4-\gamma L)}{4L}\}$.
        }
        \vspace{-4mm}
    \end{tablenotes}
    \end{threeparttable}
\end{table*}
\vspace{-3mm}
Our main contributions are summarized below.
\vspace{-2mm}
\begin{itemize}[leftmargin=*]
\setlength{\itemsep}{0pt}
\item \textbf{Expected Residual.} We propose the expected residual \eqref{eq: ER Condition} condition for stochastic variational inequality problems \eqref{eq: Variational Inequality Definition}. We explain the
benefits of \ref{eq: ER Condition} and show how it can be used to derive an upper bound on $\E \|g(x)\|^2$ (see Lemma~\ref{Lemma: variance bound}) that it is strictly weaker than the bounded variance assumption and “growth conditions” previously used for the analysis of stochastic algorithms for solving \eqref{eq: Variational Inequality Definition}. We prove that \ref{eq: ER Condition} holds for a large class of operators, i.e.,  whenever  $F_i$ of \eqref{eq: Variational Inequality Definition} are Lipschitz continuous.

\item \textbf{Novel Convergence Guarantees.} We prove the first convergence guarantees for \algname{SPEG}\eqref{SPEG_UpdateRule} in the quasi-strongly monotone~\eqref{eq: Strong Monotonicity} and weak MVI~\eqref{eq: weak MVI} cases \emph{without using the bounded variance assumption}. We achieve that by using the proposed \eqref{eq: ER Condition} condition. In particular, for the class of quasi-strongly monotone VIPs, we show a linear
convergence rate to a neighborhood of $x^*$ when constant step-sizes are used. We also provide theoretically
motivated step-size switching rules that guarantee exact convergence of \algname{SPEG} to $x^*$. In the weak MVI case, we prove the convergence of \algname{SPEG} for $\rho < \nicefrac{1}{2L}$, improving the existing restrictions on $\rho$. We compare our results with the existing literature in Table~\ref{tab:comparison_of_rates}.

\item \textbf{Arbitrary Sampling.} Via a
stochastic reformulation of the variational inequality problem~\eqref{eq: Variational Inequality Definition} we explain how our convergence guarantees of \algname{SPEG} hold under the arbitrary sampling paradigm. This allows us to cover a wide range of samplings for \algname{SPEG} that were never considered in the literature before, including mini-batching, uniform sampling, and importance sampling as special cases. In this sense, our analysis of \algname{SPEG} is unified for different sampling strategies. Finally, to highlight the tightness of our analysis, we show that the best-known convergence guarantees of deterministic \algname{PEG} for strongly monotone and weak MVI can be obtained as special cases of our main theorems.
\end{itemize}

\section{Stochastic Reformulation of VIPs \& Single-Call Extragradient Methods}

In this work, we provide a theoretical analysis of single-call stochastic extragradient methods that allows us to obtain convergence guarantees of any minibatch and reasonable sampling selection.  We achieve that by using the recently proposed ``stochastic reformulation”
of the variational inequality problem \eqref{eq: Variational Inequality Definition} from \cite{loizou2021stochastic}. That is, to allow for any form of minibatching, we use the \emph{arbitrary sampling} notation 

\begin{equation}
\label{gEstimator}
g(x) = F_v(x) \eqdef \frac{1}{n} \sum _{i=1}^n v_i F_i(x),
\end{equation}
where $v\in\R^n_+$ is a random \emph{sampling vector} drawn from a user-defined distribution $\cD$ such that $\Exp_{\cD}[v_i]  = 1, \,\mbox{for }i=1,\ldots, n$. In this setting, the original problem \eqref{eq: Variational Inequality Definition} can be equivalently written as,
\begin{equation}
\label{Reformulation}
\small{\text{Find } x^* \in \R^d: \E_\cD \left[F_v(x^*)\eqdef\frac{1}{n}\sum_{i=1}^n v_i F_i(x^*) \right]=0,}
\end{equation}
where the equivalence trivially holds since $\Exp_{\cD}[F_v(x)] =\frac{1}{n} \sum _{i=1}^n \Exp_{\cD}[v_i] F_i(x) = F(x).$

In this work, we consider \emph{Stochastic Past Extragradient Method} (\algname{SPEG}) applied to~\eqref{Reformulation}:
\begin{equation}\label{SPEG_UpdateRule}
    \begin{split}
    \hat{x}_k & = x_k - \gamma_k F_{v_{k - 1}}(\hat{x}_{k - 1}) \\ 
    x_{k + 1} & = x_k - \om_k F_{v_k}(\hat{x}_k) 
    \end{split}
    \end{equation} 
where $\hat x_{-1} = x_{0}$ and $v^k \sim \cD$ is sampled i.i.d at each iteration and $\gamma_k >0$ and $\om_k >0$ are the extrapolation step-size and update step-size respectively. We note that in our convergence analysis, we allow selecting \emph{any} distribution $\cD$ that satisfies $\Exp_{\cD}[v_i]  = 1$ $\forall i$. This means that for a different selection of $\cD$, \eqref{SPEG_UpdateRule} yields different interpretations of \algname{SPEG} for solving the original problem \eqref{eq: Variational Inequality Definition}.

One example of distribution $\cD$ is $\tau$--minibatch sampling, which is defined as follows.
\begin{definition}[$\tau$-Minibatch sampling]\label{def:minibatch}
Let $\tau \in [n]$. We say that $v \in \R^n$ is a $\tau$--minibatch sampling if
for every subset $S \in [n]$ with $|S| =\tau$, we have that $\Prob{v=\frac{n}{\tau}\sum_{i \in S} e_i} \eqdef  \frac{1}{\binom{n}{\tau}} = \frac{\tau!(n-\tau)!}{n!}.$
\end{definition}

By using a double counting argument, one can show that if $v$ is a $\tau$--minibatch sampling, it is also a valid sampling vector ($\Exp_{\cD}[v_i]  = 1$)~\citep{gower2019sgd}. We highlight that our analysis holds for every form of minibatching and for several choices of sampling vectors $v$. Later in Section~\ref{section: Arbitrary Sampling}, we provide more details related to non-uniform sampling. In addition, by Definition~\ref{def:minibatch}, it is clear that if $\tau=n$, then $v_i=1$ for all $i \in [n]$. Later in Section~\ref{Theory}, we prove how our analysis captures the deterministic Past Extragradient Method as a special case. 

In \cite{loizou2021stochastic}, an analysis of stochastic gradient descent-ascent ($x_{k+1} = x_k - \om_k F_{v_k}(x_k)$) under the arbitrary sampling paradigm was proposed for solving star-co-coercive VIPs. Later \cite{gorbunov2022stochastic}, extended this approach and provided general convergence guarantees for stochastic extragradient method (\algname{SEG}) (a stochastic variant of the popular extragradient method \citep{korpelevich1976extragradient, juditsky2011solving}) for solving quasi-strongly monotone and monotone VIPs. Despite its popularity, \algname{SEG} requires two oracle calls per iteration which makes it prohibitively expensive in many large-scale applications and not easily applicable to the online learning problems \citep{golowich2020tight}. This motivates us to explore in detail the convergence guarantees of single-call variants of extragradient methods (extragradient methods that require only a single oracle call per iteration). 

\paragraph{On Single-Call Extragradient Methods.} The seminal work of \cite{popov1980modification} is the first paper that proposes the deterministic Past Extagradient method. In the stochastic setting, \cite{hsieh2019convergence} provides an analysis of several stochastic single-call extragradient methods for solving strongly monotone VIPs. In \cite{hsieh2019convergence}, it was also shown that in the unconstrained setting, the update rules of Past Extragradient and Optimistic Gradient are exactly equivalent (see also Proposition \ref{proposition: equivalence of SPEG and SOG} in appendix). Through this connection, and via our stochastic reformulation~\eqref{Reformulation} our theoretical results hold also for the \emph{Stochastic Optimistic Gradient Method} (\algname{SOG}): $ x_{k+1} = x_k - \om_k F_{v_k}(x_k) - \gamma_k (F_{v_k}(x_k) - F_{v_{k-1}}(x_{k-1}))$.

\cite{bohm2022solving} provides the convergence guarantees of \algname{SOG} for weak MVI. To the best of our knowledge, our work is the first that provides convergence guarantees for \algname{SOG} under the arbitrary sampling paradigm (captures sampling beyond uniform sampling) and also without using the bounded variance assumption. 

\section{Expected Residual}

In our theoretical results, we rely on Expected Residual (ER) condition. In this section, we define ER and explain how it is connected with similar conditions used in optimization literature. We further provide sufficient conditions for ER to hold and prove how it can be used to obtain a strictly weaker upper bound of $\E \|g(x)\|^2$ than previously used growth conditions, expected co-coercivity, or bounded variance assumptions.
\begin{assumption}\label{as:expected_residual}
We say the Expected Residual (ER) condition holds if there is a parameter $\delta>0$ such that for an unbiased estimator $g(x)$ of the operator $F$, we have
\begin{align*}
\label{eq: ER Condition}
    \E \left[ \| (g(x) - g(x^*)) - (F(x) - F(x^*))\|^2 \right] \leq \frac{\delta}{2} \|x - x^*\|^2. \tag{ER}
\end{align*}
\end{assumption}

The \ref{eq: ER Condition} condition bounds how far the stochastic estimator $g(x)= F_v(x)$ \eqref{gEstimator} used in \algname{SPEG} is from the true operator $F(x)$. \ref{eq: ER Condition} depends on both the properties of the operator $F(x)$ and of the selection of sampling (via $g(x)$). 
Conditions similar to \ref{eq: ER Condition} appeared before in optimization literature but they have never been used in operator theory and the analysis of \algname{SPEG}. In particular, \cite{gower2021sgd} used a similar condition for analyzing \algname{SGD} in stochastic optimization problems but with the right-hand side of \ref{eq: ER Condition} to be the function suboptimality $f(x)-f(x^*)$ (such concept is not available in VIPs). In \cite{szlendak2021permutation} and \cite{gorbunov2022variance}, similar conditions appear under the name ``Hessian variance'' assumption for distributed minimization problems. In the context of distributed VIPs, a similar but stronger condition to \ref{eq: ER Condition} is used by \cite{beznosikov2023compression}. 

\textbf{Bound on Operator Noise.}
A common approach for proving the convergence of stochastic algorithms for solving the VIPs is assuming uniform boundedness of the stochastic operator or uniform boundedness of the variance. However, as we explain below, these assumptions either do not hold or are true only for a restrictive set of problems. In our work, we do not assume such bounds. Instead,
we use the following direct consequence of \ref{eq: ER Condition}.

\begin{lemma}\label{Lemma: variance bound}
Let $\sigma_*^2 \eqdef \E\|g(x^*)\|^2 <\infty$ (operator noise at the optimum is finite). If \ref{eq: ER Condition} holds, then
\begin{equation}\label{eq: variance bound}
    \begin{split}
        \E \|g(x)\|^2 \leq \delta \|x - x^*\|^2 + \|F(x)\|^2 + 2\sigma_*^2.
    \end{split}
\end{equation}
\end{lemma}
\paragraph{Sufficient Conditions for \ref{eq: ER Condition}.}
Let us now provide sufficient conditions which guarantee that the \ref{eq: ER Condition} condition holds and give a closed-form expression for the expected residual parameter $\delta$ and $\sigma_*^2 = \E\|g(x^*)\|^2$ for the case of $\tau$-minibatch sampling (Def.~\ref{def:minibatch}).
\begin{proposition}\label{Prop_SufficientCondition}
Let $F_i$ of problem~\eqref{eq: Variational Inequality Definition} be $L_i$-Lipschitz operators, then \ref{eq: ER Condition} holds.  If, in addition, 
vector $v \in \R^n$ is a $\tau$--minibatch sampling (Def.~\ref{def:minibatch}) then: $\delta = \frac{2}{n \tau} \frac{n- \tau}{n -1} \sum_{i = 1}^n L_i^2, \text{and } \sigma_*^2 = \frac{1}{n \tau} \frac{n- \tau}{n -1} \sum_{i = 1}^n \|F_i(x^*)\|^2.$
\end{proposition}
Similar results to Prop.~\ref{Prop_SufficientCondition} but under different sufficient conditions have been obtained for $\tau$--minibatch sampling under expected smoothness and a variant of expected residual for solving minimization problems in \cite{gower2019sgd} and \cite{gower2021sgd} respectively. In \cite{loizou2021stochastic}, a similar proposition was derived but for the much more restrictive class of co-coercive operators.

\paragraph{Connection to Other Assumptions.}  In the proofs of our convergence results, we use the bound \eqref{eq: variance bound}, which, as we explained above, is a direct consequence of \ref{eq: ER Condition}. In this paragraph, we place this bound in a hierarchy of common assumptions used for the analysis of stochastic algorithms for solving VIPs. 
In the literature on stochastic algorithms for solving the VIPs and min-max optimization problems, previous works assume either bounded operator ($\E\|g(x)\|^2 \leq c$) \citep{abernethy2021last, nemirovski2009robust}, bounded variance ($\E \|g(x) - F(x)\|^2 \leq c$) \citep{lin2020gradient, tran2020hybrid, juditsky2011solving} (in Appendix \ref{sec:BoundedVarianceCounterExample} we provide a simple example where bounded variance assumption does not hold) or growth condition ($\E \|g(x)\|^2 \leq c_1 \|F(x)\|^2 + c_2$) \citep{lin2020finite}.  In all of these conditions, the parameters $c$, $c_1$, and $c_2$ are usually constants that do not have a closed-form expression. The closer works to our results are \cite{loizou2021stochastic,beznosikov2022stochastic} which assumes existence of $l_{F}>0$ such that the expected co-coercivity condition ($\E \|g(x) - g(x^*)\|^2 \leq l_{F} \la F(x), x- x^*\ra$) holds. Their convergence guarantees provide an efficient analysis for several variants of \algname{SGDA} for solving co-coercive VIPs. In the proposition below, we prove how these conditions are related to the bound \eqref{eq: variance bound} obtained using \ref{eq: ER Condition}.

\begin{proposition}\label{Proposition connecting assumptions} Suppose $F$ is a $L$-Lipschitz operator. Then we have the following hierarchy of assumptions:
\begin{center}
\begin{tikzcd}[column sep=.4em, row sep = .4em]
\boxed{\text{Bounded Operator}} \arrow[r] & \boxed{\text{Bounded Variance}} \arrow[r] & \boxed{\text{Growth Condition}} \arrow[r] & \boxed{\eqref{eq: variance bound}}\\
& \boxed{\text{$F_i$ are $L_i$-Lipschitz}} \arrow[r] & \boxed{\eqref{eq: ER Condition}}  \arrow[ur] & \\
& \boxed{\text{Expected Cocoercivity}}\arrow[ur] & & &
\end{tikzcd}
\end{center}
\end{proposition}
\cqmo{I would still like to see some evidence in real world that the supposed condition can hold while other assumptions like variance bounded condition does not. A simple convex problem and a simple real data shall be good enough for me.}
Let us also mention that \cite{hsieh2020explore} provided convergence guarantee of double-oracle stochastic extragradient (\algname{SEG}) method under the variance control condition $\E \|g(x) - F(x)\|^2 \leq (a \|x - x^*\| + b)^2$ where $a, b \geq 0$. In their work, they focus on solving VIPs satisfying the error-bound condition, and they did not provide closed-form expressions of parameters $a$ and $b$. Although the analysis of \cite{hsieh2020explore} can be conducted with $a > 0$, the authors only provide rates for the case $a = 0$. The main difference between their results (for \algname{SEG}) and our results (for \algname{SPEG}) is that our bound \eqref{eq: variance bound} is not really an assumption, but it holds for free when $F_i$ are $L_i$-Lipschitz. In addition, the values of parameters $\delta$ and $\sigma_*^2$ in \eqref{eq: variance bound} could have different values based on the sampling used in the update rule of \algname{SPEG}.
\section{Convergence Analysis}\label{Theory}
In this section, we present and discuss the main convergence results of this work. In the first part, we focus on the ones derived for $\mu$-quasi strongly monotone problems \eqref{eq: Strong Monotonicity} (both for constant and decreasing step-sizes), and in the second part on the Weak Minty VIP \eqref{eq: weak MVI}. 
\subsection{Quasi-Strongly Monotone Problems}
\paragraph{Constant Step-size:}We start with the case of $\mu$-quasi strongly monotone problems and consider the convergence of \algname{SPEG} with constant step-size.

\begin{theorem}\label{Theorem: constant stepsize theorem}
Let $F$ be $L$-Lipschitz, $\mu$-quasi strongly monotone, and let \ref{eq: ER Condition} hold.  Choose  step-sizes $\gamma_k = \omega_k = \omega$ such that 
\begin{equation}\label{eq:constant_stepsize}
0 < \om \leq \min \left\{ \frac{\mu}{18 \delta}, \frac{1}{4L}\right\}    
\end{equation}
for all $k$. Then the iterates produced by \algname{SPEG}, given by \eqref{SPEG_UpdateRule} satisfy
\begin{equation}
    R_{k}^2 \leq \left(1 - \frac{\omega\mu}{2}\right)^{k} R_0^2 + \frac{24\om \sigma_*^2}{\mu}, \label{eq:SPEG_const_steps_neighborhood}
\end{equation}
where $R_{k}^2 \coloneqq \E \left[\|x_{k} - x^*\|^2 + \|x_{k} - \hat{x}_{k-1}\|^2 \right]$. Hence, given any $\varepsilon > 0$, and choosing $\om = \min \left\{\frac{\mu}{18 \delta}, \frac{1}{4L}, \frac{\varepsilon \mu}{48 \sigma_*^2} \right\}$,
\algname{SPEG} achieves $\E \|x_K - x^*\|^2 \leq \varepsilon$ after $K \geq \max \bigg\{\frac{8L}{\mu}, \frac{36 \delta}{\mu^2}, \frac{96 \sigma_*^2}{\varepsilon \mu^2} \bigg\} \log \bigg( \frac{2 R_0^2}{\varepsilon} \bigg)$
iterations.
\end{theorem}
To the best of our knowledge, the above theorem is the first result on the convergence of \algname{SPEG} that does not rely on the bounded variance assumption. Theorem~\ref{Theorem: constant stepsize theorem} recovers the same rate of convergence with the Independent-Samples \algname{SEG} (\algname{I-SEG}) under assumption \eqref{eq: variance bound} \citep{gorbunov2022stochastic}, although \cite{gorbunov2022stochastic} simply assume \eqref{eq: variance bound}, while we show that it follows from Assumption~\ref{as:expected_residual} holding whenever all summands $F_i$ are Lipschitz. However, in the case when all $F_i$ are $\mu$-quasi strongly monotone and $L_i$-Lipschitz (on average), one can use Same-Sample \algname{SEG} (\algname{S-SEG}). The existing results for \algname{S-SEG} have better exponentially decaying term \citep{mishchenko2020revisiting, gorbunov2022stochastic} then Theorem~\ref{Theorem: constant stepsize theorem}, e.g., in the case when $L_i = L $ for all $i\in [n]$, we have $\delta = \cO(L^2)$ meaning that the exponentially decaying term in \eqref{eq:SPEG_const_steps_neighborhood} is $\cO(R_0^2\exp(-\nicefrac{\mu^2k}{L^2})))$, while \algname{S-SEG} has much better exponentially decaying term $\cO(R_0^2\exp(-\nicefrac{\mu k}{L})))$. 

Such a discrepancy can be partially explained by the following fact: \algname{S-SEG} can be seen as one step of deterministic Extragradient for stochastic operator $F_{v_k}$ allowing to use one-iteration analysis of Extragradient without controlling the variance. In contrast, there is no version of \algname{SPEG} that uses the same sample for extrapolation and update steps. This forces to use different samples for these steps and this is a key reason why \algname{SPEG} cannot be seen as one iteration of deterministic Past-Extragradient for some operator. Due to this, we need to rely on some bound on the variance to handle the stochasticity in the updates; see also \citep[Appendix F.1]{gorbunov2022stochastic}. Therefore, in our analysis, we use Assumption~\ref{as:expected_residual}, implying \eqref{eq: variance bound}. Nevertheless, it is still an open question whether it is possible to improve the rate of \algname{SPEG} in the case of $\mu$-quasi strongly monotone and Lipschitz operators $F_i$.

To highlight the generality of Theorem~\ref{Theorem: constant stepsize theorem}, we note that for the
deterministic \algname{PEG}, $\delta = 0$ and $\sigma_*^2 = 0$ (by selecting $\tau=n$ in the definition~\ref{def:minibatch} of minibatch sampling). In this case, Theorem~\ref{Theorem: constant stepsize theorem} recovers the well-known result (up to $\nicefrac{1}{2}$ factor in the rate) for deterministic \algname{PEG} proposed in \cite{gidel2018variational} as shown in the following corollary.
\begin{corollary}
\label{dawna}
    Let the assumptions of Theorem~\ref{Theorem: constant stepsize theorem} hold and a deterministic version of \algname{SPEG} is considered, i.e., $\delta = 0$, $\sigma_*^2 = 0$. Then, Theorem~\ref{Theorem: constant stepsize theorem} implies that for all $k\geq 0$ the iterates produced by \algname{SPEG} with step-sizes $\gamma_k = \omega_k = \omega$ such that $ 0 < \om \leq \frac{1}{4L} $ satisfy $R_{k}^2 \leq \left(1 - \frac{\omega\mu}{2}\right)^{k} R_0^2$,
where $R_{k}^2 \coloneqq \|x_{k} - x^*\|^2 + \|x_{k} - \hat{x}_{k-1}\|^2$.
\end{corollary}
\paragraph{Decreasing Step-size:}In this section, we consider two different decreasing step-sizes policies for \algname{SPEG} applied to solve quasi-strongly monotone problems.
\begin{theorem}\label{SPEG switching rule}
Let $F$ be $L$-Lipschitz, $\mu$-quasi strongly monotone, and Assumption~\ref{as:expected_residual} hold. Let
\begin{equation}
    \gamma_k = \om_k \coloneqq 
\begin{cases}
\Bar{\om}, &\text{if } k \leq k^*, \\
\frac{2k+1}{(k+1)^2}\frac{2}{\mu}, &\text{if } k > k^*,
\end{cases}\label{eq:stepsize_switching_1}
\end{equation}
where $\Bar{\om} \coloneqq \min \left \{\nicefrac{1}{(4L)},\nicefrac{\mu}{(18 \delta)} \right\}$ and $k^* = \lceil \nicefrac{4}{(\mu \Bar{\om})} \rceil$. Then for all $K \geq k^*$ the iterates produced by \algname{SPEG} with step-sizes \eqref{eq:stepsize_switching_1} satisfy 
\begin{equation}
    R_{K}^2  \leq \left(\frac{k^*}{K}\right)^2 \frac{R_0^2}{\exp(2)} + \frac{192 \sigma_*^2}{\mu^2 K}, \label{eq:SPEG_convergence_decr_steps_1}
\end{equation}
where $R_{K}^2 \coloneqq \E \left[\|x_{K} - x^*\|^2 + \|x_{K} - \hat{x}_{K-1}\|^2 \right]$.
\end{theorem}
\algname{SPEG} with step-size policy\footnote{Similar step-size policy is used for \algname{SGD} \citep{gower2019sgd} and \algname{SGDA} \citep{loizou2021stochastic}.} \eqref{eq:stepsize_switching_1} has two stages of convergence: during first $k^*$ iterations it uses constant step-size to reach some neighborhood of the solution and then the method switches to the decreasing $\cO(\nicefrac{1}{k})$ step-size allowing to reduce the size of the neighborhood.

For the case of strongly monotone problems (a special case of our quasi-strongly monotone setting) \cite{hsieh2019convergence} also analyze \algname{SPEG} with decreasing $\cO(\nicefrac{1}{k})$ step-size\footnote{We point out the proof by \cite{hsieh2019convergence} can be generalized to the case of constant step-size, though the authors do not consider this step-size schedule explicitly.} under bounded variance assumption, i.e., when \eqref{eq: variance bound} holds with $\delta = 0$ and some $\sigma_*^2 \geq 0$, which is equivalent to the uniformly bounded variance assumption. In particular, Theorem 5 \cite{hsieh2019convergence} states
$\E\left[\|x_{K} - x^*\|^2\right] \leq \frac{C\sigma_*^2}{\mu^2 K} + o\left(\frac{1}{K}\right)$ where $C$ is some numerical constant. If the problem is strongly monotone, the result of \cite{hsieh2019convergence} is closely related to what is obtained in Theorem~\ref{SPEG switching rule}: the main difference in the upper-bound is that we provide an explicit form of $o\left(\nicefrac{1}{K}\right)$ term. Moreover, in contrast to the result from \citep{hsieh2019convergence}, Theorem~\ref{SPEG switching rule} holds even when $\delta > 0$ in \eqref{eq: variance bound}, which covers a larger class of problems. 

Following \cite{stich2019unified, gorbunov2022stochastic, beznosikov2022stochastic}, we also consider another decreasing step-size policy.
\begin{theorem} \label{Theorem: Total number of iteration knowledge}
Let $F$ be $L$-Lipschitz, $\mu$-quasi strongly monotone, and Assumption~\ref{as:expected_residual} hold. Let $\Bar{\om} \coloneqq \min \left\{ \nicefrac{1}{(4L)}, \nicefrac{\mu}{(18\delta)} \right\}$. If for $K \geq 0$ step-sizes $\{\gamma_k\}_{k \geq 0}$, $\{\om_k\}_{k \geq 0}$ satisfy $\gamma_k = \om_k$ and
\begin{equation}\label{eq:stepsize_switching_2}
    \om_k \coloneqq 
\begin{cases}
\Bar{\om}, &\text{if $K \leq \frac{2}{\mu \Bar{\om}}$}, \\
\Bar{\om}, &\text{if $K > \frac{2}{\mu \Bar{\om}}$ and $k \leq k_0$,}\\
\frac{2}{\frac{2}{\Bar{\om}} + \frac{\mu}{2}(k - k_0)}, &\text{if $K > \frac{2}{\mu \Bar{\om}}$ and $k > k_0$}
\end{cases}
\end{equation}
where $k_0 = \lceil \nicefrac{K}{2} \rceil$, then the iterates produced by \algname{SPEG} with the step-sizes defined above satisfy
\begin{equation}\label{rate for Total number of iteration knowledge}
\begin{split}
R_K^2\leq \frac{64R_0^2}{\Bar{\om} \mu} \exp \left\{-\min \left\{ \frac{\mu}{16L}, \frac{\mu^2}{72\delta} \right\}K \right\} + \frac{ 1728\sigma_*^2}{\mu^2 K},
\end{split}    
\end{equation}
where $R_{K}^2 \coloneqq \E \left[\|x_{K} - x^*\|^2 + \|x_{K} - \hat{x}_{K-1}\|^2 \right]$.
\end{theorem}
In contrast to \eqref{eq:SPEG_convergence_decr_steps_1}, the rate from \eqref{rate for Total number of iteration knowledge} has much better (exponentially decaying) $o\left(\nicefrac{1}{K}\right)$ term. When $\sigma_*^2$ is large and one needs to achieve very good accuracy of the solution, this difference is negligible, since the dominating $\cO(\nicefrac{1}{K})$ term is the same for both bounds (up to numerical factors). However, when $\sigma_*^2$ is small enough, e.g., the model is close to over-parameterized, or it is sufficient to achieve low accuracy of the solution, the dominating term in \eqref{rate for Total number of iteration knowledge} is typically much smaller than the one from \eqref{eq:SPEG_convergence_decr_steps_1}. Finally, it is worth mentioning, that the improvement of $o\left(\nicefrac{1}{K}\right)$ is not achieved for free: unlike the policy from \eqref{eq:stepsize_switching_1}, step-size rule \eqref{eq:stepsize_switching_2} relies on the knowledge of the total number of steps $K$, which can be inconvenient for the practical use in some cases.
\subsection{Weak Minty Variational Inequality Problems}
In this subsection we will discuss convergence of Stochastic Past Extragradient method for Minty Variational Inequality problem.
To solve the Minty variational inequality problem we use different step-sizes for \algname{SPEG} iterates (\ref{SPEG_UpdateRule}).  
\begin{theorem}\label{cor:weak_MVI_convergence}
    Let $F$ be $L$-Lipschitz and satisfy Weak Minty condition with parameter $\rho < \nicefrac{1}{(2L)}$. Let Assumption~\ref{as:expected_residual} hold. Assume that $\gamma_k = \gamma$, $\omega_k = \omega$ such that $\max\left\{2\rho, \frac{1}{2L}\right\} < \gamma < \frac{1}{L}$ and $0 < \omega < \min\left\{\gamma - 2\rho, \frac{1}{4L} - \frac{\gamma}{4}\right\}.$
    Then, for all $K \geq 2$ the iterates produced by mini-batched \algname{SPEG} with batch-size 
    \begin{eqnarray}
        \tau \geq \max\Bigg\{1, \frac{32\delta}{(1-L\gamma)L^3\omega}, \frac{48\omega\gamma \delta(K-1)}{(1 - L\gamma)^2},\frac{2\omega\gamma\sigma_*^2(K-1)}{(1-L\gamma)\|x_0 - x^*\|^2}\Bigg\} \label{eq:SPEG_weak_MVI_batchsize}
    \end{eqnarray}
    satisfy $\min\limits_{0\leq k \leq K-1}\E\left[\|F(\hat x_k)\|^2\right] \leq \frac{C\|x_0 - x^*\|^2}{K-1},$
where $C = \frac{48}{\omega\gamma (1 - L(\gamma + 4\omega))}$.
\end{theorem}
The above result establishes $\cO(\nicefrac{1}{K})$ convergence with $\cO(K)$ batchsizes for \algname{SPEG} applied to problems satisfying Weak Minty condition.\footnote{See also Appendix~\ref{AppendixE5} for a discussion related to the oracle complexity of Theorem \ref{cor:weak_MVI_convergence}.} The closest result is obtained by \cite{bohm2022solving}, for the same method under bounded variance assumption, i.e., when $\delta = 0$. In particular, the result of \cite{bohm2022solving} also gives $\cO(\nicefrac{1}{K})$ rate and requires $\cO(K)$ batchsizes at each step. We extend this result to the case of non-zeroth $\delta$ and we also improve the assumption on $\rho$: \cite{bohm2022solving} assumes that $\rho < \nicefrac{3}{8L}$, while Theorem~\ref{cor:weak_MVI_convergence} holds for $\rho < \nicefrac{1}{2L}$. The bound on $\rho$ cannot be improved even in the deterministic case \citep{gorbunov2022convergence}. Moreover, it is worth mentioning that the proof of Theorem~\ref{cor:weak_MVI_convergence} noticeably differs from the one obtained by \cite{bohm2022solving}.

In the case of a deterministic oracle, we recover the best-known result for Optimistic Gradient in the Weak Minty setup \citep{bohm2022solving, gorbunov2022convergence}.

\begin{corollary}
    Let the assumptions of Theorem~\ref{cor:weak_MVI_convergence} hold and deterministic version of \algname{SPEG} is considered, i.e., $\delta = 0$, $\sigma_*^2 = 0$. Then, Theorem~\ref{cor:weak_MVI_convergence} implies that for all $k\geq 0$ the iterates produced by \algname{SPEG} with step-sizes $\max\left\{2\rho, \frac{1}{2L}\right\} < \gamma < \frac{1}{L}$ and $0 < \omega < \min\left\{\gamma - 2\rho, \frac{1}{4L} - \frac{\gamma}{4}\right\}$
    satisfy $\min\limits_{0\leq k \leq K-1}\|F(\hat x_k)\|^2 \leq \frac{C\|x_0 - x^*\|^2}{K-1},$
    where $C = \frac{48}{\omega\gamma (1 - L(\gamma + 4\omega))}$.
\end{corollary}
\section{Beyond Uniform Sampling}\label{section: Arbitrary Sampling}
In this section, we illustrate the generality of our analysis by focusing on the non-uniform sampling. In particular, we focus on \emph{single-element sampling} in which only the singleton sets $\{i\}$ for $i=\{1,\ldots, n\}$ have a non-zero probability of being sampled; that is,
$\Prob{|S|=1} = 1$. We have $\Prob{v = e_i/p_i} = p_i$. \cite{gower2019sgd} proved that if $v$ is a single-element sampling, it is also a valid sampling vector ($\Exp_{\cD}[v_i]  = 1$). With the following proposition, we provide closed-form expressions for the \ref{eq: ER Condition} parameter $\delta$ and $\sigma_*^2 = \E\|g(x^*)\|^2$ for the case of (non-uniform) single-element sampling.

\begin{proposition}\label{Prop_SingleElement}
Let $F_i$ of problem~\eqref{eq: Variational Inequality Definition} be $L_i$-Lipschitz operators.  If, vector $v \in \R^n$ is a single element sampling then $\delta = \frac{2}{n^2} \sum_{i = 1}^n \frac{L_i^2}{p_i}$ and $\sigma_*^2 = \frac{1}{n^2} \sum_{i=1}^n \frac{1}{p_i} \|F_i(x^*)\|^2.$
\end{proposition}
\textbf{Importance Sampling.} In importance sampling we aim to choose the probabilities $p_i$ that optimize the iteration complexity. \cite{gower2019sgd} and \cite{gorbunov2022stochastic} analyze importance sampling for \algname{SGD} and \algname{SEG} respectively. In this work, we provide the first convergence guarantees of \algname{SPEG} with importance sampling. In particular, we optimize the expected residual parameter $\delta$ with respect to $p_i$, which in turn affects the iteration complexity. Note that, by using Cauchy-Schwarz inequality~\eqref{eq: Cauchy Schwarz Inequality}, we have 
$\sum_{i = 1}^n \frac{L_i^2}{p_i} \geq \left(\sum_{i = 1}^n L_i \right)^2$,
and this lower bound can be achieved for $p_i^{\delta} = \nicefrac{L_i}{\sum_{j = 1}^n L_j}$. In case of importance sampling, we will use these probabilities $p_i^{\delta}$ which optimizes $\delta$ and define the corresponding $\delta$ as $\delta_{\text{IS}} \coloneqq \frac{2}{n^2}\left( \sum_{i = 1}^n L_i\right)^2$.
For uniform sampling \big(i.e. $p_i = \frac{1}{n}$\big), the value of the parameter is $\delta_{\text{US}} = \frac{2}{n} \sum_{i = 1}^n L_i^2$. Note that, $\delta_{\text{IS}}$ equals $\delta_{\text{US}}$ when all $L_i$ are equal, however $\delta_{\text{IS}}$ can be much smaller than $\delta_{\text{US}}$ when $L_i$ are very different from each other, e.g., when all $L_i$ are relatively small (close to zero) and one $L_i$ is large, $\delta_{\text{IS}}$ is almost $n$ times smaller than $\delta_{\text{US}}$. In this latter scenario (when $\delta_{\text{IS}}$ is much smaller than $\delta_{\text{US}}$), importance sampling could be useful and can significantly improve the performance of \algname{SPEG}. For example, note that the exponentially decaying term in \eqref{rate for Total number of iteration knowledge} decreases with $\delta$. Hence, this term will decrease much faster with importance sampling than with uniform sampling.
\vspace{-2mm}
\section{Numerical Experiments}
\label{Numerical Experiments}
To verify our theoretical results, we run several experiments on two classes of problems, i.e., strongly monotone problems (a special case of the quasi-strongly monotone VIPs) and weak MVI problems. The code to reproduce our results can be found at \href{https://github.com/isayantan/Single-Call-Stochastic-Extragradient-Methods}{https://github.com/isayantan/Single-Call-Stochastic-Extragradient-Methods}.
\subsection{Strongly Monotone Problems}\label{sec: experiments on quasi monotone}
Our experiments consider the quadratic strongly-convex strongly-concave min-max problem from \cite{gorbunov2022stochastic}. That is, we implement \algname{SPEG} on quadratic games of the form $\min_{x \in \mathbb{R}^d} \max_{y \in \mathbb{R}^d} \frac{1}{n}\sum_{i = 1}^n f_i(x,y)$ where 
\vspace{-1mm}
\begin{equation}
    f_i(x,y) \coloneqq \frac{1}{2} x^{\intercal}A_i x + x^{\intercal} B_i y - \frac{1}{2}y^{\intercal} C_i y+ a_i^{\intercal}x - c_i^{\intercal}y.
\end{equation}
Here $A_i, B_i, C_i$ are generated such that the quadratic game is strongly monotone and smooth. In all our experiments, we take $n = 100$ and $d = 30$. We generate positive semi-definite matrices $A_i, B_i, C_i$ such that their eigenvalues lie in the interval $[\mu_A, L_A], [\mu_B, L_B]$ and $[\mu_C, L_C]$ respectively. In all our experiments, we consider $L_A = L_B = L_C = 1$ and $\mu_A = \mu_C = 0.1, \mu_B = 0$ unless otherwise mentioned. The vectors $a_i$ and $c_i$ are generated from $\mathcal{N}_d (0,I_d)$. 
Here, the $i$th operator is given by
\vspace{-2mm}
\begin{equation*}\label{operator for quad game}
    \begin{split}
        F_i \begin{pmatrix} x\\
        y \end{pmatrix} = \begin{pmatrix} \nabla_x f_i(x, y) \\ - \nabla_y f_i(x, y) \end{pmatrix} =\begin{pmatrix} A_ix + B_i y + a_i \\
        C_i y - B_i^{\intercal}x + c_i
        \end{pmatrix}
    \end{split}
\end{equation*}
In Figures~\ref{fig: Synthetic Dataset 1}, \ref{fig:hsieh vs our steps}, and \ref{fig:us_vs_is}, 
we plot the relative error on the $y$-axis i.e. $\frac{\|x_k - x^*\|^2}{\|x_0 - x^*\|^2}$. 

\begin{wrapfigure}{r}{0.35\textwidth} 
    \centering
    \includegraphics[width=.35\textwidth]{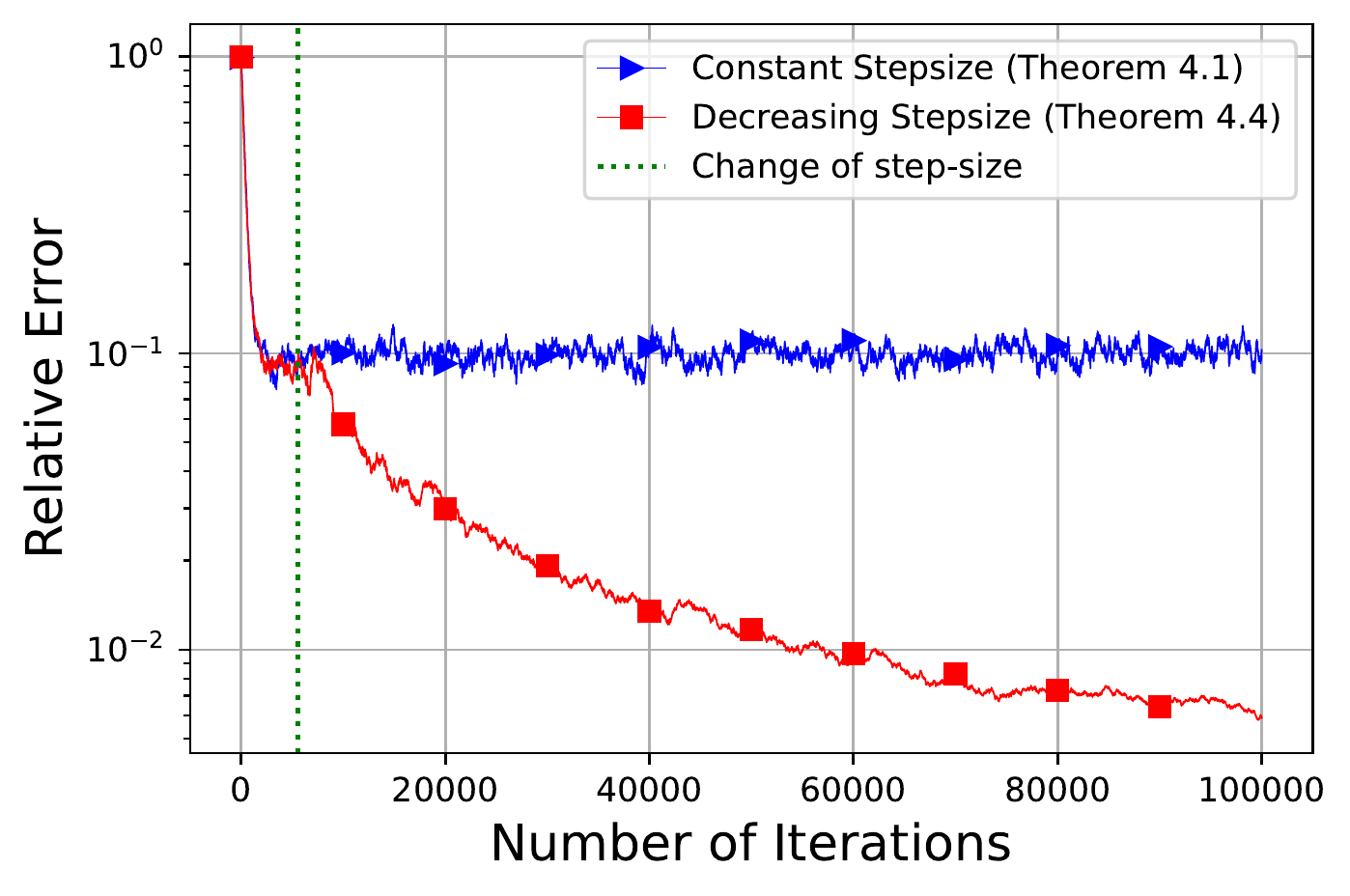}
    \caption{Constant vs Switching}\label{fig: Synthetic Dataset 1}
    \vspace{-2mm}
\end{wrapfigure}
\paragraph{Constant vs Switching Step-size Rule.}
In Fig.~\ref{fig: Synthetic Dataset 1}, we illustrate the step-size switching rule of Theorem \ref{SPEG switching rule}. We place the dotted line to mark when we switch from constant step-size to decreasing step-size. In Fig.~\ref{fig: Synthetic Dataset 1}, the trajectory of switching step-size rule \eqref{eq:stepsize_switching_1} matches that of constant step-size \eqref{eq:constant_stepsize} in the first phase \big(where \algname{SPEG} runs with constant step-size following \eqref{eq:stepsize_switching_1}\big). However, it becomes stagnant when the constant step-size \algname{SPEG} reaches a neighbourhood of optimality. In contrast, the step-size of Theorem \ref{SPEG switching rule} helps the method to converge to better accuracy.

\paragraph{Comparison to \citet{hsieh2019convergence}.}
In this experiment, we compare \algname{SPEG} step-sizes proposed in Theorems \ref{Theorem: constant stepsize theorem} and \ref{SPEG switching rule} with step-sizes from \citep{hsieh2019convergence}.
To implement \algname{SPEG} with the step-sizes from \cite{hsieh2019convergence}, we choose $\gamma$ and $b$ such that $\frac{1}{\mu} < \gamma \leq \frac{b}{4L}$ and set $\om_k = \gamma_k = \frac{\gamma}{k + b}$. For Fig.~\ref{fig: Interpolated Model}, we generate $A_i, B_i, C_i$ as before. First, we sample optimal points $x^*, y^*$ from $\mathcal{N}_d(0, I_d)$ and then generate $a_i, c_i$ such that $F(x^*, y^*) = 0$.   
\begin{equation*}
    \begin{split}
        \begin{pmatrix} a_i \\c_i \end{pmatrix} = \begin{pmatrix} A_i & B_i \\ -B_i^{\intercal} & C_i \end{pmatrix}^{-1} \begin{pmatrix} x^* \\y^* \end{pmatrix}.
    \end{split}
\end{equation*}
In Fig.~\ref{fig: Interpolated Model}, we run the algorithms on interpolated model $\big( F_i(x^*) = 0$ for all $i \in [n]\big)$. Since the model is interpolated, we have $\sigma_*^2 = 0$ in Theorem \ref{Theorem: constant stepsize theorem} and linear convergence to the exact optimum asymptotically. In this setting, as shown in Fig.~\ref{fig: Interpolated Model}, our proposed step-size results in major improvement compared to the decreasing step-size selection analyzed in~\cite{hsieh2019convergence}. In Fig.~\ref{fig: Non-Interpolated Model}, we compare the switching step-size rule with step-size from \cite{hsieh2019convergence}. In Fig.~\ref{fig: Non-Interpolated Model}, we generate $a_i, c_i$ from the normal distribution. In this plot, we manually switch the step-size from constant to decreasing after $305$ steps. We observe that such a semi-empirical rule has comparable performance to the step-size selection of \citet{hsieh2019convergence}. 

\begin{figure}[t]
\centering
\begin{subfigure}[b]{.4\textwidth}
    \centering
    \includegraphics[width=\textwidth]{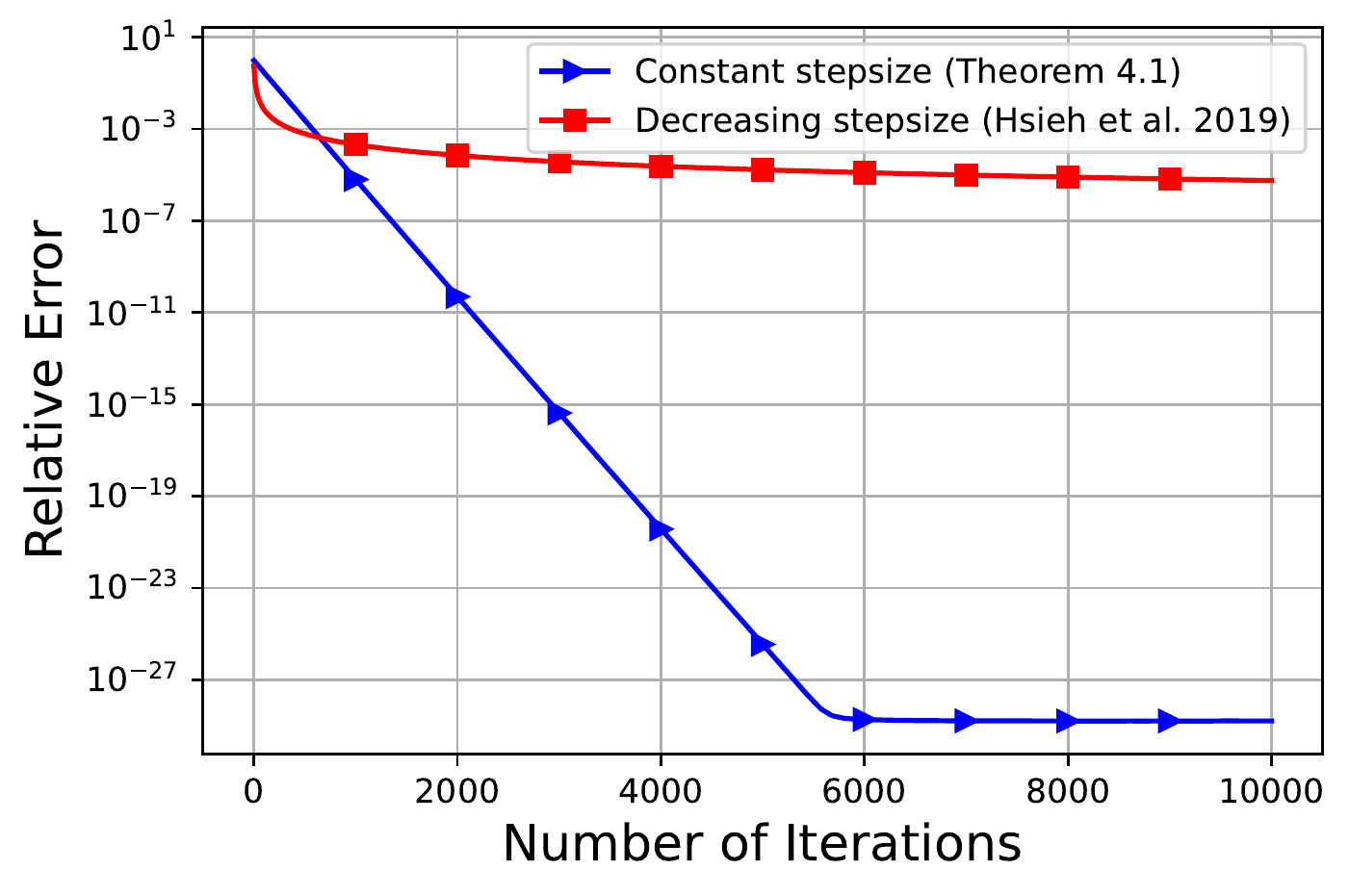}
    \caption{Interpolated Model}\label{fig: Interpolated Model}
\end{subfigure}
\begin{subfigure}[b]{0.4\textwidth}
    \centering
    \includegraphics[width=\textwidth]{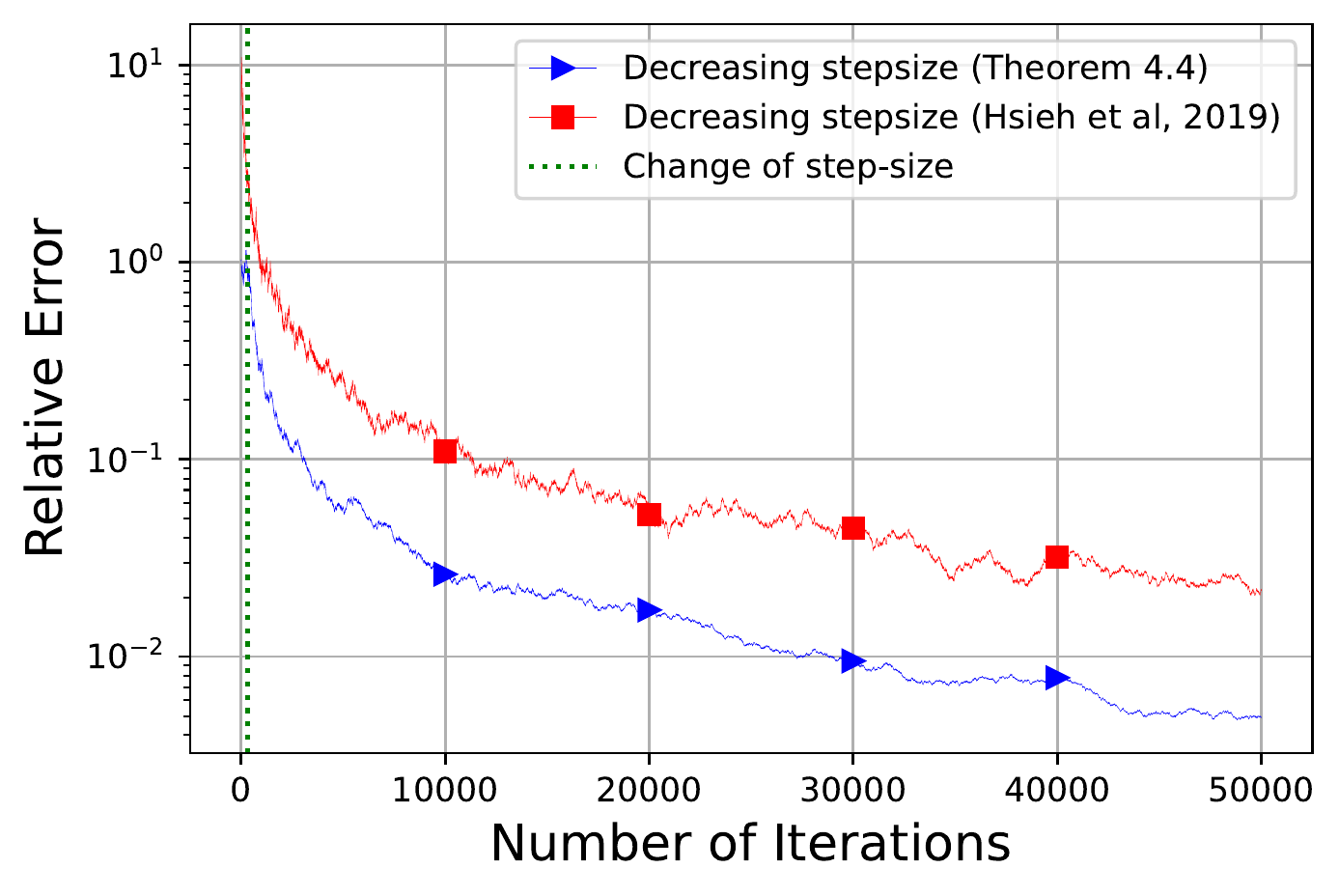}
    \caption{Non-Interpolated Model}\label{fig: Non-Interpolated Model}
\end{subfigure}
\caption{\emph{Comparison of our \algname{SPEG} using our step-size against decreasing step-size of \citet{hsieh2019convergence}. In plot (a), for constant step-size of \algname{SPEG} we use the upper bound of \eqref{eq:constant_stepsize}. In plot (b), we run our switching step-size \algname{SPEG}~\eqref{eq:stepsize_switching_1}.}}
\label{fig:hsieh vs our steps}
\end{figure}

\begin{figure}[t]
\centering
\begin{subfigure}[b]{0.24\linewidth}
    \centering
    \includegraphics[width=\textwidth]{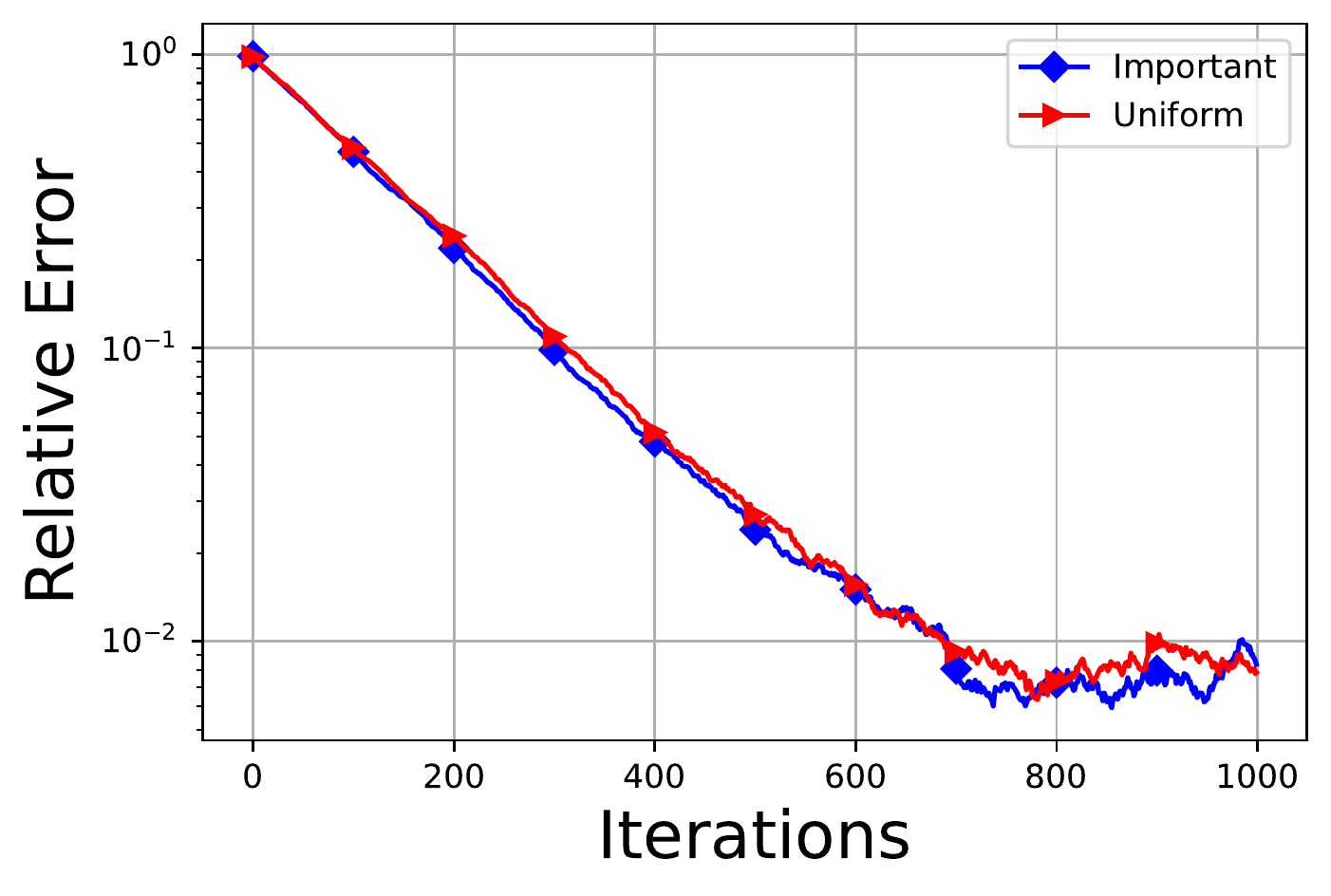}
    \caption{$\Lambda = 2 $}
\end{subfigure}
 \begin{subfigure}[b]{0.24\linewidth}
      \centering
      \includegraphics[width=\textwidth]{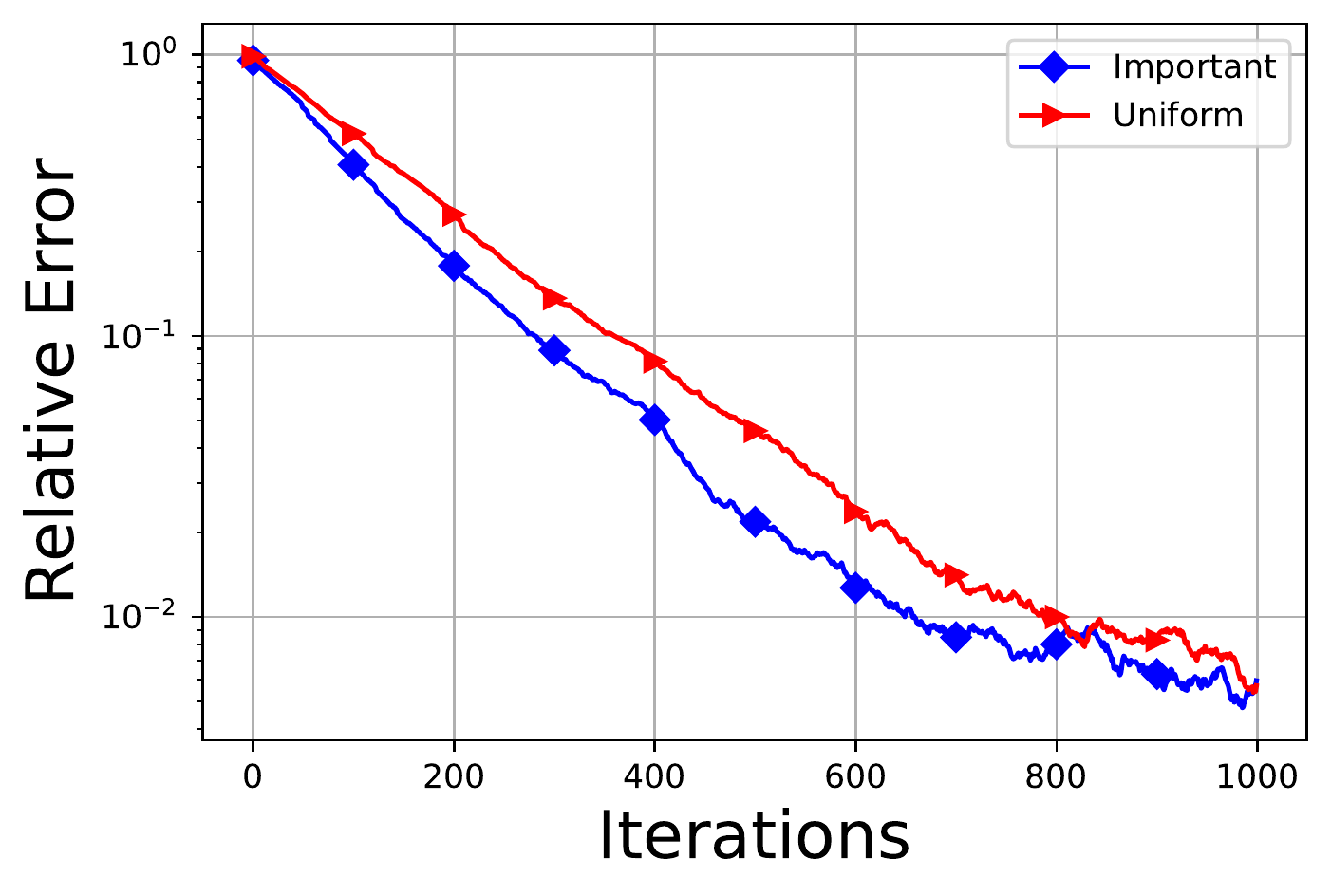}
     \caption{$\Lambda = 5 $}
  \end{subfigure}
\begin{subfigure}[b]{0.24\textwidth}
    \centering
    \includegraphics[width=\textwidth]{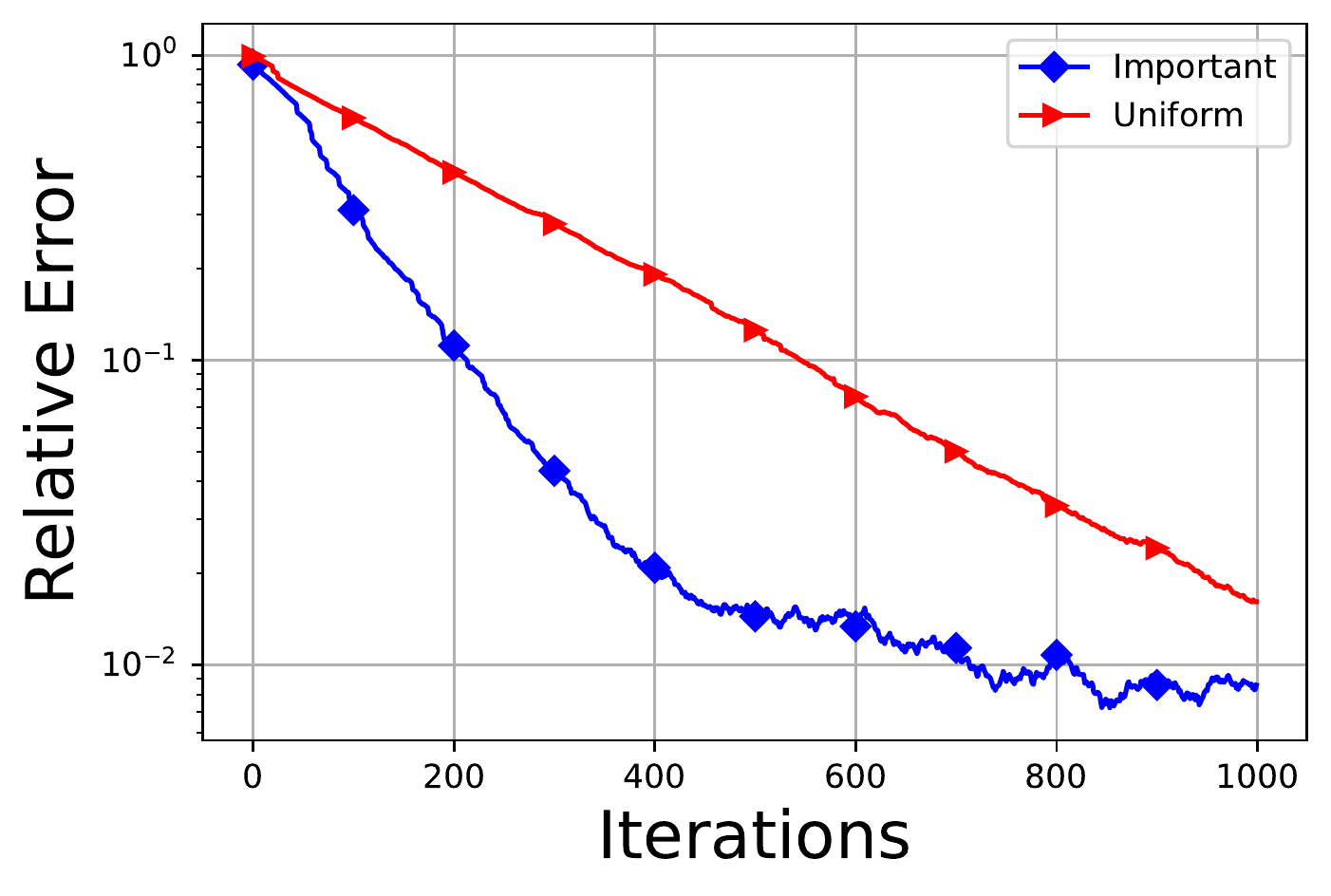}
    \caption{$\Lambda = 10$}
\end{subfigure}
\begin{subfigure}[b]{0.24\textwidth}
    \centering
    \includegraphics[width=\textwidth]{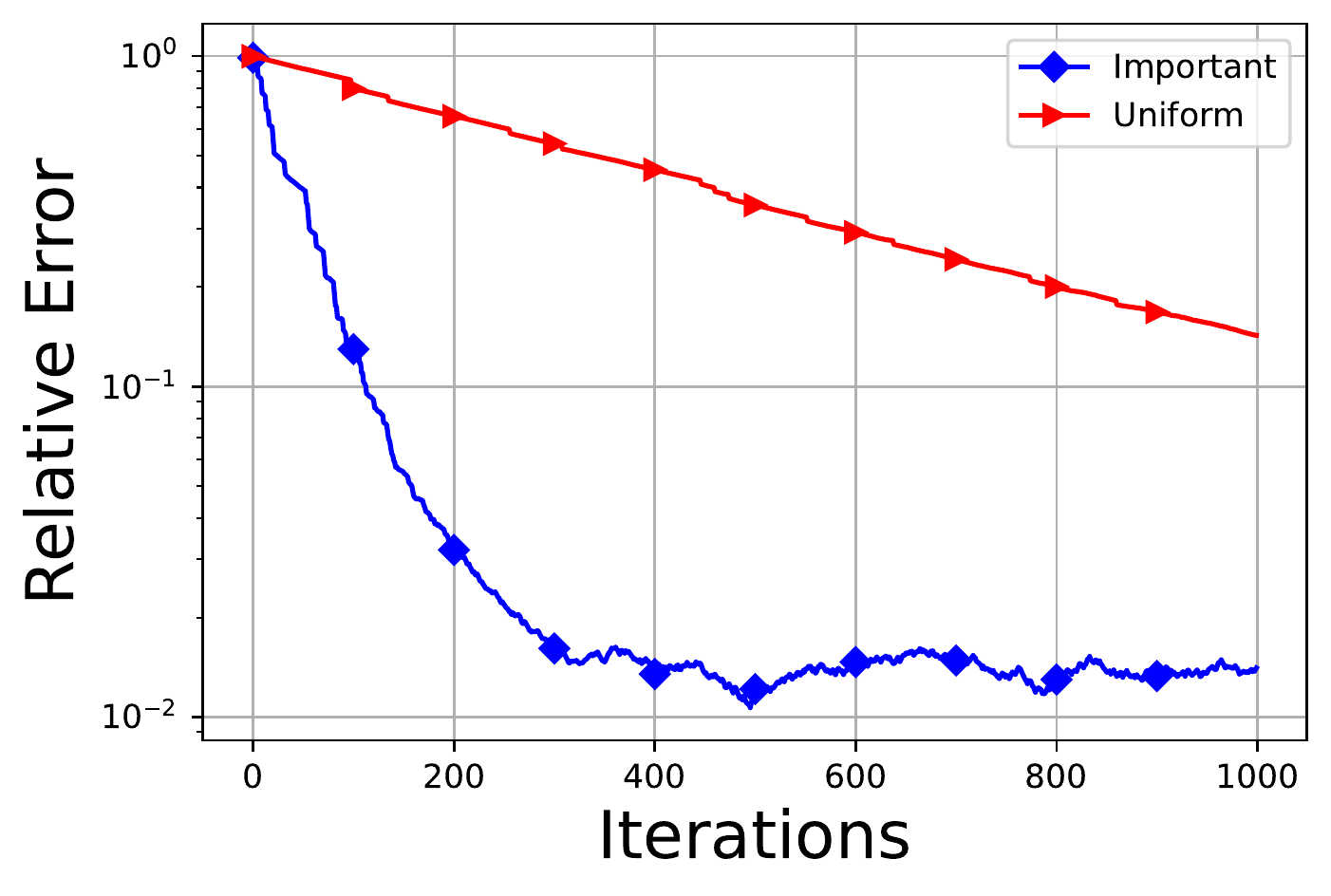}
    \caption{$\Lambda = 20$}
\end{subfigure}  
\caption{\emph{Comparison of \algname{SPEG} with Uniform and Importance Sampling for different $\Lambda \in\{2,5,10,20\}$, where the eigenvalues of matrices $A_1, C_1$ are uniformly generated from the interval $[0.1, \Lambda].$}}
\label{fig:us_vs_is}
\end{figure}

\paragraph{Uniform vs. Importance Sampling.}  In this experiment, we highlight the advantage of using importance sampling over uniform sampling. The eigenvalues of $A_1, C_1$ are uniformly generated from the interval $[0.1, \Lambda]$ while the rest of the matrices are generated as mentioned before. We vary the value of $\Lambda \in \{2,5,10,20\}$ and run and compare \algname{SPEG} with both uniform and importance sampling (see~Fig.~\ref{fig:us_vs_is}).  For importance sampling, we use the probabilities $p_i = \nicefrac{L_i}{\sum_{j = 1}^n L_j}$. In Fig.~\ref{fig:us_vs_is}, it is clear that as the value of $\Lambda$ increases, the trajectories under uniform sampling get worse, while the trajectory under importance sampling remains almost identical. This behaviour aligns well with our discussion in Section~\ref{section: Arbitrary Sampling}.
\subsection{Weak Minty Variational Inequality Problems}\label{sec: Experiment on WMVI}
\begin{wrapfigure}{r}{0.35\textwidth} 
    \centering
    \includegraphics[width=0.35\textwidth]{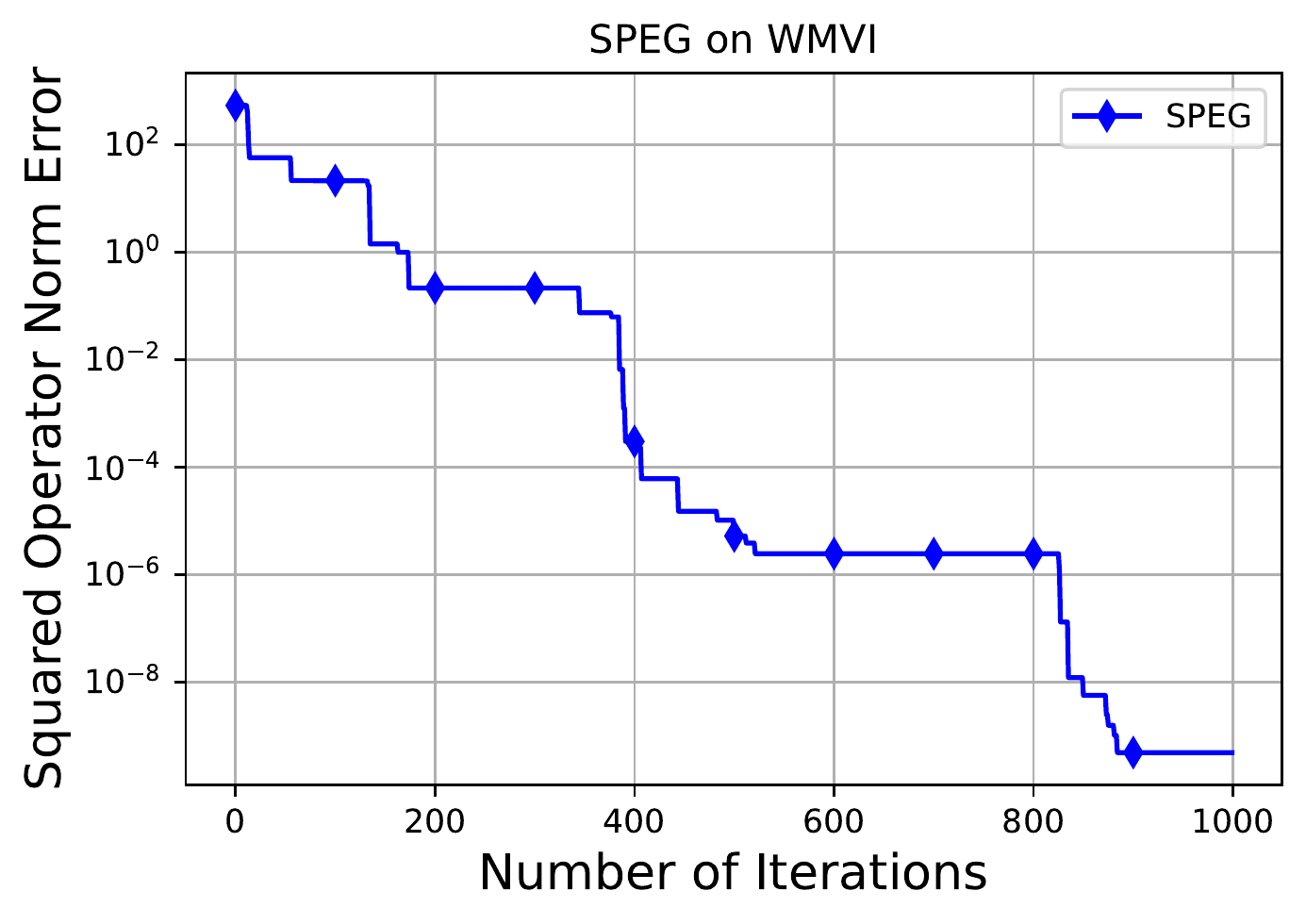}
    \caption{Trajectory of \algname{SPEG} for solving weak MVI.  "Squared Operator Norm Error'' in vertical axis denotes the $\min\limits_{0\leq k \leq K-1}\E\left[\|F(\hat x_k)\|^2\right]$ of Theorem~\ref{cor:weak_MVI_convergence}.}
    \label{fig:performance of SPEG on WMVI_a}
    \vspace{-5mm}
\end{wrapfigure}
This experiment verifies the convergence guarantees of \algname{SPEG} in Theorem \ref{cor:weak_MVI_convergence}. Following the min-max problem mentioned in \cite{bohm2022solving}, we consider the objective function 
\begin{equation}\label{asoxasl}
\begin{split}
       \min_{x \in \mathbb{R}} \max_{y \in \mathbb{R}} \frac{1}{n} \sum_{i = 1}^n \xi_{i} xy + \frac{\zeta_i}{2}(x^2 - y^2). 
\end{split}
\end{equation}
In this experiment, we generate $\xi_i, \zeta_i$ such that $L = 8$ and $\rho = \nicefrac{1}{32}$ for the above min-max problem \citep{bohm2022solving}. We implement \algname{SPEG} with extrapolation step $\gamma_k = 0.08$ and update step $\omega_k = 0.01$ which satisfies the conditions on step-size in Theorem \ref{cor:weak_MVI_convergence}. In Fig.~\ref{fig:performance of SPEG on WMVI_a}, we use a batchsize of $6$. This plot illustrates that for some weak MVI problems the requirement on the step-size from Theorem~\ref{cor:weak_MVI_convergence} can be too pessimistic and \algname{SPEG} with relatively small batchsize achieves reasonable accuracy of the solution. The choice of batchsize ensures that bound \eqref{eq:SPEG_weak_MVI_batchsize} holds and $\delta$ is small enough to guarantee convergence of \algname{SPEG}. We also tried to compare \algname{SPEG} with \algname{SEG+} from \cite{pethick2022escaping}, however, the authors do not mention their choice of update step-size. We examined several decreasing update step-size for which \algname{SEG+} failed to converge. Further details on experiments can be found in Appendix \ref{subsec:FurtherDetails}.

\section*{Acknowledgement}
Sayantan Choudhury acknowledges support from the Acheson J. Duncan Fund for the Advancement of Research in Statistics. Nicolas Loizou acknowledges support from CISCO Research. 
\bibliography{sample}

\appendix
\newpage
\part*{Supplementary Material}

We organize the Supplementary Material as follows: Section~\ref{section: Related Work} discusses the existing literature related to our work. In Section~\ref{section: tachnical preliminaries}, we present some technical lemmas required for our analysis, while in Section~\ref{sec:BoundedVarianceCounterExample}, we provide a simple problem where the bounded variance assumption does not hold. Then, in Section~\ref{section: proofs of results on expected residual}, we provide the proofs of propositions related to Expected Residual. Next, Section~\ref{section: main convergence analysis results} presents the proofs of the main theorems, while a proposition related to arbitrary sampling is proved in Section~\ref{section: further results on arbitrary sampling}. Finally, additional numerical experiments are presented in Section~\ref{Appendix_AddExperiments}.

{\small\tableofcontents}

\newpage
\section{Further Related Work}\label{section: Related Work}

The references necessary to motivate our work and connect it to the most relevant literature are included in the appropriate sections of the main body of the paper. In this section, we present a broader view of the
literature, including more details on closely related work and more references to papers that are not directly related to our main results. 

\begin{itemize}[leftmargin=*]
\setlength{\itemsep}{0pt}
\item \textbf{Classes of Structured Non-monotone Operators.} 
With an increasing interest in improved computational speed, first-order methods are the primary choice for solving VIPs. However, computation of an approximate first-order locally optimal solution of a general non-monotone VIP is intractable \citep{daskalakis2021complexity,lee2021fast}. It motivates us to exploit the additional structures prevalent in large classes of non-monotone VIPs. Recently \cite{gorbunov2022stochastic, hsieh2019convergence} provide convergence guarantees of stochastic methods for solving quasi-strongly monotone VIPs, while \cite{hsieh2020explore} for problems satisfying error-bound conditions. \cite{diakonikolas2021efficient} defined the notion of a weak MVI~\eqref{eq: weak MVI} covering classes of non-monotone VIPs. 

\vspace{2mm}

\item \textbf{Assumptions on Operator Noise.} The standard analysis of stochastic methods for solving VIPs relies on bounded variance assumption. \cite{bohm2022solving, diakonikolas2021efficient, hsieh2019convergence, gidel2018variational} use bounded variance assumption (i.e. $\E \|F_i(x) - F(x)\|^2 \leq \sigma^2$ for all $x$) while \cite{nemirovski2009robust, abernethy2021last} assume bounded operators for their analysis. However, there are examples of simple quadratic games that do not satisfy these conditions. It has motivated researchers to look for alternative/relaxed assumptions on distributions. \cite{loizou2021stochastic} provides convergence of Stochastic Gradient Descent Ascent Method under Expected Cocoercivity. \cite{hsieh2020explore, mishchenko2020revisiting} considered alternative assumptions for analyzing Stochastic Extragradient Methods that do not imply boundedness of the variance. However, there is no analysis of single-call extragradient methods without bounded variance assumption.

\vspace{2mm}

\item \textbf{Weak Minty Variational Inequalities.}  Numerous contemporary studies look to identify first-order methods for efficiently solving min-max optimization problems. It varies from simple convex-concave to nontrivial nonconvex nonconcave objectives. Though there has been a significant development in the convex-concave setting, \cite{daskalakis2021complexity} demonstrates that even finding local solutions are intractable for general nonconvex nonconcave objectives. Therefore, researchers seek to identify the structure of objective functions for which it is possible to resolve the intractability issues. \cite{diakonikolas2021efficient} proposes the notion of non-monotonicity, which generalizes the existence of a Minty solution (i.e., $\rho = 0 $ in \eqref{eq: weak MVI}). This problem is known as weak Minty variational inequality in the literature. \cite{diakonikolas2021efficient, pethick2022escaping} provides convergence guarantees of the Extragradient Method for weak Minty variational inequality. They establish a convergence rate of $\cO(1/k)$ for the squared operator norm. \cite{lee2021fast} shows that it is possible to have an accelerated extragradient method even for non-monotone problems. Furthermore, \cite{bohm2022solving} provides a convergence guarantee for the \algname{SOG} with a complexity bound of $\cO(\varepsilon^{-2})$. However, all papers exploring stochastic extragradient methods for solving weak Minty variational inequality consider bounded variance assumption~\citep{bohm2022solving, diakonikolas2021efficient}. Moreover, all algorithms solving Weak Minty variational inequality require increasing batchsize. Recently, \cite{pethick2023solving} introduced \algname{BCSEG+} which can solve weak minty variational inequality without increasing batchsize. \algname{BCSEG+} involves three oracle calls per iteration and addition of a bias-corrected term in the extrapolation step.

\vspace{2mm}

\item \textbf{Arbitrary Sampling Paradigm.} As we mentioned in the main paper, the stochastic reformulation \eqref{Reformulation} of the original problem \eqref{eq: Variational Inequality Definition} allows us to analyze single-call extragradient methods under the arbitrary sampling paradigm. That is, provide  a unified analysis for \algname{SPEG} that captures multiple sampling strategies, including $\tau$-minibatch and importance samplings. An arbitrary sampling analysis of a stochastic optinmization method was first proposed in the context of the randomized coordinate descent method for solving strongly convex functions in \cite{richtarik2013optimal}. Since then, several other stochastic methods were studied in this regime, including accelerated coordinate descent algorithms \citep{qu2016coordinate, hanzely2019accelerated}, randomized iterative methods for solving consistent linear systems~\citep{richtarik2020stochastic, loizou2020momentum,loizou2020convergence}, randomized gossip algorithms \citep{loizou2016new, loizou2021revisiting}, stochastic gradient descent (\algname{SGD})~\citep{gower2019sgd,gower2021sgd}, and variance reduced methods~\citep{qian2019saga, horvath2019nonconvex, khaled2020unified}. The first analysis of stochastic algorithms under the arbitrary sampling paradigm for solving variational inequality problems was proposed in ~\cite{loizou2020stochastic,loizou2021stochastic}.  In \cite{loizou2020stochastic,loizou2021stochastic}, the
authors focus on algorithms like the stochastic Hamiltonian method, the stochastic gradient descent ascent,
and the stochastic consensus optimization. These ideas were later extended to the case of Stochastic Extragradient by \cite{gorbunov2022stochastic}. To the best of our knowledge, our work is the first that provides an analysis of single-call extragradient methods under the arbitrary sampling paradigm.

\vspace{2mm}

\item \textbf{Overparameterized Models and Interpolation.} For a function $f(x) \coloneqq \frac{1}{n} \sum_{i = 1}^n f_i(x)$ we say that interpolation condition holds if there exists $x^*$ such that $\min_{x} f_i(x) = f_i(x^*)$ for all $i \in [n]$ (or equivalently $\nabla f_i(x^*) = 0$ for smooth convex functions)~\citep{gower2021sgd}. The interpolation condition is satisfied when the underlying models are sufficiently overparameterized ~\citep{vaswani2019fast}. Some known examples include deep matrix factorization and classification using neural networks \citep{arora2019implicit,rolinek2018l4, vaswani2019fast}. The interpolated model structure enables \algname{SGD} and other optimization algorithms to have faster convergence \citep{gower2021sgd,pmlr-v130-loizou21a, d2021stochastic}. 
Inspired by this, one can extend the notion of the interpolation condition to operators. In this scenario, we say that the VIP \eqref{eq: Variational Inequality Definition} is interpolated if there exists solution $x^*$ of \eqref{eq: Variational Inequality Definition} such that $F_i(x^*) = 0$ for all $i \in [n]$. This concept has been explored for analyzing the stochastic extragradient method in \cite{vaswani2019painless, li2022convergence}. We highlight that our proposed theorems show fast convergence of \algname{SPEG} in this interpolated regime (when $\sigma_*^2=0$). To the best of our knowledge, our work is the first that proves such convergence for \algname{SPEG}. In Fig.~\ref{fig: Interpolated Model}, we experimentally verify the fast convergence for solving a strongly monotone interpolated problem.

\vspace{2mm}

\item \textbf{Deterministic Extragradient Methods.} The Extragradient method (\algname{EG})~\citep{korpelevich1976extragradient} and its single-call variant, Optimistic Gradient (\algname{OG}) \citep{popov1980modification}, were proposed to overcome the convergence issues of gradient descent-ascent method for solving monotone problems. Since their introduction, these methods have been revisited and explored in various ways. \cite{mokhtari2020unified} analyzed \algname{EG} and \algname{OG} as an approximation of the Proximal Point method to solve bilinear and strongly convex-strongly concave min-max problems. \cite{solodov1999hybrid} and \cite{ryu2019ode} provide the best-iterate convergence guarantees of \algname{EG} and \algname{OG} with a rate of $\cO(\nicefrac{1}{K})$ for solving monotone problems. However, providing a last-iterate convergence rate of \algname{EG} and \algname{OG} for monotone VIPs has been a long-lasting open problem that was only recently resolved. The works of \cite{golowich2020tight, gorbunov2022extragradient, cai2022tight} prove a last-iterate $\cO(\nicefrac{1}{K})$ convergence rate for these methods. Finally, in the deterministic setting, some recent works provide convergence analysis of \algname{EG} and \algname{OG} for solving weak MVI~\eqref{eq: weak MVI} \citep{diakonikolas2021efficient, pethick2022escaping, bohm2022solving, gorbunov2022convergence}.

\end{itemize}

\newpage
\section{Technical Preliminaries}\label{section: tachnical preliminaries}
Throughout our work, we assume 
\begin{assumption}
Operator $F$ in \eqref{eq: Variational Inequality Definition} is L Lipschitz, i.e.,  $\forall x,y \in \mathbb{R}^d$ operator $F$ satisfies
\begin{equation}\label{eq: F lipschitz}
    \|F(x) - F(y)\| \leq L \|x - y\|.
\end{equation}
Operators $F_i: \mathbb{R}^d \to \mathbb{R}^d $  of problem \eqref{eq: Variational Inequality Definition} are $L_i$- Lipschitz, i.e., $\forall x,y \in \mathbb{R}^d$ operator $F_i$ satisfies
\begin{equation}\label{eq: F_i lipschitz}
    \|F_i(x) - F_i(y)\| \leq L_i \|x - y\|.
\end{equation}
\end{assumption}
In our proofs, we often use the following simple inequalities.

\begin{lemma}\label{Lemma: Young's Inequality} For all $a, b, a_1, a_2, \cdots a_n \in \mathbb{R}^d, n \geq 1, \alpha > 0$, we have the following inequalities:

\begin{eqnarray}
     \la a, b \ra &\leq& \|a\| \|b\|, \label{eq: Cauchy Schwarz Inequality} \\
     \la a,b \ra &\leq& \frac{1}{2\alpha}\|a\|^2 + \frac{\alpha}{2}\|b\|^2,\\
    \|a+b\|^2 &\leq& 2\|a\|^2 + 2\|b\|^2,  \label{eq: Young's Inequality}\\
    \|a\|^2 &\geq& \frac{1}{2}\|a+b\|^2 - \|b\|^2, \\
    \bigg\|\sum_{i = 1}^n a_i \bigg\|^2 &\leq& n \sum_{i = 1}^n \|a_i\|^2. \label{eq: n dimensional young's inequality}    
\end{eqnarray}

\end{lemma}
Inequality~\eqref{eq: Young's Inequality} is well known as Young's Inequality. Now, we present a simple property of unbiased estimators.

\vspace{2mm}

\begin{lemma}\label{Lemma: variance of an unbiased estimator}
For an unbiased estimator $g$ of operator $F$ i.e. $\E[g(x)] = F(x)$ we have 
\begin{equation}\label{eq: variance of an unbiased estimator}
    \E \|g(x) - F(x)\|^2 = \E \|g(x)\|^2 - \|F(x)\|^2.
\end{equation}
\end{lemma}
Next, we present the following lemma from \cite{stich2019unified}, which plays a vital role in proving the convergence guarantee of Theorem \ref{Theorem: Total number of iteration knowledge}.
\begin{lemma}(Simplified Verison of Lemma 3 from \cite{stich2019unified})\label{Lemma: Stich lemma}
Let the non-negative sequence $\{r_k\}_{k \geq 0}$ satisfy the relation $r_{k+1} \leq (1 - a \gamma_k)r_k + c \gamma_k^2$ for all $k \geq 0$, parameters $a,c \geq 0$ and any non-negative sequence $\{\g_k\}_{k \geq 0}$ such that $\gamma_k \leq \frac{1}{h}$ for some $h \geq a, h >0$. Then for any $K \geq 0$ one can choose $\{\gamma_k\}_{k\geq 0}$ as follows:
\begin{equation*}
    \begin{split}
        \text{if $K \leq \frac{h}{a}$}, & \qquad \gamma_k = \frac{1}{h},\\
        \text{if $K > \frac{h}{a}$ and $k < k_0$}, & \qquad \gamma_k = \frac{1}{h},\\
        \text{if $K > \frac{h}{a}$ and $k \geq k_0$}, & \qquad \gamma_k = \frac{2}{a(\kappa + k - k_0)},
    \end{split}
\end{equation*}
where $\kappa = \frac{2h}{a}$ and $k_0 = \ceil*{\frac{K}{2}}$. For this choice of $\g_k$ the following inequality holds:
$$
r_K \leq \frac{32hr_0}{a} \exp{\Bigg(-\frac{aK}{2h}\Bigg)} + \frac{36c}{a^2K}.
$$
\end{lemma}
We use the next lemma to bound the trace of matrix products. 
\vspace{2mm}

\begin{lemma}
For positive semidefinite matrices $A,B \in \mathbb{R}^{d \times d}$ we have 
\begin{equation}\label{eq: trace inequality}
    \text{tr}(AB) \leq \lambda_{\max}(B) \text{tr}(A),
\end{equation}
where $\lambda_{\max}(B)$ denotes the maximum eigenvalue of $B$.
\end{lemma}
Next lemma proves equivalence of \algname{SPEG} and \algname{SOG}:
\begin{proposition}[\textbf{Equivalence of \algname{SPEG} and \algname{SOG}}] \label{proposition: equivalence of SPEG and SOG}
Consider the iterates of \algname{SPEG} $\{x_k, \hat{x}_k\}_{k = 1}^{\infty}$ with constant step-sizes $\om_k = \om, \gamma_k = \gamma$ in \eqref{SPEG_UpdateRule}. Then $\hat{x}_k$ follows the iteration rule of \algname{SOG} i.e.
\begin{eqnarray}
     \hat{x}_{k+1} = \hat{x}_k - \om_k F_{v_k}(\hat{x}_k) - \gamma_k \large[F_{v_k}(\hat{x}_k) - F_{v_{k-1}}(x_{k-1})\large]
\end{eqnarray}
\end{proposition}
\begin{proof}
From the update rule of \algname{SPEG}~\eqref{SPEG_UpdateRule} we get
\begin{equation*}
    \begin{split}
        \hat{x}_{k+1} = \quad & x_{k + 1} - \gamma F_{v_k}(\hat{x}_k) \\
        = \quad & x_k - \om F_{v_k}(\hat{x}_k) - \gamma F_{v_k}(\hat{x}_k) \\
        = \quad & x_k - (\om + \gamma) F_{v_k} (\hat{x}_k) \\
        = \quad & \hat{x}_k + \gamma F_{v_{k - 1}}(\hat{x}_{k - 1}) - (\om + \gamma) F_{v_k} (\hat{x}_k) \\
        = \quad & \hat{x}_k  - \om  F_{v_k} (\hat{x}_k) - \gamma \Big(F_{v_k} (\hat{x}_k) - F_{v_{k - 1}}(\hat{x}_{k - 1}) \Big).
    \end{split}
\end{equation*}
This shows that \algname{SPEG} iterations are equivalent to \algname{SOG}, with $\hat{x}_k$ being the $k$-th iterate of \algname{SOG}.
\end{proof}

\section{Example:  A Problem where the Bounded Variance Condition not Hold}\label{sec:BoundedVarianceCounterExample}
Here, we provide a simple problem that does not satisfy the bounded variance assumption. Consider the linear regression problem
\begin{eqnarray*}
        \min_{x \in \mathbb{R}}f(x) := \frac{1}{2} (a_1x - b_1)^2 + \frac{1}{2} (a_2x - b_2)^2
\end{eqnarray*}
where $x \in \mathbb{R}$. Here $f_1(x) = (a_1x - b_1)^2$ and $f_2(x) = (a_2x - b_2)^2$. Now consider the estimator $g(x)$ of $\nabla f(x)$ under uniform sampling i.e. $g(x)$ takes the value $\nabla f_1(x)$ with probability $\frac{1}{2}$ and $\nabla f_2(x)$ with probability $\frac{1}{2}$. Then we have
\begin{eqnarray*}
        \Exp \|g(x) -  \nabla f(x)\|^2 & = & \frac{1}{2} \|\nabla f_1 (x) - \nabla f(x)\|^2 + \frac{1}{2} \|\nabla f_2(x) - \nabla f(x)\|^2 \\
        & = & \frac{1}{2} \cdot \frac{1}{4}\|\nabla f_1 (x) - \nabla f_2(x)\|^2 + \frac{1}{2} \cdot \frac{1}{4}\|\nabla f_2 (x) - \nabla f_1(x)\|^2 \\
        & = & \frac{1}{4} \|\nabla f_1(x) - \nabla f_2(x)\|^2 \\
        & = & \frac{1}{4} \left( 2(a_1x - b_1)a_1 - 2(a_2x - b_2)a_2 \right)^2 \\
        & = & \left( (a_1^2 - a_2^2) x - (a_1b_1 - a_2b_2) \right)^2
\end{eqnarray*}
Therefore, $\Exp \|g(x) -  \nabla f(x)\|^2$ is a quadratic function of $x$ with the coefficient of $x$ being positive. Hence, as $x \to \infty$, we have $\Exp \|g(x) -  \nabla f(x)\|^2 \to \infty$, which means that a constant can not bound the variance.

\newpage
\section{Proofs of Results on Expected Residual}\label{section: proofs of results on expected residual}
\subsection{Proof of Lemma \ref{Lemma: variance bound}}
\begin{proof}
Using Young's Inequality~\eqref{eq: Young's Inequality}, we get 
\begin{equation*}
    \begin{split}
        \E \|g(x) - F(x)\|^2 \overset{\eqref{eq: Young's Inequality}}{\leq} \quad & 2 \E \|g(x) - F(x) - g(x^*)\|^2 + 2 \E \|g(x^*)\|^2 \\
        \overset{\eqref{eq: ER Condition}}{\leq} \quad & \delta \|x -x^*\|^2 + 2 \E \|g(x^*)\|^2. \\
    \end{split}
\end{equation*}
Then breaking down the RHS, we obtain 
\begin{equation*}
    \begin{split}
         \E \|g(x)\|^2 - \|F(x)\|^2 \overset{\eqref{eq: variance of an unbiased estimator}}{\leq} \delta \|x -x^*\|^2 + 2 \E \|g(x^*)\|^2.
    \end{split}
\end{equation*}
Now we rearrange the terms and set $\sigma_*^2 = \E \|g(x^*)\|^2$ to complete the proof of this Lemma.
\end{proof}

\begin{proposition}\label{Proposition lipschitz implies ER}
 If $F_i$ are $L_i$-lipschitz then Expected Residual condition~\eqref{eq: ER Condition} holds. In that case
\begin{equation*}
    \begin{split}
        \delta =  \frac{2}{n} \sum_{i = 1}^n L_i^2 \E (v_i^2). 
    \end{split}
\end{equation*}
In addition, if $F$ is $\mu$-quasi strongly monotone \eqref{eq: Strong Monotonicity} then we have 
\begin{equation*}
    \begin{split}
       \delta = \frac{2}{n} \sum_{i = 1}^n L_i^2 \E (v_i^2) - 2\mu^2. 
    \end{split}
\end{equation*}
\end{proposition}

\begin{proof}
Note that
\begin{eqnarray}
        \E\|(F_v(x) - F_v(x^*)) - (F(x) - F(x^*))\|^2 &=&  \E\|F_v(x) - F_v(x^*)\|^2 + \|F(x) - F(x^*)\|^2 \notag\\
        && \quad - 2 \E \la F_v(x) - F_v(x^*), F(x) - F(x^*)\ra \notag\\
        &=&  \E \|F_v(x) - F_v(x^*)\|^2 - \|F(x) - F(x^*)\|^2 \notag\\
        &=&  \E \|F_v(x) - F_v(x^*)\|^2  - \|F(x)\|^2  \notag\\
        &=& \E \bigg\| \frac{1}{n} \sum_{i = 1}^n v_i (F_i(x) - F_i(x^*)) \bigg\|^2  - \|F(x)\|^2 \notag\\
        &=& \frac{1}{n^2} \E \bigg \| \sum_{i = 1}^n v_i (F_i(x) - F_i(x^*)) \bigg\|^2  - \|F(x)\|^2 \notag\\
        &\overset{\eqref{eq: n dimensional young's inequality}}{\leq}& \frac{1}{n} \sum_{i = 1}^n \E(v_i^2) \|F_i(x) - F_i(x^*)\|^2  - \|F(x)\|^2 \notag\\
        &\overset{\eqref{eq: F_i lipschitz}}{\leq}& \frac{\|x - x^*\|^2}{n} \sum_{i = 1}^n \E(v_i^2) L_i^2  - \|F(x)\|^2 \label{eq: Expected Residual bound for lipschitz quasi strongly monotone F}.
\end{eqnarray}
The first part of the lemma follows by ignoring the positive term $\|F(x)\|^2$. For the second part we assume $F$ is $\mu$-quasi strongly monotone. Then we have
$$\mu \|x - x^*\|^2 \overset{\eqref{eq: Strong Monotonicity}}{\leq} \la F(x), x - x^*\ra \overset{\eqref{eq: Cauchy Schwarz Inequality}}{\leq} \|F(x)\| \|x - x^*\|.$$ 
Cancelling $\|x - x^*\|$ from both sides we get 
\begin{equation}\label{eq: lower bound on norm of quasi strongly monotone F}
\mu \|x - x^*\| \leq \|F(x)\|.    
\end{equation}
Therefore we have the following bound for $\mu$-quasi strongly monotone operator $F$:
\begin{equation*}
        \E \|(F_v(x) - F_v(x^*)) - (F(x) - F(x^*))\|^2
        \overset{\eqref{eq: Expected Residual bound for lipschitz quasi strongly monotone F}, \eqref{eq: lower bound on norm of quasi strongly monotone F}}{\leq} \Bigg( \frac{1}{n} \sum_{i = 1}^n \E(v_i^2) L_i^2 - \mu^2 \Bigg) \|x - x^*\|^2.
\end{equation*}
This proves the second part of the lemma. This lemma ensures that the Lipschitz property is sufficient to guarantee Expected Residual~\eqref{eq: ER Condition} condition. 
\end{proof}

\subsection{Proof of Proposition \ref{Prop_SufficientCondition}}
\begin{proof}
    Proposition \ref{Proposition lipschitz implies ER} implies that Lipschitzness of all operators $F_i$ is enough to ensure that \ref{eq: ER Condition} holds. For $\tau$- minibatch sampling, denote the matrix $\textbf{R} = \Big(F_1(x) - F_1(x^*), \cdots, F_n(x) - F_n(x^*)\Big) \in \mathbb{R}^{d \times n}$. Then we obtain the following bound:
\begin{eqnarray*}
\E \|F_v(x) - F_v(x^*) - (F(x) - F(x^*))\|^2 \hspace{-3mm} &=& \hspace{-3mm} \E \left\| \frac{1}{n} \sum_{i = 1}^n v_i (F_i(x) - F_i(x^*)) - (F_i(x) - F_i(x^*)) \right \|^2 \\
&=& \hspace{-3mm}\frac{1}{n^2} \E \bigg \| \sum_{i = 1}^n (v_i -1) (F_i(x) - F_i(x^*)) \bigg\|^2 \\
&=& \hspace{-3mm}\frac{1}{n^2} \E \big \| \textbf{R}(v - \mathbf{1}) \big\|^2 \\
&=& \hspace{-3mm} \frac{1}{n^2} \E (v - \mathbf{1})^{\intercal} \textbf{R}^{\intercal}\textbf{R} (v - \mathbf{1}) \\
&=& \hspace{-3mm} \frac{1}{n^2} \E \bigg( \text{tr} \bigg( \textbf{R}^{\intercal}\textbf{R} (v - \mathbf{1}) (v - \mathbf{1})^{\intercal} \bigg) \bigg) \\
&=& \hspace{-3mm}\frac{1}{n^2} \text{tr} \bigg( \textbf{R}^{\intercal}\textbf{R} \E \bigg((v - \mathbf{1}) (v - \mathbf{1})^{\intercal} \bigg) \bigg) \\
&=& \hspace{-3mm}\frac{1}{n^2} \text{tr} \bigg( \textbf{R}^{\intercal}\textbf{R} \textbf{Var}[v] \bigg) \bigg) \\
&\overset{\eqref{eq: trace inequality}}{\leq}& \hspace{-3mm} \frac{\lambda_{\max}\big(\textbf{Var}[v]\big)}{n^2} \text{tr}(\textbf{R}^{\intercal}\textbf{R}) \\
&=& \hspace{-3mm} \frac{\lambda_{\max}\big(\textbf{Var}[v]\big)}{n^2} \sum_{i = 1}^n \|F_i(x) - F_i(x^*)\|^2 \\
&\overset{\eqref{eq: F_i lipschitz}}{\leq}& \frac{\lambda_{\max}(\textbf{Var}[v]) \|x - x^*\|^2}{n^2}\sum_{i = 1}^n L_i^2.   
\end{eqnarray*}
From the proof details of Lemma F.3 in \cite{sebbouh2019towards} we have $\lambda_{\max}(\textbf{Var}[v]) = \frac{n(n - \tau)}{\tau (n - 1)}$ for $\tau$-minibatch sampling. Thus we obtain
\begin{equation*}
    \E \big\|F_v(x) - F_v(x^*) - (F(x) - F(x^*)) \big  
    \|^2 \leq \frac{2(n - \tau)}{n \tau (n - 1)} \sum_{i = 1}^n L_i^2 \|x - x^*\|^2.
\end{equation*}
Now we focus on the derivation of $\sigma_*^2 = \E\|F_v(x^*)\|^2$ for $\tau$-minibatch sampling. We expand $\E \|F_v(x^*)\|^2$ as follows:
\begin{eqnarray}
        \E \|F_v(x^*)\|^2 &=& \frac{1}{n^2} \E \bigg\|\sum_{i = 1}^n v_iF_i(x^*) \bigg\|^2 \notag \\
        &=& \frac{1}{n^2} \E \bigg\|\sum_{i \in S} \frac{1}{p_i}F_i(x^*) \bigg\|^2 \notag\\
        &=& \frac{1}{n^2} \E \bigg\|\sum_{i = 1}^ n \textbf{1}_{i \in S}\frac{1}{p_i}F_i(x^*) \bigg\|^2 \notag\\
        &=& \frac{1}{n^2}\E \bigg \langle \sum_{i = 1}^n \textbf{1}_{i \in S}\frac{1}{p_i} F_i(x^*), \sum_{j = 1}^n \textbf{1}_{j \in S}\frac{1}{p_j} F_j(x^*) \bigg \rangle \notag\\
        &=& \frac{1}{n^2} \sum_{i,j = 1}^n \frac{P_{ij}}{p_i p_j} \langle F_i(x^*), F_j(x^*) \rangle, \label{eq: expansion for sigma}
\end{eqnarray}
where $P_{ij} = P(i, j \in S)$ and $p_i = P(i \in S)$. For $\tau$-minibatch sampling, we obtain $P_{ij} = \frac{\tau(\tau - 1)}{n(n-1)}$ and $p_i = \frac{\tau}{n}$. Plugging in these values of $P_{ij}$ and $p_i$ in \eqref{eq: expansion for sigma} we get the closed-form expression of $\sigma_*^2$. This completes the proof of Proposition \ref{Prop_SufficientCondition}.
\end{proof}

\subsection{Proof of Proposition \ref{Proposition connecting assumptions}}
Here we enlist the assumptions made on operators. Suppose $g$ is an estimator of operator $F$.
\begin{equation*}
    \begin{split}
        & \textbf{1. Bounded Operator:} \quad  \E\|g(x)\|^2 \leq \sigma^2 \\
        & \textbf{2. Bounded Variance:} \quad  \E\|g(x) - F(x)\|^2 \leq \sigma^2 \\
        & \textbf{3. Growth Condition:} \quad  \E\|g(x)\|^2 \leq \alpha \|F(x)\|^2 + \beta \\
        & \textbf{4. Expected Co-coercivity:} \quad \E \|g(x) - g(x^*)\|^2 \leq l_F \la F(x), x-x^* \ra \\
        & \textbf{5. Expected Residual:} \quad \E \|(g(x) - g(x^*)) - (F(x) - F(x^*))\|^2 \leq \frac{\delta}{2} \|x-x^*\|^2 \\
        & \textbf{6. Bound from Lemma \ref{Lemma: variance bound}:} \quad\E\|g(x)\|^2 \leq \delta \|x - x^*\|^2 + \|F(x)\|^2 + 2\sigma_*^2 \\
        & \textbf{7. $F_i$ are Lipschitz:} \quad \|F_i(x) - F_i(y)\| \leq L_i\|x - y\|\quad \forall \; i=1,\ldots,n
    \end{split}
\end{equation*}

\begin{proof} Here we will prove Proposition \ref{Proposition connecting assumptions}
\begin{itemize}
    \item $1 \implies 2$. Note that $\E \|g(x)\|^2 \leq \sigma^2 \leq \|F(x)\|^2 + \sigma^2 \implies \E \|g(x) - F(x)\| \leq \sigma^2$.
    \item $2 \implies 3$. Here $\E\|g(x) - F(x)\|^2 \leq \sigma^2 \implies \E \|g(x)\|^2 \leq \|F(x)\|^2 + \sigma^2$ as $g$ is an unbiased for estimator of $F$. Then take $\alpha = 1$ and $\beta = \sigma^2$.
    \item $3 \implies 6$. Note that $\E \|g(x)\|^2 \leq \alpha \|F(x)\|^2 + \beta \leq \alpha L^2 \|x - x^*\|^2 + \beta$. The last inequality follows from lipschitz property of $F$ and $F(x^*) = 0$. Then choose $\delta = \alpha L^2$ and $\sigma_*^2 = \nicefrac{\beta}{2}$ to get the result.
    \item $4\implies 5$. Note that expected cocoercivity and $L$-Lipschitzness of $F$ imply $\E \|(g(x) - g(x^*)) - (F(x) - F(x^*))\|^2 = \E \|g(x) - g(x^*)\|^2 - \|F(x) - F(x^*)\|^2 \leq \E \|g(x) - g(x^*)\|^2 \leq l_F \la F(x), x-x^* \ra \overset{\eqref{Lemma: Young's Inequality}}{\leq} \frac{l_F}{2L} \|F(x)\|^2 + \frac{l_F L}{2} \|x -x^*\|^2 \leq l_F L \|x -x^*\|^2$.
    \item $7 \implies 5$. This follows from Proposition \ref{Proposition lipschitz implies ER}.
    \item $5 \implies 6$. This follows from Lemma \ref{Lemma: variance bound}
\end{itemize}
\end{proof}

\newpage
\section{Main Convergence Analysis Results}\label{section: main convergence analysis results}
First, we present some results followed by iterates of \algname{SPEG}. These will play a key role in proving the Theorems later in this section. Recall that iterates of \algname{SPEG} are
\begin{equation*}
    \begin{split}
        \hat{x}_k & = x_k - \gamma_k F_{v_{k - 1}}(\hat{x}_{k - 1}), \\
    x_{k + 1} & = x_k - \om_k F_{v_k}(\hat{x}_k).
    \end{split}
\end{equation*}

\begin{lemma}\label{Lemma: Breakdown} For \algname{SPEG} iterates with step-size $\omega_k = \gamma_k = \om$, we have
\begin{equation}\label{eq: Breakdown}
    \begin{split}
        \|x_{k+1} - x^*\|^2 = \|x_{k+1} - \hat{x}_k\|^2 + \|x_k - x^*\|^2 - \|\hat{x}_k - x_k\|^2 - 2 \om \la F_{v_k}(\hat{x}_k), \hat{x}_k - x^*\ra.
    \end{split}
\end{equation}
\end{lemma}

\begin{proof}
We have
\begin{eqnarray*}
    \|x_{k+1} - x^*\|^2 &=& \|x_{k+1} - \hat{x}_k + \hat{x}_k - x_k + x_k - x^*\|^2 \\
    &=& \|x_{k+1} - \hat{x}_k\|^2 + \|\hat{x}_k - x_k\|^2 + \|x_k - x^*\|^2 +  2 \la \hat{x}_k - x_k, x_k - x^* \ra \\
    && \quad + 2 \la x_{k+1} - \hat{x}_k, \hat{x}_k - x_k \ra + 2 \la x_{k+1} - \hat{x}_k, x_k - x^* \ra \\
    &=& \|x_{k+1} - \hat{x}_k\|^2 + \|\hat{x}_k - x_k\|^2 + \|x_k - x^*\|^2 + 2 \la x_{k+1} - \hat{x}_k, \hat{x}_k - x^* \ra \\
    && \quad + 2 \la \hat{x}_k - x_k, x_k - x^* \ra\\
    &=& \|x_{k+1} - \hat{x}_k\|^2 + \|\hat{x}_k - x_k\|^2 + \|x_k - x^*\|^2 + 2 \la x_{k+1} - \hat{x}_k, \hat{x}_k - x^* \ra \\
    && \quad + 2 \la \hat{x}_k - x_k, x_k - \hat{x}_k + \hat{x}_k - x^* \ra \\
    &=&\|x_{k+1} - \hat{x}_k\|^2 + \|\hat{x}_k - x_k\|^2 + \|x_k - x^*\|^2 + 2 \la x_{k+1} - \hat{x}_k, \hat{x}_k - x^* \ra \\
    && \quad + 2 \la \hat{x}_k - x_k,  \hat{x}_k - x^* \ra - 2\|\hat{x}_k - x_k\|^2 \\
    &=&\|x_{k+1} - \hat{x}_k\|^2 - \|\hat{x}_k - x_k\|^2 + \|x_k - x^*\|^2 + 2 \la x_{k+1} - \hat{x}_k, \hat{x}_k - x^* \ra \\
    &&\quad + 2 \la \hat{x}_k - x_k,  \hat{x}_k - x^* \ra  \\
    &=& \|x_{k+1} - \hat{x}_k\|^2 - \|\hat{x}_k - x_k\|^2 + \|x_k - x^*\|^2 + 2 \la x_{k+1} - x_k, \hat{x}_k - x^* \ra \\
    &=& \|x_{k+1} - \hat{x}_k\|^2 - \|\hat{x}_k - x_k\|^2 + \|x_k - x^*\|^2 - 2 \om \la F_{v_k}(\hat{x}_k), \hat{x}_k - x^* \ra.
\end{eqnarray*}
\end{proof}

\begin{lemma}\label{Lemma: bound on difference of gradients}
Let $F$ be $L$-Lipschitz, and let \ref{eq: ER Condition} hold. Then \algname{SPEG} iterates satisfy
\begin{equation}\label{eq: bound on difference of gradients}
    \begin{split}
        \E_{\mathcal{D}} \|F_{v_k}(\hat{x}_k) - F_{v_{k-1}}(\hat{x}_{k-1})\|^2 \leq \quad & \delta \|\hat{x}_k - x^*\|^2 + 2\delta \|\hat{x}_{k-1} - x^*\|^2 + 2 L^2 \|\hat{x}_k - \hat{x}_{k-1}\|^2  + 6 \sigma_*^2.
    \end{split}
\end{equation}
\end{lemma}
\begin{proof}
\begin{eqnarray*}
    \E_{\mathcal{D}} \|F_{v_k}(\hat{x}_k) - F_{v_{k-1}}(\hat{x}_{k-1})\|^2 &=&\E_{\mathcal{D}} \|F_{v_k}(\hat{x}_k) - F(\hat{x}_{k})\|^2 + \E_{\mathcal{D}} \|F(\hat{x}_k) - F_{v_{k-1}}(\hat{x}_{k-1})\|^2  \\
    && \quad + 2 \E_{\mathcal{D}} \la F_{v_k}(\hat{x}_k) - F(\hat{x}_{k}), F(\hat{x}_k) - F_{v_{k-1}}(\hat{x}_{k-1}) \ra \\
    &=& \E_{v_{k}} \|F_{v_k}(\hat{x}_k) - F(\hat{x}_{k})\|^2 + \E_{\mathcal{D}} \|F(\hat{x}_k) - F_{v_{k-1}}(\hat{x}_{k-1})\|^2 \\
    &\overset{\eqref{eq: Young's Inequality}}{\leq}& \E_{\mathcal{D}} \|F_{v_k}(\hat{x}_k) - F(\hat{x}_{k})\|^2 + 2\E_{\mathcal{D}} \|F(\hat{x}_k) - F(\hat{x}_{k-1})\|^2 \\
    && \quad + 2\E_{\mathcal{D}} \|F(\hat{x}_{k-1}) - F_{v_{k-1}}(\hat{x}_{k-1})\|^2 \\
    &=& \E_{\mathcal{D}} \|F_{v_k}(\hat{x}_k)\|^2 - \|F(\hat{x}_{k})\|^2 + 2 \|F(\hat{x}_k) - F(\hat{x}_{k-1})\|^2\\
    && \quad + 2\E_{\mathcal{D}} \|F_{v_{k-1}}(\hat{x}_{k-1}) \|^2 - 2 \|F(\hat{x}_{k-1})\|^2 \\
    &\overset{\eqref{eq: variance bound}}{\leq}& \delta \|\hat{x}_k - x^*\|^2 + 2\delta \|\hat{x}_{k-1} - x^*\|^2 + 6\sigma_*^2\\
    && \quad + 2 \|F(\hat{x}_k) - F(\hat{x}_{k-1})\|^2 \\
    &\overset{\eqref{eq: F lipschitz}}{\leq}& \delta \|\hat{x}_k - x^*\|^2 + 2\delta \|\hat{x}_{k-1} - x^*\|^2 + 6\sigma_*^2\\
    && \quad + 2 L^2 \|\hat{x}_k - \hat{x}_{k-1}\|^2.
\end{eqnarray*}
\end{proof}

\begin{lemma}\label{Lemma: conditions}
For $\om \in \bigg[0,  \frac{1}{4L} \bigg]$ the following two conditions hold:
\begin{align}
        & 2 \om(\mu - \om \delta) + 8 \om^2 L^2 -1 \leq 0,  \label{eq: conditions 1}\\
        \text{and}\quad & 8 \om^2 (\delta + L^2) \leq 1 - \om \mu + 9 \om^2\delta.  \label{eq: conditions 2}
\end{align}
\end{lemma}

\begin{proof}
Note that for $\om \in \bigg[0,  \frac{1}{4L} \bigg]$, we have $$2 \om(\mu - \om \delta) + 8 \om^2 L^2 -1 \overset{\om^2\delta \geq 0}{\leq} 2 \om\mu  + 8 \om^2 L^2 -1 \overset{\om \leq \frac{1}{4L}}{\leq} \frac{\mu}{2L} + \frac{1}{2} - 1 \overset{\mu \leq L}{\leq} 0.$$
This proves the first condition. The second condition is equivalent to $ \om(\mu - \om \delta) + 8 \om^2 L^2 -1 \leq 0 $, which is again true using similar arguments.
\end{proof}

\subsection{Proof of Theorem \ref{Theorem: constant stepsize theorem}}
\begin{proof}
For $\om \in \bigg[0, \frac{\mu}{18 \delta}\bigg]$ we have $\om(\mu - 9\om \delta) \geq 0$ and $1 - \om(\mu - 9 \om \delta) \leq 1 - \frac{\om \mu}{2}$. Then we derive
\begin{eqnarray*}
    \E_{\mathcal{D}}[\|x_{k+1} - x^*\|^2 + \|x_{k+1} - \hat{x}_k\|^2] &\overset{\eqref{eq: Breakdown}}{=}&\|x_k - x^*\|^2 + 2\E_{\mathcal{D}} \|x_{k+1} - \hat{x}_k\|^2 - \|\hat{x}_k - x_k\|^2 \\
    && \quad - 2 \om \E_{\mathcal{D}} \la F_{v_{k}}(\hat{x}_k), \hat{x}_k - x^*\ra \\
    &=&\|x_k - x^*\|^2 + 2\E_{\mathcal{D}} \|x_{k+1} - \hat{x}_k\|^2 - \|\hat{x}_k - x_k\|^2\\
    && \quad - 2 \om \la F(\hat{x}_k), \hat{x}_k - x^*\ra \\
    &\overset{\eqref{eq: Strong Monotonicity}}{\leq}&\|x_k - x^*\|^2 + 2\E_{\mathcal{D}} \|x_{k+1} - \hat{x}_k\|^2 - \|\hat{x}_k - x_k\|^2\\
    && \quad - 2 \om \mu \| \hat{x}_k - x^* \|^2 \\
    &=&\|x_k - x^*\|^2 + 2\om^2 \E_{\mathcal{D}} \|F_{v_k}(\hat{x}_k) - F_{v_{k-1}}(\hat{x}_{k-1})\|^2 \\
    && \quad - \|\hat{x}_k - x_k\|^2 - 2 \om \mu \| \hat{x}_k - x^* \|^2 \\
    &\overset{\eqref{eq: bound on difference of gradients}}{\leq}&\|x_k - x^*\|^2 + 2\om^2 \bigg(\delta \|\hat{x}_k - x^*\|^2 + 2\delta \|\hat{x}_{k-1} - x^*\|^2\\
    &&\quad  + 2 L^2 \|\hat{x}_k - \hat{x}_{k-1}\|^2  + 6\sigma_*^2 \bigg) - \|\hat{x}_k - x_k\|^2 \\
    && \quad - 2 \om \mu \| \hat{x}_k - x^* \|^2 \\
    &=&\|x_k - x^*\|^2 - 2 \om (\mu - \om \delta)  \| \hat{x}_k - x^*\|^2 \\
    && \quad + 4 \om^2 \delta  \|\hat{x}_{k-1} - x^*\|^2 + 4 \om^2 L^2 \|\hat{x}_k - \hat{x}_{k-1}\|^2 \\
    && \quad - \|\hat{x}_k - x_k\|^2 + 12 \om^2 \sigma_*^2 \\
    &\overset{\eqref{eq: Young's Inequality}}{\leq}& \|x_k - x^*\|^2 - \om (\mu - \om \delta) \| x_k - x^*\|^2 \\
    && \quad + 2 \om (\mu - \om \delta) \| x_k - \hat{x}_k\|^2 + 4 \om^2 \delta \|\hat{x}_{k-1} - x^*\|^2 \\
    &&\quad + 4 \om^2 L^2 \|\hat{x}_k - \hat{x}_{k-1}\|^2 - \|\hat{x}_k - x_k\|^2\\
    &&\quad + 12 \om^2 \sigma_*^2 \\
    &\overset{\eqref{eq: Young's Inequality}}{\leq}& \|x_k - x^*\|^2 - \om (\mu - \om \delta) \| x_k - x^*\|^2 \\
    &&\quad + 2 \om (\mu - \om \delta) \| x_k - \hat{x}_k\|^2 + 8 \om^2 \delta \|\hat{x}_{k-1} - x_k\|^2 \\
    &&\quad + 8 \om^2 \delta \|x_k - x^*\|^2+ 8 \om^2 L^2 \|\hat{x}_k - x_k\|^2 \\
    &&\quad  + 8 \om^2 L^2 \|x_k - \hat{x}_{k-1}\|^2 - \|\hat{x}_k - x_k\|^2 + 12 \om^2 \sigma_*^2 \\
    &=& (1 - \om \mu + 9 \om^2 \delta ) \|x_k - x^*\|^2 \\
    &&\quad + (8\om^2 \delta + 8\om^2 L^2 ) \|x_k - \hat{x}_{k-1} \|^2 \\
    && \quad  + (2\om(\mu - \om\delta) +  8 \om^2L^2 -1) \|x_k - \hat{x}_k \|^2 + 12\om^2 \sigma_*^2.    
\end{eqnarray*}
Then using \eqref{eq: conditions 1}, \eqref{eq: conditions 2} we have
\begin{eqnarray*}
    \E_{\mathcal{D}}[\|x_{k+1} - x^*\|^2 + \|x_{k+1} - \hat{x}_k\|^2] &\leq& (1 - \om \mu + 9 \om^2\delta) \bigg(\|x_k - x^*\|^2 + \|x_k - \hat{x}_{k-1}\|^2 \bigg) \\
    && \quad + 12 \om^2 \sigma_*^2.
\end{eqnarray*}
Then we take total expectation with respect to the algorithm to obtain the following recurrence:
\begin{equation}\label{eq: recurrece after total expectation}
    R_{k+1}^2 \leq (1 - \omega \mu + 9 \omega^2 \delta) R_k^2 + 12 \omega^2 \sigma_*^2.
\end{equation}
Using the inequality $ 1- \omega(\mu - 9 \omega \delta) \leq 1 - \frac{\omega \mu}{2}$, we have
\begin{equation}\label{eq: before recurrence}
    \begin{split}
    \E\bigg[\|x_{k+1} - x^*\|^2 + \|x_{k+1} - \hat{x}_k\|^2 \bigg] & \leq \bigg(1 - \frac{\om \mu}{2}\bigg) \E \bigg[\|x_k - x^*\|^2 + \|x_k - \hat{x}_{k-1}\|^2 \bigg] + 12 \om^2 \sigma_*^2.
    \end{split}
\end{equation}
The theorem follows by unrolling the above recurrence. In order to compute the iteration complexity of \algname{SPEG}, we consider any arbitrary $\varepsilon > 0$. Then we choose the step-size $\om$ such that $\frac{24 \om \sigma_*^2}{\mu} \leq \frac{\varepsilon}{2}$ i.e. $\om \leq \frac{\varepsilon \mu}{48 \sigma_*^2}$. Next we will choose the number of iterations $k$ such that $(1 - \frac{\om \mu}{2})^k R_0^2 \leq \frac{\varepsilon}{2}$. It is equivalent to choosing $k$ such that 
\begin{equation*}
    \log \bigg( \frac{2 R_0^2}{\varepsilon} \bigg) \leq k \log \bigg(  \frac{1}{1 - \frac{\om \mu}{2}}\bigg).
\end{equation*}
Now using the fact $\log \big( \frac{1}{\rho}\big) \geq 1 - \rho$ for $0 < \rho \leq 1$, we get $\log \Big( \frac{2 R_0^2}{\varepsilon} \Big) \leq  \frac{k\om \mu}{2}$, or equivalently $k \geq \frac{2}{\om \mu} \log \Big( \frac{2 R_0^2}{\varepsilon} \Big)$. Therefore, with step-size $\om = \min \left\{\frac{\mu}{18 \delta}, \frac{1}{4L}, \frac{\varepsilon \mu}{48 \sigma_*^2} \right\}$ we get the following lower bound on the number of iterations
\begin{equation*}
    k \geq \max \bigg\{\frac{8L}{\mu}, \frac{36 \delta}{\mu^2}, \frac{96\sigma_*^2}{\varepsilon \mu^2} \bigg\} \log \bigg( \frac{2 R_0^2}{\varepsilon} \bigg).
\end{equation*}
\end{proof}

\subsection{Proof of Theorem \ref{SPEG switching rule}}
\begin{proof}
For $\om \leq \min \big\{\frac{1}{4L}, \frac{\mu}{18 \delta} \big\}$, from Theorem \ref{Theorem: constant stepsize theorem} we obtain 
$$
R_{k+1}^2 \leq \bigg(1 - \frac{\om \mu}{2} \bigg)^{k+1} R_0^2 + \frac{24 \om \sigma_*^2}{\mu}. 
$$
Let the step-size $\om_k = \frac{2k+1}{(k+1)^2} \frac{2}{\mu}$ and $k^*$ be an integer that satisfies $\om_{k^*} \leq \Bar{\om}$. In particular this holds when $k^* \geq \ceil*{\frac{4}{\mu \Bar{\om}} - 1}$. Note that $\om_k$ is decreasing in $k$ and consequently $\om_k \leq \Bar{\om}$ for all $k \geq k^*$. Therefore, from \eqref{eq: before recurrence} we derive 
$$
R_{k+1}^2 \leq \bigg(1 - \om_k\frac{\mu}{2} \bigg) R_{k}^2 + 12\om_k^2 \sigma_*^2
$$
for all $k \geq k^*$. Then we replace $\om_k$ with $\frac{2k+1}{(k+1)^2} \frac{2}{\mu}$ to obtain
\begin{eqnarray*}
    R_{k+1}^2 &\leq& \bigg(1 - \frac{2k +1}{(k+1)^2} \bigg) R_{k}^2 + 48 \sigma_*^2 \frac{(2k + 1)^2}{\mu^2(k+1)^4} \\
     &=&\frac{k^2}{(k+1)^2} R_{k}^2 + 48 \sigma_*^2 \frac{(2k + 1)^2}{\mu^2(k+1)^4}.
\end{eqnarray*}
Multiplying both sides by $(k+1)^2$ we get
\begin{eqnarray*}
    (k+1)^2 R_{k+1}^2 &\leq& k^2 R_k^2 + \frac{48 \sigma_*^2}{\mu^2} \bigg(\frac{2k+1}{k+1} \bigg)^2 \\
    &\leq& k^2 R_k^2 + \frac{192\sigma_*^2}{\mu^2},
\end{eqnarray*}
where in the last line follows from $\frac{2k + 1}{k + 1} < 2$. Rearranging and summing the last expression for $t = k^*,\cdots, k$ we obtain 
\begin{equation*}
    \begin{split}
        & \sum_{t = k^*}^k (t+1)^2 R_{t+1}^2 -  t^2 R_t^2 \leq \frac{192 \sigma_*^2}{\mu^2}(k - k^*).
    \end{split}
\end{equation*}
Using telescopic sum and dividing both sides by $(k+1)^2$ we obtain
\begin{equation}\label{eq: decreasing stepsize bound}
    R_{k+1}^2 \leq \bigg(\frac{k^*}{k+1}\bigg)^2 R_{k^*}^2+ \frac{192 \sigma_*^2 (k - k^*)}{\mu^2 (k+1)^2}.
\end{equation}
Suppose for $k \leq k^*$, we have $\om_k = \Bar{\om} = \min \Big\{\frac{1}{4L}, \frac{\mu}{18 \delta} \Big\}$ i.e. constant step-size. Then from \eqref{eq:SPEG_const_steps_neighborhood}, we obtain $R_{k^*}^2 \leq \Big(1 - \frac{\mu \Bar{\om}}{2}\Big)^{k^*} R_0^2 + \frac{24 \Bar{\om} \sigma_*^2}{\mu}$. This bound on $R_{k^*}^2$, combined with \eqref{eq: decreasing stepsize bound} yields
\begin{equation*}
    \begin{split}
        R_{k+1}^2 \leq \bigg(\frac{k^*}{k+1}\bigg)^2 \bigg(1 - \frac{\mu \Bar{\om}}{2}\bigg)^{k^*} R_0^2 + \bigg(\frac{k^*}{k+1}\bigg)^2 \frac{24 \Bar{\om} \sigma_*^2}{\mu} + \frac{192 \sigma_*^2 (k - k^*)}{\mu^2 (k+1)^2}.
    \end{split}
\end{equation*}
Now we want to choose $k^*$ which minimizes the expression $\big(\frac{k^*}{k+1}\big)^2 \frac{24\Bar{\om} \sigma_*^2}{\mu} + \frac{192 \sigma_*^2 (k - k^*)}{\mu^2 (k+1)^2}$. Note that, it is minimized at $\frac{4}{\mu \Bar{\om}}$, hence we choose $k^* = \ceil*{\frac{4}{\mu \Bar{\om}}}$. Therefore, using this value of $k^*$, we obtain
\begin{eqnarray*}
    R_{k+1}^2 &\leq&\bigg(\frac{k^*}{k+1}\bigg)^2 \bigg(1 - \frac{2}{k^*} \bigg)^{k^*} R_0^2 + \frac{24 \sigma_*^2}{\mu^2 (k+1)^2} (8k - 4k^*) \\
    &\leq&\bigg(\frac{k^*}{k+1}\bigg)^2 \bigg(1 - \frac{2}{k^*} \bigg)^{k^*} R_0^2 + \frac{192 k \sigma_*^2}{\mu^2 (k+1)^2}\\
    &\leq&\bigg(\frac{k^*}{k+1}\bigg)^2 \frac{R_0^2}{e^2} + \frac{192 \sigma_*^2}{\mu^2 (k+1)}.
\end{eqnarray*}
The last line follows from $\Big(1 - \frac{1}{x} \Big)^x \leq e^{-1}$ for all $x \geq 1$. This completes the proof.
\end{proof}

\subsection{Proof of Theorem \ref{Theorem: Total number of iteration knowledge}}
\begin{proof}
For $ 0 < \om_k \leq \big\{ \frac{1}{4L},\frac{\mu}{18 \delta} \big\}$ we obtain the following bound from Theorem \ref{Theorem: constant stepsize theorem}:
$$
R_k^2 \leq \bigg(1 - \frac{\mu \om_k}{2} \bigg) R_{k-1}^2 + 12 \om_k^2 \sigma_*^2.
$$
Then using Lemma \ref{Lemma: Stich lemma} with $a = \frac{\mu}{2}, h = \frac{1}{\Bar{\om}} $ and $c = 12 \sigma_*^2$ we complete the proof of this Theorem.
\end{proof}

\subsection{Proof of Theorem \ref{thm:weak_MVI}}
\vspace{5mm}
\begin{theorem}\label{thm:weak_MVI}
    Let $F$ be $L$-Lipschitz and satisfy Weak Minty condition with parameter $\rho < \nicefrac{1}{(2L)}$. Assume that inequality \eqref{eq: variance bound} holds (e.g., it holds whenever Assumption~\ref{as:expected_residual} holds, see Lemma~\ref{Lemma: variance bound}). Assume that $\gamma_k = \gamma$, $\omega_k = \omega$ and
    \begin{equation*}
        \max\left\{2\rho, \frac{1}{2L}\right\} < \gamma < \frac{1}{L},\quad 0 < \omega < \min\left\{\gamma - 2\rho, \frac{1}{4L} - \frac{\gamma}{4}\right\}, \quad \delta \leq \frac{(1-L\gamma)L^3\omega}{32}.
    \end{equation*}
    Then, for all $K \geq 2$ the iterates produced by \algname{SPEG} satisfy
    \begin{eqnarray}
        \min\limits_{0\leq k \leq K-1}\E\left[\|F(\hat x_k)\|^2\right] &\leq& \frac{(1 + 8\omega\gamma (\delta + L^2) - L\gamma)\left(1+\frac{48\omega\gamma \delta}{(1-L\gamma)^2}\right)^{K-1}\|x_0 - x^*\|^2}{\omega\gamma (1 - L(\gamma + 4\omega)) (K-1)}\notag \\
        &&\quad + \frac{8 \left(8 + \frac{(1-L\gamma)^2}{K-1}\left(1+\frac{48\omega\gamma \delta}{(1-L\gamma)^2}\right)^{K-1}\right) \sigma_*^2}{(1-L\gamma)^2(1-L(\gamma + 4\omega))}. \label{eq:SPEG_weak_MVI_result_appendix}
    \end{eqnarray}
\end{theorem}

\begin{proof}
    The proof closely follows the proof of Lemma C.3 and Theorem C.4 from \citep{gorbunov2022convergence}. The update rule of \algname{SPEG} implies for $k \geq 1$
    \begin{eqnarray*}
        \|x_{k+1} - x^*\|^2 &=& \|x_k - x^*\|^2 - 2\omega \langle x_k - x^*, F_{v_k}(\hx_k) \rangle + \omega^2 \|F_{v_k}(\hx_k)\|^2\\
        &=& \|x_k - x^*\|^2 - 2\omega \langle \hx_k - x^*, F_{v_k}(\hx_k) \rangle - 2\omega\gamma \langle F_{v_{k-1}}(\hx_{k-1}), F_{v_k}(\hx_k) \rangle \\
        && \quad + \omega^2 \|F_{v_k}(\hx_k)\|^2\\
        &=& \|x_k - x^*\|^2 - 2\omega \langle \hx_k - x^*, F_{v_k}(\hx_k) \rangle - \omega\gamma\|F_{v_{k-1}}(\hx_{k-1})\|^2\\
        &&\quad - \omega(\gamma - \omega)\|F_{v_k}(\hx_k)\|^2 + \omega \gamma \|F_{v_k}(\hx_k) - F_{v_{k-1}}(\hx_{k-1})\|^2,
    \end{eqnarray*}
    where in the last step we apply $2\langle a, b \rangle = \|a\|^2 + \|b\|^2 - \|a - b\|^2$, which holds for all $a,b \in \R^d$. Taking the full expectation and using $\E[\E_{v_k}[\cdot]] = \E[\cdot]$ and Weak Minty condition, we derive
    \begin{eqnarray}
        \E\left[\|x_{k+1} - x^*\|^2\right] &\leq& \E\left[\|x_k - x^*\|^2\right] - 2\omega \E\left[\langle \hx_k - x^*, F(\hx_k) \rangle\right] - \omega\gamma\E\left[\|F_{v_{k-1}}(\hx_{k-1})\|^2\right] \notag\\
        &&\quad - \omega(\gamma - \omega)\E\left[\|F_{v_k}(\hx_k)\|^2\right] + \omega \gamma \E\left[\|F_{v_k}(\hx_k) - F_{v_{k-1}}(\hx_{k-1})\|^2\right] \notag\\
        &\overset{\eqref{eq: weak MVI}}{\leq}& \E\left[\|x_k - x^*\|^2\right] + 2\omega\rho \E\left[\|F(\hx_k)\|^2\right] - \omega\gamma\E\left[\|F_{v_{k-1}}(\hx_{k-1})\|^2\right] \notag\\
        &&\quad - \omega(\gamma - \omega)\E\left[\|F_{v_k}(\hx_k)\|^2\right] + \omega \gamma \E\left[\|F_{v_k}(\hx_k) - F_{v_{k-1}}(\hx_{k-1})\|^2\right] \notag\\
        &\leq& \E\left[\|x_k - x^*\|^2\right] - \omega\gamma\E\left[\|F_{v_{k-1}}(\hx_{k-1})\|^2\right] \notag \\
        &&\quad - \omega(\gamma - 2\rho - \omega)\E\left[\|F_{v_k}(\hx_k)\|^2\right] \notag \\
        &&\quad + \omega \gamma \E\left[\|F_{v_k}(\hx_k) - F_{v_{k-1}}(\hx_{k-1})\|^2\right]\notag\\
        &\leq& \E\left[\|x_k - x^*\|^2\right] - \omega\gamma\E\left[\|F_{v_{k-1}}(\hx_{k-1})\|^2\right] \notag\\
        &&\quad  + \omega \gamma \E\left[\|F_{v_k}(\hx_k) - F_{v_{k-1}}(\hx_{k-1})\|^2\right], \label{eq:MVI_technical_ineq_1}
    \end{eqnarray}
    where we apply Jensen's inequality $\|F(\hx_k)\|^2 = \|\E_{v_k}[F_{v_k}(\hx_k)]\|^2 \leq \E_{v_k}[\|F_{v_k}(\hx_k)\|^2]$ and $\gamma > 2\rho + \omega$. For $k = 0$ we have $x_1 = x_0 - \omega F_{v_0}(\hx_0) = x_0 - \omega F_{v_0}(x_0)$ and
    \begin{eqnarray}
        \E\left[\|x_1 - x^*\|^2\right] &=& \|x_0 - x^*\|^2  - 2\omega \E\left[ \langle x_0 - x^*, F_{v_0}(x_0)  \rangle\right] + \omega^2 \E\left[\|F_{v_0}(x_0)\|^2\right] \notag\\
        &=& \|x_0 - x^*\|^2  - 2\omega  \langle x_0 - x^*, F(x_0)  \rangle+ \omega^2 \E\left[\|F_{v_0}(x_0)\|^2\right]. \notag
    \end{eqnarray}
    Applying Weak Minty condition, we get
    \begin{eqnarray}
        \E\left[\|x_1 - x^*\|^2\right] &=& \|x_0 - x^*\|^2  + 2\omega\rho  \|F(x_0)\|^2 + \omega^2 \E\left[\|F_{v_0}(x_0)\|^2\right] \notag\\
        &\leq& \|x_0 - x^*\|^2 + \omega(\omega + 2\rho)\E\left[\|F_{v_0}(x_0)\|^2\right]. \label{eq:MVI_technical_ineq_1_1}
    \end{eqnarray}
    
    \newpage 
    The next step of our proof is in estimating the last term from \eqref{eq:MVI_technical_ineq_1}. Using Young's inequality $\|a+b\|^2  \leq (1 + \alpha)\|a\|^2 + (1 + \alpha^{-1})\|b\|^2$, which holds for any $a,b \in \R^d$, $\alpha > 0$, we get for all $k \geq 2$
    \begin{eqnarray*}
        \E\left[\|F_{v_k}(\hx_k) - F_{v_{k-1}}(\hx_{k-1})\|^2\right] \hspace{-3mm} &\leq& \hspace{-3mm}(1+\alpha)\E\left[\|F(\hx_k) - F(\hx_{k-1})\|^2\right]\\
        && \hspace{-3mm}+ (1+\alpha^{-1})\E\left[\|F_{v_k}(\hx_k) - F(\hx_k) \right. \\
        && \hspace{-3mm}\quad \left. - (F_{v_{k-1}}(\hx_{k-1}) - F(\hx_{k-1}))\|^2\right]\\
        &\leq& \hspace{-3mm}(1+\alpha)L^2\E\left[\|\hx_k - \hx_{k-1}\|^2\right]\\
        &&\hspace{-3mm} \quad + 2(1+\alpha^{-1})\E\left[\|F_{v_k}(\hx_k) - F(\hx_k)\|^2 \right. \\
        && \hspace{-3mm} \quad \left. + \|F_{v_{k-1}}(\hx_{k-1}) - F(\hx_{k-1})\|^2\right]\\
        &\overset{\eqref{eq: variance bound}}{\leq}& \hspace{-3mm} (1+\alpha)L^2\E\left[\|\hx_k - x_k + x_k - x_{k-1} + x_{k-1} - \hx_{k-1}\|^2\right]\\
        &&\hspace{-3mm} \quad + 2(1+\alpha^{-1})\delta\E\left[\|\hx_k - x^*\|^2 + \|\hx_{k-1} - x^*\|^2\right] \\
        && \hspace{-3mm} \quad + 8(1+\alpha^{-1})\sigma_*^2\\
        &\leq& \hspace{-3mm}  (1+\alpha)L^2\E\left[\|(\gamma + \omega)F_{v_{k-1}}(\hx_{k-1}) - \gamma F_{v_{k-2}}(\hx_{k-2})\|^2\right]\\
        && \hspace{-3mm}\quad + 4(1+\alpha^{-1})\delta\E\left[\|x_k - x^*\|^2 + \|x_{k-1} - x^*\|^2\right]\\
        &&\hspace{-3mm} \quad + 4(1+\alpha^{-1})\delta\gamma^2\E\left[\|F_{v_{k-1}}(\hx_{k-1})\|^2 + \|F_{v_{k-2}}(\hx_{k-2})\|^2\right]\\
        &&\hspace{-3mm} \quad + 8(1+\alpha^{-1})\sigma_*^2\\
        &=&\hspace{-3mm} (1+\alpha)L^2(\gamma + \omega)^2\E\left[\|F_{v_{k-1}}(\hx_{k-1})\|^2\right] \\
        && \hspace{-3mm}\quad + (1+\alpha)L^2\gamma^2\E\left[\|F_{v_{k-2}}(\hx_{k-2})\|^2\right]\\
        &&\hspace{-3mm}\quad - 2(1+\alpha)L^2\gamma(\gamma+\omega)\E\left[\langle F_{v_{k-1}}(\hx_{k-1}), F_{v_{k-2}}(\hx_{k-2}) \rangle\right]\\
        &&\hspace{-3mm}\quad + 4(1+\alpha^{-1})\delta\E\left[\|x_k - x^*\|^2 + \|x_{k-1} - x^*\|^2\right]\\
        &&\hspace{-3mm}\quad + 4(1+\alpha^{-1})\delta\gamma^2\E\left[\|F_{v_{k-1}}(\hx_{k-1})\|^2 + \|F_{v_{k-2}}(\hx_{k-2})\|^2\right] \\
        &&\hspace{-3mm} \quad + 8(1+\alpha^{-1})\sigma_*^2\\
        &=&\hspace{-3mm} (1+\alpha)L^2(\gamma + \omega)^2\E\left[\|F_{v_{k-1}}(\hx_{k-1})\|^2\right] \\
        &&\hspace{-3mm} \quad + (1+\alpha)L^2\gamma^2\E\left[\|F_{v_{k-2}}(\hx_{k-2})\|^2\right]\\
        &&\hspace{-3mm}\quad - (1+\alpha)L^2\gamma(\gamma+\omega)\E\left[\| F_{v_{k-1}}(\hx_{k-1})\|^2 + \|F_{v_{k-2}}(\hx_{k-2})\|^2\right]\\
        &&\hspace{-3mm}\quad + (1+\alpha)L^2\gamma(\gamma+\omega)\E\left[\|F_{v_{k-1}}(\hx_{k-1}) - F_{v_{k-2}}(\hx_{k-2})\|^2\right]\\
        &&\hspace{-3mm}\quad + 4(1+\alpha^{-1})\delta\E\left[\|x_k - x^*\|^2 + \|x_{k-1} - x^*\|^2\right]\\
        &&\hspace{-3mm}\quad + 4(1+\alpha^{-1})\delta\gamma^2\E\left[\|F_{v_{k-1}}(\hx_{k-1})\|^2 + \|F_{v_{k-2}}(\hx_{k-2})\|^2\right] \\
        && \hspace{-3mm} \quad + 8(1+\alpha^{-1})\sigma_*^2\\
        &=& \hspace{-3mm} (1+\alpha)L^2\omega(\gamma + \omega)\E\left[\|F_{v_{k-1}}(\hx_{k-1})\|^2\right]\\
        && \hspace{-3mm}\quad - (1+\alpha)L^2\gamma\omega\E\left[\|F_{v_{k-2}}(\hx_{k-2})\|^2\right]\\
        && \hspace{-3mm}\quad + (1+\alpha)L^2\gamma(\gamma+\omega)\E\left[\|F_{v_{k-1}}(\hx_{k-1}) - F_{v_{k-2}}(\hx_{k-2})\|^2\right]\\
        && \hspace{-3mm}\quad + 4(1+\alpha^{-1})\delta\E\left[\|x_k - x^*\|^2 + \|x_{k-1} - x^*\|^2\right]\\
        && \hspace{-3mm}\quad + 4(1+\alpha^{-1})\delta\gamma^2\E\left[\|F_{v_{k-1}}(\hx_{k-1})\|^2 + \|F_{v_{k-2}}(\hx_{k-2})\|^2\right]\\
        && \hspace{-3mm} \quad + 8(1+\alpha^{-1})\sigma_*^2.
    \end{eqnarray*}
    Since $\hx_0 = x_0$ and $\hx_1 = x_1 - \gamma F_{v_0}(x_0) = x_0 - (\gamma + \omega)F_{v_0}(x_0)$, for $k = 1$ we have
    \begin{eqnarray*}
        \E\left[\|F_{v_1}(\hx_1) - F_{v_{0}}(\hx_{0})\|^2\right] &=& \E\left[\|F_{v_1}(\hx_1) - F_{v_{0}}(x_{0})\|^2\right]\\
        &\leq& (1+\alpha)\E\left[\|F(\hx_1) - F(x_{0})\|^2\right]\\
        &&\quad + (1+\alpha^{-1})\E\left[\|F_{v_1}(\hx_1) - F(\hx_1) - (F_{v_{0}}(x_{0}) - F(x_{0}))\|^2\right]\\
        &\leq& (1+\alpha)L^2\E\left[\|\hx_1 - x_0\|^2\right]\\
        &&\quad + 2(1+\alpha^{-1})\E\left[\|F_{v_1}(\hx_1) - F(\hx_1)\|^2 + \|F_{v_{0}}(x_{0}) - F(x_{0})\|^2\right]
    \end{eqnarray*}
    Then using \eqref{eq: variance bound} we get,
    \begin{eqnarray*}
        \E\left[\|F_{v_1}(\hx_1) - F_{v_{0}}(\hx_{0})\|^2\right]
        &\overset{\eqref{eq: variance bound}}{\leq}& (1+\alpha)L^2(\gamma+\omega)^2\E\left[\|F_{v_0}(x_0)\|^2\right]\\
        &&\quad + 2(1+\alpha^{-1})\delta \E\left[\|\hx_1 - x^*\|^2 + \|x_0 - x^*\|^2\right] + 8(1+\alpha)\sigma_*^2\\
        &\leq& \left((1+\alpha)L^2 + 4(1+\alpha^{-1})\delta\right)(\gamma+\omega)^2\E\left[\|F_{v_0}(x_0)\|^2\right]\\
        &&\quad + 6(1+\alpha^{-1})\delta \|x_0 - x^*\|^2 + 8(1+\alpha)\sigma_*^2.
    \end{eqnarray*}
    Let $\{w_k\}_{k=0}^{K-1}$ be a non-increasing sequence of positive numbers that will be specified later and $W_K = \sum_{k=0}^{K-1} w_k$. Summing up the above two inequalities with weights $\{w_k\}_{k=1}^{K-1}$, we derive
    \begin{eqnarray*}
        \sum\limits_{k=1}^{K-1}w_k \E\left[\|F_{v_k}(\hx_k) - F_{v_{k-1}}(\hx_{k-1})\|^2\right] \leq && \hspace{-6.5mm} (1+\alpha)L^2\sum\limits_{k=1}^{K-3}\left(\omega(\gamma + \omega)w_{k+1} \E\left[\|F_{v_{k}}(\hx_{k})\|^2\right] \right. \\
        && \hspace{-6.5mm} \left. -\gamma\omega w_{k+2}\E\left[\|F_{v_{k}}(\hx_{k})\|^2\right]\right)\\
        && \hspace{-6.5mm} + (1+\alpha)L^2\omega(\gamma + \omega)w_{K-1}\E\left[\|F_{v_{K-2}}(\hx_{K-2})\|^2\right]\\
        && \hspace{-6.5mm} - (1+\alpha)L^2 \gamma\omega w_2 \E\left[\|F_{v_{0}}(x_{0})\|^2\right]\\
        && \hspace{-6.5mm} + (1+\alpha)L^2\gamma(\gamma+\omega)\sum\limits_{k=1}^{K-2}w_{k+1}\E\left[\|F_{v_k}(\hx_k) \right. \\
        && \hspace{-6.5mm} \left. - F_{v_{k-1}}(\hx_{k-1})\|^2\right] + 4(1+\alpha^{-1})\delta\sum\limits_{k=2}^{K-1}w_k\E\left[\|x_k - x^*\|^2 \right] \\
        && \hspace{-6.5mm} + w_k\E \left[ \|x_{k-1} - x^*\|^2\right] \\
        && \hspace{-6.5mm} + 4(1+\alpha^{-1})\delta\gamma^2\sum\limits_{k=1}^{K-2}w_{k+1}\E\left[\|F_{v_{k}}(\hx_{k})\|^2 \right. \\
        && \hspace{-6.5mm} \left. + \|F_{v_{k-1}}(\hx_{k-1})\|^2\right] + 8(1+\alpha^{-1})(W_K - w_0 - w_1)\sigma_*^2\\
        && \hspace{-6.5mm} + \left((1+\alpha)L^2 + 4(1+\alpha^{-1})\delta\right)(\gamma+\omega)^2w_1\E\left[\|F_{v_0}(x_0)\|^2\right]\\
        && \hspace{-6.5mm} + 6(1+\alpha^{-1})\delta w_1 \|x_0 - x^*\|^2 + 8(1+\alpha)w_1\sigma_*^2.
    \end{eqnarray*}
    Next, we rearrange the terms using $w_k \geq w_{k+1}$ and new notation $\Delta_k = \E\left[\|F_{v_k}(\hx_k) - F_{v_{k-1}}(\hx_{k-1})\|^2\right]$:
    \begin{eqnarray*}
        \left(1 - (1+\alpha)L^2\gamma(\gamma+\omega)\right)\sum\limits_{k=1}^{K-1}w_k\Delta_k \hspace{-3.5mm}&\leq& \hspace{-3.5mm} \sum\limits_{k=1}^{K-2}(1+\alpha)L^2\omega(\gamma+\omega) w_k\E\left[\|F_{v_{k}}(\hx_{k})\|^2\right]\\
        && \hspace{-3mm} + 8(1+\alpha^{-1})\delta\gamma^2 w_k\E\left[\|F_{v_{k}}(\hx_{k})\|^2\right]\\
        && \hspace{-3mm}+ \left((1+\alpha)L^2 + 8(1+\alpha^{-1})\delta\right)(\gamma+\omega)^2w_0\E\left[\|F_{v_0}(x_0)\|^2\right]\\
        && \hspace{-3mm}+ 12(1+\alpha^{-1})\delta\sum\limits_{k=1}^{K-1}w_k\E\left[\|x_k - x^*\|^2\right] \\
        && \hspace{-3mm} + 8(1+\alpha^{-1})(W_K - w_0)\sigma_*^2.
    \end{eqnarray*}
    To simplify the above inequality we choose $\alpha = \frac{1}{2L^2\gamma(\gamma+\omega)} - \frac{1}{2}$, which is positive due to $\gamma < \nicefrac{1}{L}$ and $\gamma+\omega < \nicefrac{1}{L}$. In this case, we have
    \begin{eqnarray*}
        (1+\alpha)L^2\gamma(\gamma + \omega) &=& \frac{1}{2}L^2\gamma(\gamma + \omega) + \frac{1}{2},\\
        (1+\alpha)L^2(\gamma + \omega)^2 &=& \frac{1}{2}L^2(\gamma + \omega)^2 + \frac{\gamma+\omega}{2\gamma} \leq \frac{3}{2},\\
        (1+\alpha)L^2\omega (\gamma + \omega) &=& \frac{1}{2}L^2\omega(\gamma + \omega) + \frac{\omega}{2\gamma} = \frac{L\omega}{2}\left(L(\gamma + \omega) + \frac{1}{\gamma L}\right) \leq \frac{3 L\omega}{2},\\
        1 + \alpha^{-1} &=& 1 + \frac{2L^2\gamma (\gamma + \omega)}{1 - L^2\gamma (\gamma + \omega)} = \frac{1 + L^2\gamma (\gamma + \omega)}{1 - L^2\gamma (\gamma + \omega)} \leq \frac{2}{1 - L^2\gamma (\gamma + \omega)},
    \end{eqnarray*}
    where we also use $\nicefrac{1}{2L} < \gamma < \nicefrac{1}{L}$ and $\gamma + \omega < \nicefrac{1}{L}$. Using these relations, we can continue our derivation as follows:
    \begin{eqnarray*}
        \frac{1}{2}\left(1 - L^2\gamma(\gamma+\omega)\right)\sum\limits_{k=1}^{K-1}w_k\Delta_k &\leq& \sum\limits_{k=1}^{K-2}\left(\frac{3L\omega}{2} + \frac{16}{1 - L^2\gamma (\gamma + \omega)}\delta\gamma^2\right) w_k\E\left[\|F_{v_{k}}(\hx_{k})\|^2\right]\\
        &&\quad + \left(\frac{3}{2} + \frac{16}{1 - L^2\gamma (\gamma + \omega)}\delta(\gamma+\omega)^2\right)w_0\E\left[\|F_{v_0}(x_0)\|^2\right]\\
        &&\quad + \frac{24}{1 - L^2\gamma (\gamma + \omega)}\delta\sum\limits_{k=1}^{K-1}w_k\E\left[\|x_k - x^*\|^2\right]\\
        && \quad + \frac{16}{1 - L^2\gamma (\gamma + \omega)}(W_K - w_0)\sigma_*^2.
    \end{eqnarray*}
    Dividing both sides by $\frac{1}{2}\left(1 - L^2\gamma(\gamma+\omega)\right)$, we derive
    \begin{eqnarray}
        \sum\limits_{k=1}^{K-1}w_k\Delta_k &\leq& \sum\limits_{k=1}^{K-2}\left(\frac{3L\omega}{1 - L^2\gamma (\gamma + \omega)} + \frac{32}{(1 - L^2\gamma (\gamma + \omega))^2}\delta\gamma^2\right) w_k\E\left[\|F_{v_{k}}(\hx_{k})\|^2\right]\notag\\
        &&\quad + \left(\frac{3}{1 - L^2\gamma (\gamma + \omega)} + \frac{32}{(1 - L^2\gamma (\gamma + \omega))^2}\delta(\gamma+\omega)^2\right)w_0\E\left[\|F_{v_0}(x_0)\|^2\right]\notag\\
        &&\quad + \frac{48}{(1 - L^2\gamma (\gamma + \omega))^2}\delta\sum\limits_{k=1}^{K-1}w_k\E\left[\|x_k - x^*\|^2\right] \notag\\
        && \quad + \frac{32}{(1 - L^2\gamma (\gamma + \omega))^2}(W_K - w_0)\sigma_*^2 \notag\\
        &=& \sum\limits_{k=1}^{K-2}C_1 w_k\E\left[\|F_{v_{k}}(\hx_{k})\|^2\right] + C_2 w_0\E\left[\|F_{v_0}(x_0)\|^2\right]\notag\\
        &&\quad + 3C_3\delta\sum\limits_{k=1}^{K-1}w_k\E\left[\|x_k - x^*\|^2\right] + 2C_3W_K\sigma_*^2, \label{eq:MVI_technical_ineq_2}
    \end{eqnarray}
    where $C_1 = \frac{3L\omega}{1 - L^2\gamma (\gamma + \omega)} + \frac{32}{(1 - L^2\gamma (\gamma + \omega))^2}\delta\gamma^2$, $C_2 = \frac{3}{1 - L^2\gamma (\gamma + \omega)} + \frac{32}{(1 - L^2\gamma (\gamma + \omega))^2}\delta(\gamma+\omega)^2$, and $C_3 = \frac{16}{(1 - L^2\gamma (\gamma + \omega))^2}$. Summing up inequalities \eqref{eq:MVI_technical_ineq_1} for $k = 1,\ldots, K-1$ with weights $w_1,\ldots, w_{K-1}$ and \eqref{eq:MVI_technical_ineq_1_1} with weight $w_0$, we get
    \begin{eqnarray*}
        \sum\limits_{k=0}^{K-1}w_k\E\left[\|x_{k+1} - x^*\|^2\right] &\leq& \sum\limits_{k=0}^{K-1}w_k\E\left[\|x_{k} - x^*\|^2\right] - \omega\gamma \sum\limits_{k=1}^{K-1}w_k \E\left[\|F_{v_{k-1}}(\hx_{k-1})\|^2\right]\\
        &&\quad  + \omega\gamma \sum\limits_{k=1}^{K-1}w_k\Delta_k  + \omega (\omega + 2\rho) w_0\E\left[\|F_{v_0}(x_0)\|^2\right].
    \end{eqnarray*}
    Since $w_k \geq w_{k+1}$, we can continue the derivation as follows:
    \begin{eqnarray*}
        \sum\limits_{k=0}^{K-1}w_k\E\left[\|x_{k+1} - x^*\|^2\right] &\leq& \sum\limits_{k=0}^{K-1}w_k\E\left[\|x_{k} - x^*\|^2\right] - \omega\gamma \sum\limits_{k=0}^{K-2}w_{k} \E\left[\|F_{v_{k}}(\hx_{k})\|^2\right]\\
        &&\quad  + \omega\gamma \sum\limits_{k=1}^{K-1}w_k\Delta_k  + \omega (\omega + 2\rho) w_0\E\left[\|F_{v_0}(x_0)\|^2\right]\\
        &\overset{\eqref{eq:MVI_technical_ineq_2}}{\leq}& \sum\limits_{k=0}^{K-1}(1+ 3C_3\omega\gamma \delta)w_k\E\left[\|x_{k} - x^*\|^2\right]\\
        && \quad - \omega\gamma (1 - C_1)\sum\limits_{k=0}^{K-2}w_{k} \E\left[\|F_{v_{k}}(\hx_{k})\|^2\right]\\
        &&\quad + 2\omega\gamma C_2 w_0 \E\left[\|F_{v_{0}}(\hx_{0})\|^2\right] + 2\omega\gamma C_3 W_K \sigma_*^2.
    \end{eqnarray*}
    Now we need to specify the weights $w_{-1}, w_0, w_1, \ldots, w_{K-1}$. Let $w_{K-2} = 1$ and $w_{k-1} = (1+ 3C_3\omega\gamma \delta)w_k$. Then, rearranging the terms, dividing both sides by $\omega\gamma (1 - C_1)W_{K-1}$, we get
    \begin{eqnarray*}
        \min\limits_{0\leq k \leq K-1}\E\left[\|F(\hat x_k)\|^2\right] &\leq& \min\limits_{0\leq k \leq K-1}\E\left[\|F_{v_k}(\hat x_k)\|^2\right]\\
        &\leq& \sum\limits_{k=0}^{K-2}\frac{w_k}{W_{K-1}}\E\left[\|F_{v_{k}}(\hx_{k})\|^2\right]\\
        &\leq& \frac{1}{\omega\gamma (1 - C_1) W_{K-1}}\sum\limits_{k=0}^{K-1}\left(w_{k-1}\E\left[\|x_{k} - x^*\|^2\right] \right.\\
        && \quad \left. - w_k\E\left[\|x_{k+1} - x^*\|^2\right]\right) + \frac{2C_2 w_0 \E\left[\|F_{v_{0}}(\hx_{0})\|^2\right]}{(1-C_1)W_{K-1}} \\
        && \quad + \frac{2C_3 W_K \sigma_*^2}{(1-C_1)W_{K-1}}\\
        &\leq&  \frac{w_{-1}\|x_0 - x^*\|^2}{\omega\gamma (1 - C_1) W_{K-1}} + \frac{2C_2 w_0 \E\left[\|F_{v_{0}}(\hx_{0})\|^2\right]}{(1-C_1)W_{K-1}} + \frac{2C_3 W_K \sigma_*^2}{(1-C_1)W_{K-1}}.
    \end{eqnarray*}
    It remains to simplify the right-hand side of the above inequality. First, we notice that $W_{K-1} = \sum_{k=0}^{K-2} w_k \geq (K-1)w_{K-2} = K-1$ since $w_{k} \geq w_{k+1}$. Moreover, $w_{-1} = (1 + 3C_3 \omega\gamma\delta)^{K-1}$. Next,
    \begin{eqnarray*}
        C_1 &=& \frac{3L\omega}{1 - L^2\gamma (\gamma + \omega)} + \frac{32}{(1 - L^2\gamma (\gamma + \omega))^2}\delta\gamma^2\\
        &\leq& \frac{3L\omega}{1 - L\gamma} + \frac{32}{(1 - L\gamma)^2} \cdot \frac{(1-L\gamma)L^3\omega}{32} \cdot \gamma^2 \leq \frac{4L\omega}{1 - L\gamma},\\
        C_2 &=& \frac{3}{1 - L^2\gamma (\gamma + \omega)} + \frac{32}{(1 - L^2\gamma (\gamma + \omega))^2}\delta(\gamma+\omega)^2\\
        &\leq& \frac{3}{1 - L\gamma} + \frac{32}{(1 - L\gamma)^2}\cdot \frac{(1-L\gamma)L^3\omega}{32} \cdot (\gamma+\omega)^2 \leq \frac{4}{1 - L\gamma},\\
        C_3 &=& \frac{16}{(1 - L^2\gamma (\gamma + \omega))^2} \leq \frac{16}{(1 - L\gamma)^2},
    \end{eqnarray*}
    where we use $\delta \leq \nicefrac{(1-L\gamma)L^3\omega}{16}$ and $\gamma + \omega < \nicefrac{1}{L}$. Using these inequalities, we simplify the bound as follows:
    \begin{eqnarray}
        \min\limits_{0\leq k \leq K-1}\E\left[\|F(\hat x_k)\|^2\right] &\leq& \frac{(1 - L\gamma)(1+3C_3\omega\gamma \delta)^{K-1}\|x_0 - x^*\|^2}{\omega\gamma (1 - L(\gamma + 4\omega)) (K-1)} \notag \\
        && \quad + \frac{8 (1+3C_3\omega\gamma \delta)^{K-2} \E\left[\|F_{v_{0}}(\hx_{0})\|^2\right]}{(1 - L(\gamma + 4\omega))(K-1)} \notag \\
        &&\quad + \frac{32 \sigma_*^2}{(1-L\gamma)(1-L(\gamma + 4\omega))} \notag\\
        &\leq& \frac{(1 - L\gamma)\left(1+\frac{48\omega\gamma \delta}{(1-L\gamma)^2}\right)^{K-1}\|x_0 - x^*\|^2}{\omega\gamma (1 - L(\gamma + 4\omega)) (K-1)} \notag \\
        && \quad + \frac{8 \left(1+\frac{48\omega\gamma \delta}{(1-L\gamma)^2}\right)^{K-2} \E\left[\|F_{v_{0}}(\hx_{0})\|^2\right]}{(1 - L(\gamma + 4\omega))(K-1)} \notag\\
        &&\quad + \frac{32 \sigma_*^2}{(1-L\gamma)(1-L(\gamma + 4\omega))} \label{eq:MVI_technical_ineq_3}
    \end{eqnarray}
    where we use $W_{K} = W_{K-1} + w_{K-1} \leq W_{K-1} + w_{K-2} \leq 2W_{K-1}$. Finally, we use \eqref{eq: variance bound} to upper-bound $\E\left[\|F_{v_{0}}(\hx_{0})\|^2\right]$:
    \begin{eqnarray*}
        \E\left[\|F_{v_{0}}(\hx_{0})\|^2\right] &=& \E\left[\|F_{v_{0}}(x_{0})\|^2\right] \overset{\eqref{eq: variance bound}}{\leq} \delta \|x_0 - x^*\|^2 + \|F(x_0)\|^2 + 2\sigma_*^2\\
        &\leq& (\delta + L^2)\|x_0 - x^*\|^2 + 2\sigma_*^2.
    \end{eqnarray*}
    Plugging this inequality in \eqref{eq:MVI_technical_ineq_3}, we derive
    \begin{eqnarray*}
        \min\limits_{0\leq k \leq K-1}\E\left[\|F(\hat x_k)\|^2\right] &\leq& \frac{(1 + 8\omega\gamma (\delta + L^2) - L\gamma)\left(1+\frac{48\omega\gamma \delta}{(1-L\gamma)^2}\right)^{K-1}\|x_0 - x^*\|^2}{\omega\gamma (1 - L(\gamma + 4\omega)) (K-1)}\\
        &&\quad + \frac{4 \left(8 + \frac{1-L\gamma}{K-1}\left(1+\frac{48\omega\gamma \delta}{(1-L\gamma)^2}\right)^{K-1}\right)\sigma_*^2}{(1-L\gamma)(1-L(\gamma + 4\omega))},
    \end{eqnarray*}
    which concludes the proof.
\end{proof}

\subsection{Proof of Theorem \ref{cor:weak_MVI_convergence}}
\label{AppendixE5}

\begin{theorem}\label{cor:weak_MVI_convergence_appendix}
    Let $F$ be $L$-Lipschitz and satisfy Weak Minty condition with parameter $\rho < \nicefrac{1}{(2L)}$. Assume that inequality \eqref{eq: variance bound} holds (e.g., it holds whenever Assumption~\ref{as:expected_residual} holds, see Lemma~\ref{Lemma: variance bound}). Assume that $\gamma_k = \gamma$, $\omega_k = \omega$ and
    \begin{equation*}
        \max\left\{2\rho, \frac{1}{2L}\right\} < \gamma < \frac{1}{L},\quad 0 < \omega < \min\left\{\gamma - 2\rho, \frac{1}{4L} - \frac{\gamma}{4}\right\}.
    \end{equation*}
    Then, for all $K \geq 2$ the iterates produced by mini-batched \algname{SPEG} with batch-size 
    \begin{equation}
        \tau \geq \max\left\{1, \frac{32\delta}{(1-L\gamma)L^3\omega}, \frac{48\omega\gamma \delta(K-1)}{(1 - L\gamma)^2}, \frac{2\omega\gamma\sigma_*^2(K-1)}{(1-L\gamma)\|x_0 - x^*\|^2}\right\} \label{eq:SPEG_weak_MVI_batchsize_appendix}
    \end{equation}
    satisfy
    \begin{eqnarray}
        \min\limits_{0\leq k \leq K-1}\E\left[\|F(\hat x_k)\|^2\right] \leq \frac{48\|x_0 - x^*\|^2}{\omega\gamma (1 - L(\gamma + 4\omega)) (K-1)}. \label{eq:SPEG_weak_MVI_result_corollary_appendix}
    \end{eqnarray}
\end{theorem}
\begin{proof}
    Mini-batched \algname{SPEG} uses estimator
    \begin{eqnarray*}
        F_{v_k}(\hx_k) = \frac{1}{\tau}\sum\limits_{i=1}^{\tau} F_{v_{k,i}}(\hx_k),
    \end{eqnarray*}
    where $F_{v_{k,1}}(\hx_k),\ldots, F_{v_{k,\tau}}(\hx_k)$ are independent samples satisfying \eqref{eq: variance bound} with parameters $\delta$ and $\sigma_*^2$. Using variance decomposition and independence of $F_{v_{k,1}}(\hx_k),\ldots, F_{v_{k,\tau}}(\hx_k)$, we get
    \begin{eqnarray*}
        \E_{v_k}\left[\|F_{v_k}(\hx_k)\|^2\right] &=& \E_{v_k}\left[\|F_{v_k}(\hx_k) - F(\hx_k)\|^2\right] + \|F(\hx_k)\|^2\\
        &=& \E_{v_k}\left[\left\|\frac{1}{\tau}\sum\limits_{i=1}^b (F_{v_{k,i}}(\hx_k) - F(\hx_k))\right\|^2\right] + \|F(\hx_k)\|^2\\
        &=& \frac{1}{\tau^2}\sum\limits_{i=1}^{\tau}\E_{v_{k}}\left[\|F_{v_{k,i}}(\hx_k) - F(\hx_k)\|^2\right] + \|F(\hx_k)\|^2\\
        &\overset{\eqref{eq: variance bound}}{\leq}& \frac{\delta}{\tau}\|\hx_k - x^*\|^2 + \|F(\hx_k)\|^2 + \frac{2\sigma_*^2}{\tau}. 
    \end{eqnarray*}
    That is, mini-batched estimator $F_{v_k}(\hx_k)$ satisfies \eqref{eq: variance bound} with parameters $\nicefrac{\delta}{\tau}$ and $\nicefrac{\sigma_*^2}{\tau}$. Therefore, Theorem~\ref{thm:weak_MVI} implies
    \begin{eqnarray}
        \min\limits_{0\leq k \leq K-1}\E\left[\|F(\hat x_k)\|^2\right] &\leq& \frac{(1 + 4\omega\gamma \left(\frac{\delta}{\tau} + L^2\right) - L\gamma)\left(1+\frac{48\omega\gamma \delta}{(1-L\gamma)^2\tau}\right)^{K-1}\|x_0 - x^*\|^2}{\omega\gamma (1 - L(\gamma + 4\omega)) (K-1)} \notag\\
        &&\quad + \frac{8 \left(8 + \frac{1-L\gamma}{K-1}\left(1+\frac{48\omega\gamma \delta}{(1-L\gamma)^2\tau}\right)^{K-1}\right) \sigma_*^2}{(1-L\gamma)(1-L(\gamma + 4\omega))\tau}. \label{eq:MVI_technical_ineq_4}
    \end{eqnarray}
    Since $\tau$ satisfies \eqref{eq:SPEG_weak_MVI_batchsize_appendix} and $\gamma \leq \nicefrac{1}{L}$, $\omega \leq \nicefrac{1}{4L}$, we have
    \begin{eqnarray*}
        4\omega\gamma\left(\frac{\delta}{\tau} + L^2\right) &\leq& \frac{1}{4L^2}\left(\delta \cdot \frac{(1-L\gamma)L^3\omega}{16\delta} + L^2\right) \leq 1,\\
        \left(1+\frac{48\omega\gamma \delta}{(1-L\gamma)^2\tau}\right)^{K-1} &\leq& \left(1+\frac{48\omega\gamma \delta}{(1-L\gamma)^2} \cdot \frac{(1 - L\gamma)^2}{48\omega\gamma \delta(K-1)}\right)^{K-1} \\
        && = \left(1 + \frac{1}{K-1}\right)^{K-1} \leq \exp(1) < 3.
    \end{eqnarray*}
    Using this, we can simplify \eqref{eq:MVI_technical_ineq_4} as follows:
    \begin{eqnarray*}
        \min\limits_{0\leq k \leq K-1}\E\left[\|F(\hat x_k)\|^2\right] &\leq& \frac{6\|x_0 - x^*\|^2}{\omega\gamma (1 - L(\gamma + 4\omega)) (K-1)} + \frac{88 \sigma_*^2}{(1-L\gamma)(1-L(\gamma + 4\omega))\tau}\\
        &\overset{\eqref{eq:SPEG_weak_MVI_batchsize_appendix}}{\leq}& \frac{6\|x_0 - x^*\|^2}{\omega\gamma (1 - L(\gamma + 4\omega)) (K-1)} \\
        && \quad + \frac{88 \sigma_*^2}{(1-L\gamma)(1-L(\gamma + 4\omega))} \cdot \frac{(1-L\gamma)\|x_0 - x^*\|^2}{2\omega\gamma\sigma_*^2}\\
        &=& \frac{48\|x_0 - x^*\|^2}{\omega\gamma (1 - L(\gamma + 4\omega)) (K-1)}.
    \end{eqnarray*}
    This concludes the proof.
\end{proof}

\paragraph{On Oracle Complexity of Theorem \ref{cor:weak_MVI_convergence}.} Let us now express the result of Theorem~\ref{cor:weak_MVI_convergence} via oracle complexity.

Oracle complexity captures the computational requirements required to solve a specific optimization problem. That is, given a prespecified accuracy $\varepsilon > 0$, it measures the number of oracle calls needed to solve the problem to this $\varepsilon$ accuracy. In our setting,  an oracle call indicates the computation of one operator, $F_i$ (for some $i \in [n]$). Therefore, in Theorem \ref{cor:weak_MVI_convergence}, where a mini-batch of size $\tau$ is required in each iteration of the update rule, we have $\tau$ many oracle calls per iteration. In that scenario, the total number of oracle calls required to obtain specific accuracy $\varepsilon > 0$ is given by $K \tau$ (multiplication of $K$ iterations with $\tau$ oracle calls).

Note that according to Theorem \ref{cor:weak_MVI_convergence} to achieve an $\varepsilon$ accuracy, we need $K \geq \frac{C \left\| x_0 - x^* \right\|^2}{\epsilon}$ iterations. This follows trivially by
\begin{equation}
\label{naosdnao}
    \min\limits_{0\leq k \leq K-1}\E\left[\|F(\hat x_k)\|^2\right] \overset{\text{Theorem}~\ref{cor:weak_MVI_convergence}}{\leq} \frac{C\|x_0 - x^*\|^2}{K-1} \leq \varepsilon.
\end{equation}
 Therefore, using $K \geq \frac{C \left\| x_0 - x^* \right\|^2}{\epsilon}$ in combination with the lower bound on $\tau$ from \eqref{eq:SPEG_weak_MVI_batchsize}, the total number of oracle calls to satisfy \eqref{naosdnao} is given by:  
\begin{eqnarray*}
   K\tau \geq \max \left\lbrace \frac{C|| x_0 - x^{\ast} ||^2}{\epsilon}, \frac{32 C \delta || x_0 - x^{\ast} ||^2}{(1 - L \gamma) L^3 \omega \epsilon}, \frac{48 C^2 \omega \gamma \delta || x_0 - x^{\ast}||^4}{(1 - \gamma L)^2 \epsilon^2}, \frac{2 C^2 \omega \gamma \sigma_\ast^2 || x_0 - x^{\ast}||^2}{(1 - L \gamma) \epsilon^2}\right\rbrace. 
\end{eqnarray*}

\newpage
\section{Further Results on Arbitrary Sampling}\label{section: further results on arbitrary sampling}
\subsection{Proof of Proposition \ref{Prop_SingleElement}}
Expanding the left hand side of Expected Residual~\eqref{eq: ER Condition} condition we have
\begin{eqnarray}
    \E \|(F_v(x) - F_v(x^*)) - (F(x) - F(x^*))\|^2 
   &\overset{\eqref{eq: variance of an unbiased estimator}}{=} & \E\|(F_v(x) - F_v(x^*)) \|^2 - \|F(x) - F(x^*)\|^2 \notag \\
   &\leq& \E \|F_v(x) - F_v(x^*)\|^2 . \label{eq: bound on expected residual}
\end{eqnarray}
For any $x$ and $y$ with $v_i = \frac{1}{p_i}$ we obtain  
\begin{eqnarray*}
\|F_v(x) - F_v(y)\|^2 &=& \frac{1}{n^2} \bigg\| \sum_{i \in S} \frac{1}{p_i} (F_i(x) - F_i(y))\bigg\|^2 \\
&=& \sum_{i,j \in S} \bigg \langle \frac{1}{n p_i}(F_i(x) - F_i(y)), \frac{1}{n p_j}(F_j(x) - F_j(y))\bigg \rangle.
\end{eqnarray*}
Then taking expectation on both sides we get
\begin{eqnarray*}
    \E \|F_v(x) - F_v(y)\|^2 &=& \sum_{C} p_C \sum_{i,j \in C} \bigg\langle \frac{1}{n p_i}(F_i(x) - F_i(y)), \frac{1}{n p_j}(F_j(x) - F_j(y)) \bigg\rangle \\
    &=& \sum_{i,j = 1}^n \sum_{C: i,j \in C} p_C \bigg \langle \frac{1}{n p_i}(F_i(x) - F_i(y)), \frac{1}{n p_j}(F_j(x) - F_j(y)) \bigg \rangle \\
    &=& \sum_{i,j = 1}^n \frac{P_{ij}}{p_i p_j} \bigg \langle \frac{1}{n}(F_i(x) - F_i(y)), \frac{1}{n}(F_j(x) - F_j(y)) \bigg \rangle.
\end{eqnarray*}
Now we consider the case, where the ratio $\frac{P_{ij}}{p_i p_j} = c_2$ i.e. constant for $i \neq j$ and $P_{ii} = p_i$. Then from the above computations we derive
\begin{eqnarray*}
    \E \|F_v(x) - F_v(y)\|^2 &=& \sum_{i \neq j}^n c_2 \bigg \langle \frac{1}{n} (F_i(x) - F_i(y)), \frac{1}{n} (F_i(x) - F_i(y)) \bigg \rangle \\
    && \quad + \sum_{i = 1}^n \frac{1}{n^2 p_i} \|F_i(x) - F_i(y)\|^2 \\
    &=& \sum_{i,j = 1}^n c_2 \bigg \langle \frac{1}{n} (F_i(x) - F_i(y)), \frac{1}{n} (F_i(x) - F_i(y)) \bigg \rangle\\
    && \quad + \sum_{i = 1}^n \frac{1 - p_i c_2}{n^2 p_i} \|F_i(x) - F_i(y)\|^2\\
    &\overset{\eqref{eq: F_i lipschitz}}{\leq}&  c_2 \|F(x) - F(y)\|^2 + \sum_{i = 1}^n \frac{1 - p_i c_2}{n^2 p_i} L_i^2\|x - y\|^2 \\
    &\overset{\eqref{eq: F lipschitz}}{\leq}& \bigg(c_2 L^2 + \frac{1}{n^2} \sum_{i = 1}^n \frac{1 - p_i c_2}{p_i} L_i^2 \bigg) \|x - y\|^2.
\end{eqnarray*}
Thus replacing $y = x^*$ and combining with \eqref{eq: bound on expected residual} we get the following bound on the Expected Residual:
\begin{equation}\label{eq: single element sampling bound on ER 2}
    \E \|(F_v(x) - F_v(x^*)) - (F(x) - F(x^*))\|^2 \leq \bigg(c_2 L^2 + \frac{1}{n^2} \sum_{i = 1}^n  \frac{1 - p_i c_2}{p_i} L_i^2 \bigg) \|x - x^*\|^2.
\end{equation}
For single-element sampling $c_2 = 0$ (as probability of two points appearing in same sample is zero for single element sampling i.e. $P_{ij} = 0$). Then we obtain
$$
\delta \leq  \frac{2}{n^2} \sum_{i = 1}^n \frac{L_i^2}{p_i}
$$
from \eqref{eq: single element sampling bound on ER 2}. This completes the derivation of $\delta$ for single element sampling. To compute $\sigma_*^2$ for single element sampling, we replace
\begin{equation*}
    P_{ij} = \begin{cases}
    p_i & \text{if } i = j \\
    0 & \text{otherwise}
    \end{cases}
\end{equation*}
in \eqref{eq: expansion for sigma} to get
\begin{equation*}
    \sigma_*^2 = \frac{1}{n^2} \sum_{i=1}^n \frac{1}{p_i} \|F_i(x^*)\|^2.
\end{equation*}

\newpage
\section{Numerical Experiments}
\label{Appendix_AddExperiments}
In Appendix \ref{subsec:FurtherDetails}, we add more details on the experiments discussed in the main paper. Furthermore, in Appendix \ref{subsec:ComparisonConstantvsSwitching}, we run more experiments to evaluate the performance of \algname{SPEG} on quasi-strongly monotone and weak MVI problems.

\subsection{More Details on the Numerical Experiments of Section~\ref{Numerical Experiments}}\label{subsec:FurtherDetails}

\paragraph{On Constant vs Switching Stepsize Rule.} We run the experiments on two synthetic datasets. In Fig.~\ref{fig: Synthetic Dataset 1} of the main paper, we take $\mu_A = \mu_C = 0.6$. Here we include one more plot with a similar flavor but in a different setting. For Fig.~\ref{fig: Synthetic Dataset 2}, we generate the data such that eigenvalues of $A_1, B_1, C_1$ are generated uniformly from the interval $[0.1, 10]$. In the new plot, similar to the main paper, we can see the benefit of switching the step-size rule of Theorem \ref{SPEG switching rule}.

\begin{figure}[h]
\centering
    \includegraphics[width=.48\textwidth]{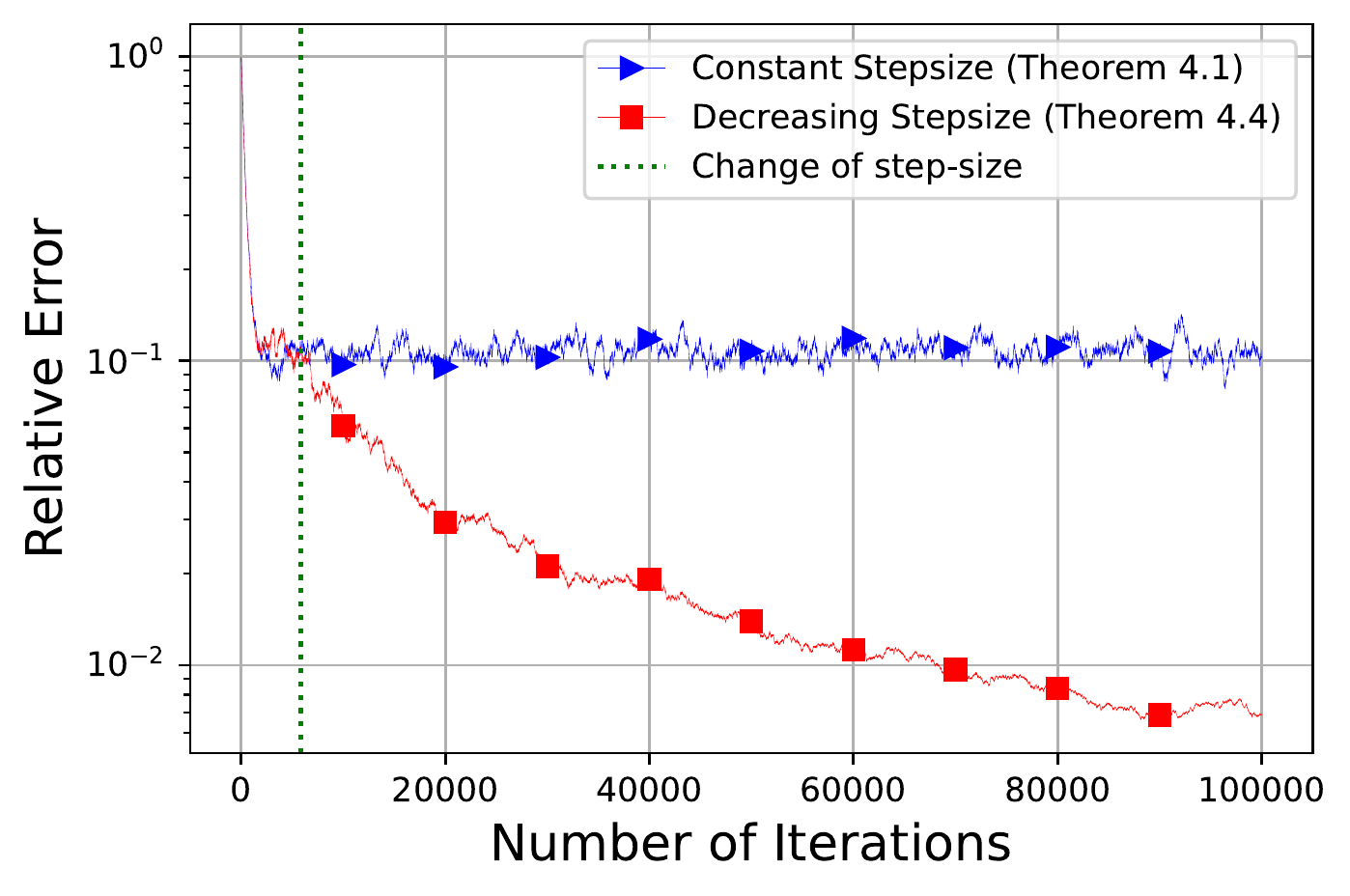}
\caption{Comparison of the constant step-size rule~\eqref{eq:constant_stepsize} with the switching step-sizes~\eqref{eq:stepsize_switching_1} on the strongly monotone quadratic game.}\label{fig: Synthetic Dataset 2}
\end{figure}

\paragraph{On Weak Minty VIPs.}
In this experiment, we generate $\xi_i, \zeta_i$ such that $\frac{1}{n}\sum_{i = 1}^n \xi_i = \sqrt{63}$ and $\frac{1}{n} \sum_{i = 1}^n \zeta_i = - 1$. This choice of $\xi_i, \zeta_i$ ensures that $L = 8$ and $\rho = \nicefrac{1}{32}$ for the min-max problem we considered in Section \ref{sec: Experiment on WMVI}. In Fig. \ref{fig:performance of SPEG on WMVI_b}, we again implement the \algname{SPEG} on \eqref{asoxasl} with batchsize = $0.15 \times n$ (different batchsize compare to the plot of the main paper). 
 \begin{figure}[h]
     \centering
    \includegraphics[width=.48\textwidth]{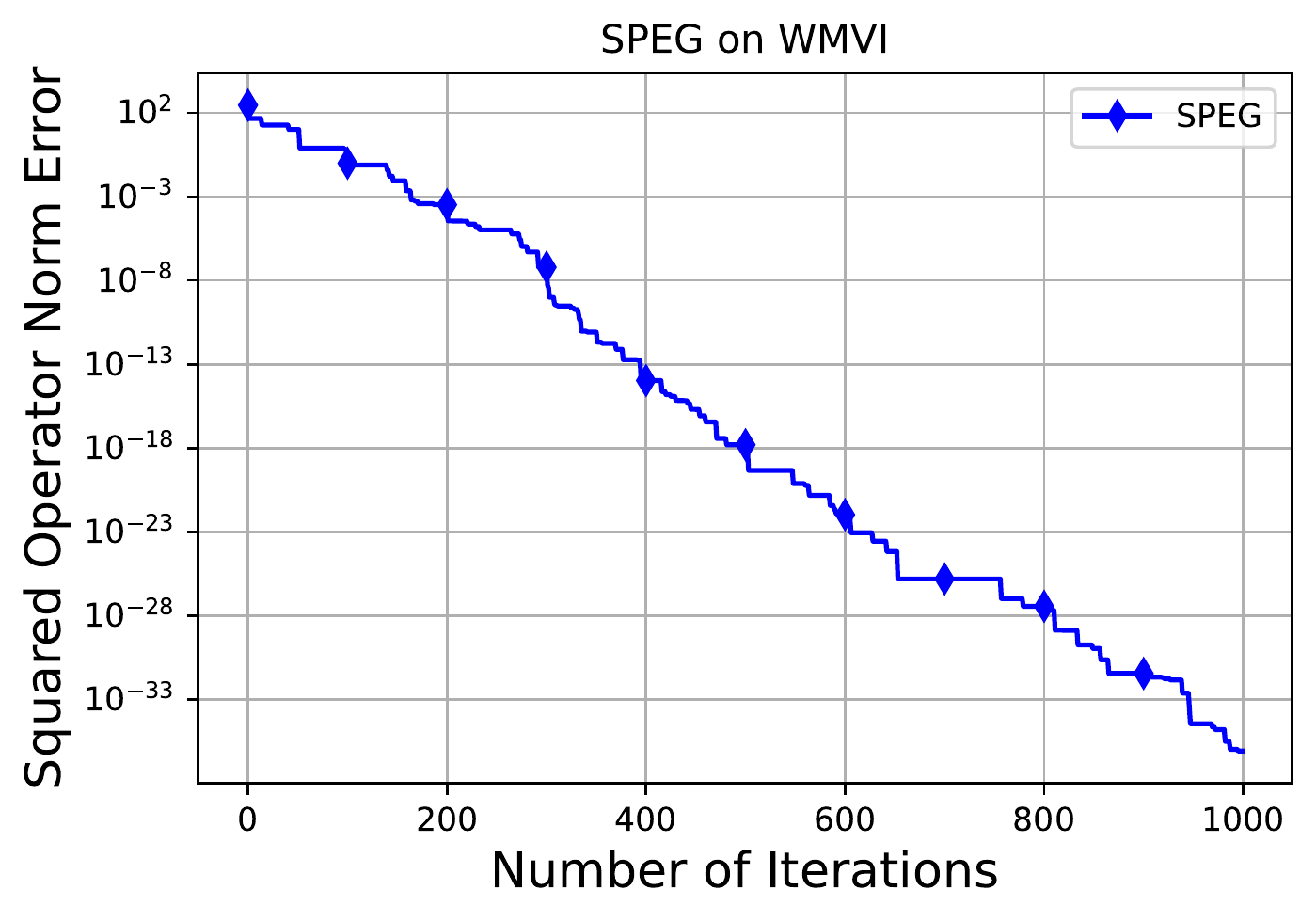}
    \caption{Trajectory of \algname{SPEG} for solving weak MVI using a batchsize = $0.15 \times n$.}
    \label{fig:performance of SPEG on WMVI_b}
\end{figure}

\subsection{Additional Experiments}\label{subsec:ComparisonConstantvsSwitching}
In this subsection, we include more experiments to evaluate the performance of \algname{SPEG} on quasi-strongly monotone and weak MVI problems. 
First, we run the experiment comparing constant and switching step-size rules on a different setup than the one we included in the main paper to analyze the performance of \algname{SPEG} under different condition numbers. Then, we implement \algname{SPEG} on the weak MVI of \eqref{asoxasl}. To evaluate the performance in this experiment, we plot $\nicefrac{\|F(\hat{x}_k)\|^2}{\|F(x_0)\|^2}$ on the $y$-axis.
\subsubsection{Strongly Monotone Quadratic Game:} In this experiment, we compare the proposed constant step-size~\eqref{eq:constant_stepsize} and the switching step-size rule~\eqref{eq:stepsize_switching_1}. We implement our algorithm on operator $F: \R^{4} \to \R^{4}$ given by 
\begin{equation*}
   F(x) \eqdef \frac{1}{3} \left(M_1(x - x_1^*) + M_2 (x - x_2^*) + M_3 (x - x_3^*)\right), 
\end{equation*}
where $M_1$, $M_2$ and $M_3$ are the diagonal matrices, 
\begin{eqnarray*}
    M_1 = \begin{pmatrix}
        \Delta & & &\\
        & 1 & & \\
        & & 1 & \\
        & & & 1
    \end{pmatrix}, \quad M_2 = \begin{pmatrix}
        1 & & &\\
        & \Delta & & \\
        & & 1 & \\
        & & & 1
    \end{pmatrix}, \quad M_3 = \begin{pmatrix}
        1 & & &\\
        & 1 & & \\
        & & \Delta & \\
        & & & 1
    \end{pmatrix}
\end{eqnarray*}
and
\begin{eqnarray*}
    x_1^* = \begin{pmatrix}
    \Delta \\
    0 \\
    0 \\
    \Delta
    \end{pmatrix}, \quad x_2^* = \begin{pmatrix}
    0 \\
    \Delta \\
    0 \\
    0
    \end{pmatrix}, \quad x_3^* = \begin{pmatrix}
    0 \\
    0 \\
    \Delta \\
    0
    \end{pmatrix}. 
\end{eqnarray*}
This choice of $M_i$ and $x_i^*$ ensures that the Lipschitz constant of operator $F$ is $\frac{\Delta + 2}{3}$ while quasi-strong monotonicity parameter~\eqref{eq: Strong Monotonicity} is $\mu = 1$. Hence the condition number of $F$ is given by $\frac{\Delta + 2}{3}$. This allows us to vary the condition number of operator $F$ by changing the value of $\Delta$. For Fig. \ref{fig:condition_no_1} we take $\Delta = 3$ (condition number = $1.67$) while for Fig. \ref{fig:condition_no_10} we choose $\Delta = 10$ (condition number = $10.67$). The vertical dotted line in plots of Fig.~\ref{fig:constant_vs_switch_exp_2} marks the transition point from constant to switching step-size rule as predicted by our theoretical result in Theorem~\ref{SPEG switching rule}.

\begin{figure}[h]
\centering
\begin{subfigure}[b]{.48\textwidth}
    \centering
    \includegraphics[width=\textwidth]{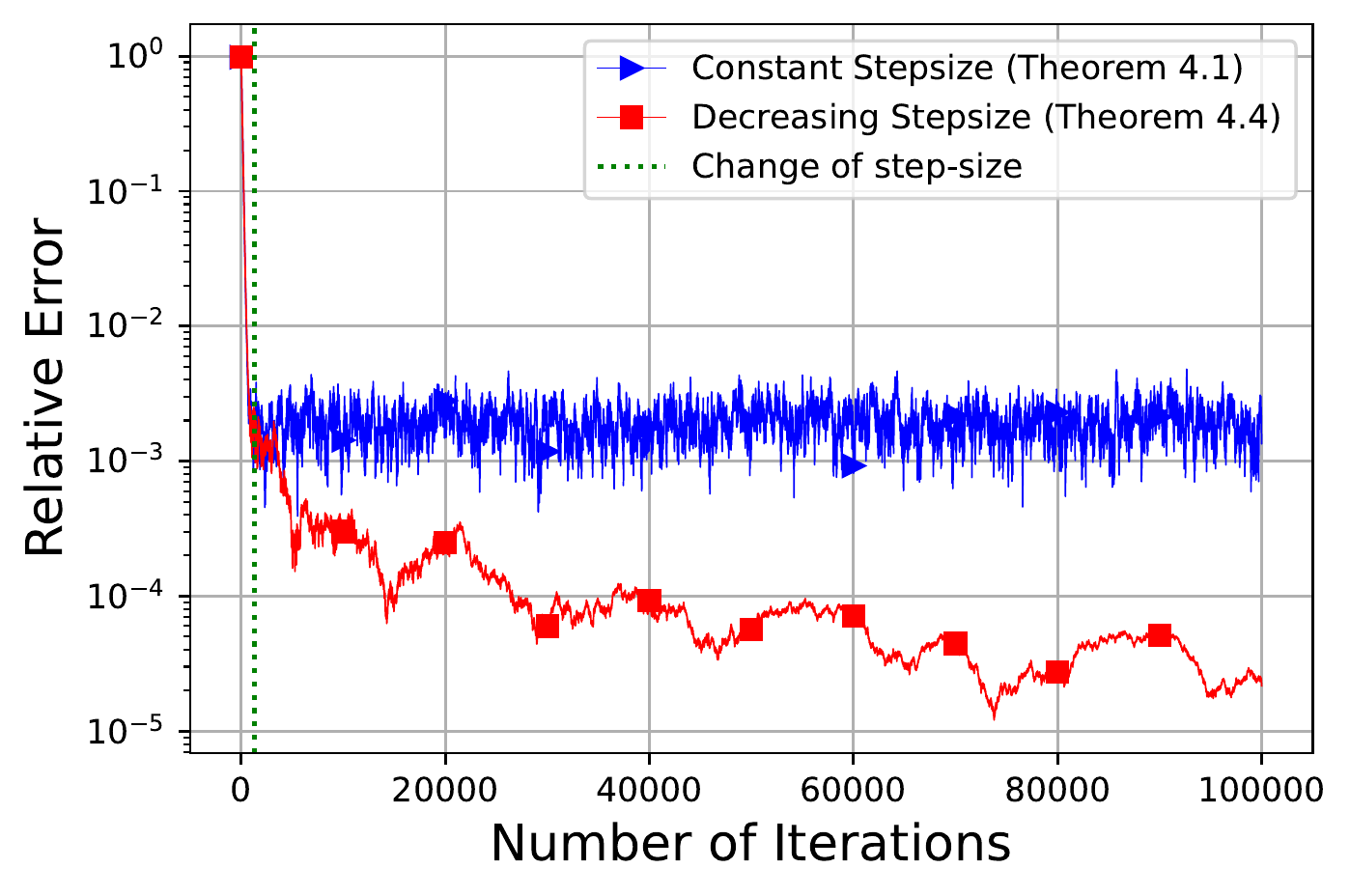}
    \caption{Condition Number $\frac{L}{\mu} = 1.67$.}\label{fig:condition_no_1}
\end{subfigure}
\begin{subfigure}[b]{0.48\textwidth}
    \centering
    \includegraphics[width=\textwidth]{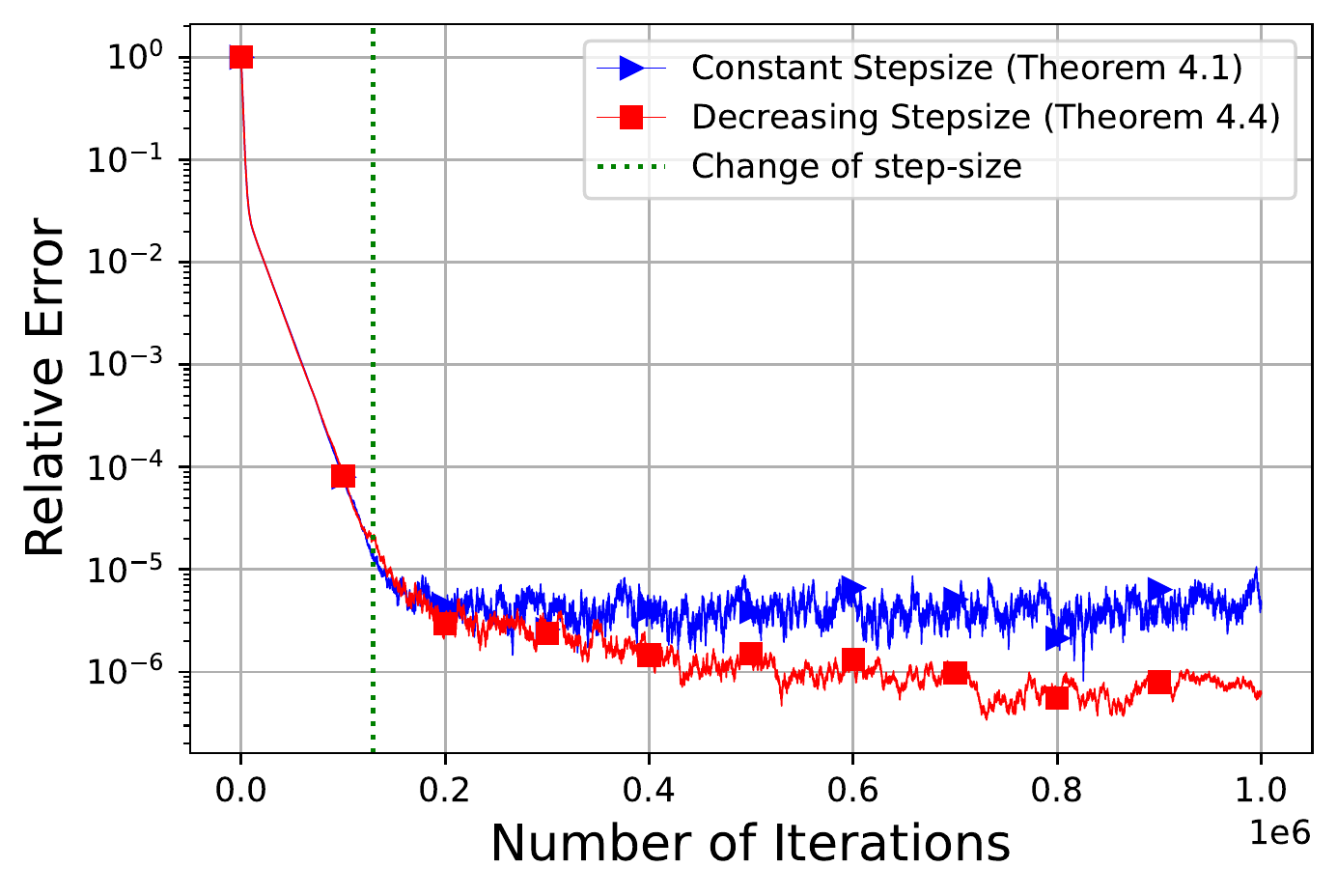}
    \caption{Condition Number $\frac{L}{\mu} = 10.67$.}\label{fig:condition_no_10}
\end{subfigure}
    \caption{\emph{Illustration of switching rule~\eqref{eq:stepsize_switching_1} in Theorem \ref{SPEG switching rule}. The dotted line marks the transition from phase 1 (where we use constant step-size) to phase 2 (where we use decreasing step-size).}}
\label{fig:constant_vs_switch_exp_2}
\end{figure}

\subsubsection{Weak Minty VIPs Continued}
\label{subsec:WeakMVI}
In this experiment, we reevaluate the performance of \algname{SPEG} on weak MVI example of \eqref{asoxasl}. That is, we generate the data in exactly the same way as the ones in section \ref{sec: Experiment on WMVI} with $n = 100$. In Fig. \ref{fig:batchsize=0.1} and \ref{fig:batchsize=0.15}, we implement \algname{SPEG} with batchsize $10$ and $15$, respectively (we note that in this setting the full-gradient evaluation requires a batchsize of $100$). For these plots, we use the relative operator norm on the $y$-axis, i.e. $\nicefrac{\|F(\hat{x}_k)\|^2}{\|F(x_0)\|^2}$, where $x_0$ denotes the starting point of \algname{SPEG}. As expected, the plots illustrate that \algname{SPEG} performs better as we increase the batchsize. From Fig.~\ref{fig:SPEG_weak_MVI_operator_norm} it is clear that with batchsize 15 \algname{SPEG} reaches an accuracy close to $10^{-10}$ while when we use a batchsize of $10$ for the same number of iterations we are only able to converge to an accuracy of $10^{-4}$. 
\begin{figure}[H]
\centering
\begin{subfigure}[b]{.48\textwidth}
    \centering
    \includegraphics[width=\textwidth]{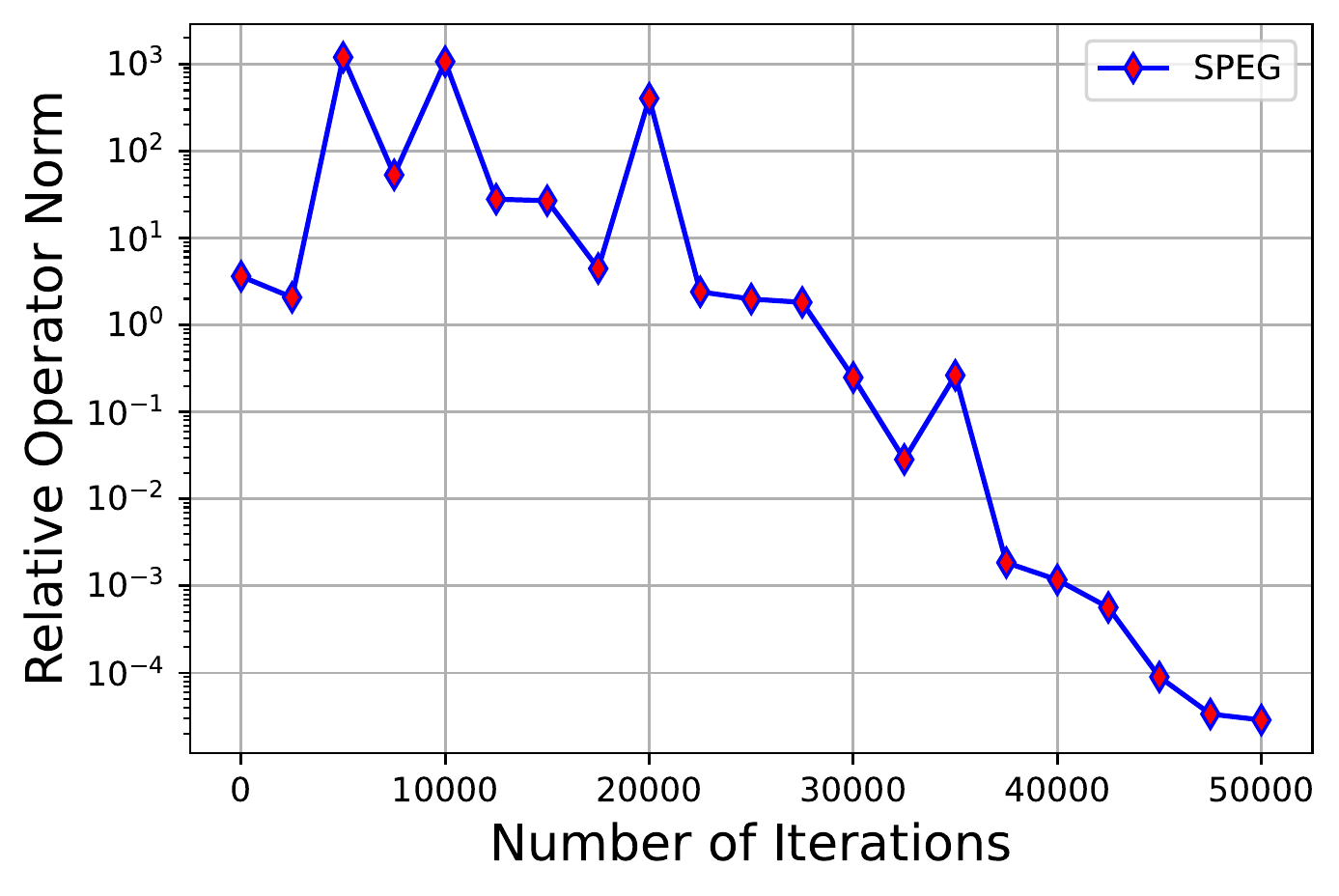}
    \caption{Batchsize = $0.1 \times n.$}\label{fig:batchsize=0.1}
\end{subfigure}
\begin{subfigure}[b]{0.48\textwidth}
    \centering
    \includegraphics[width=\textwidth]{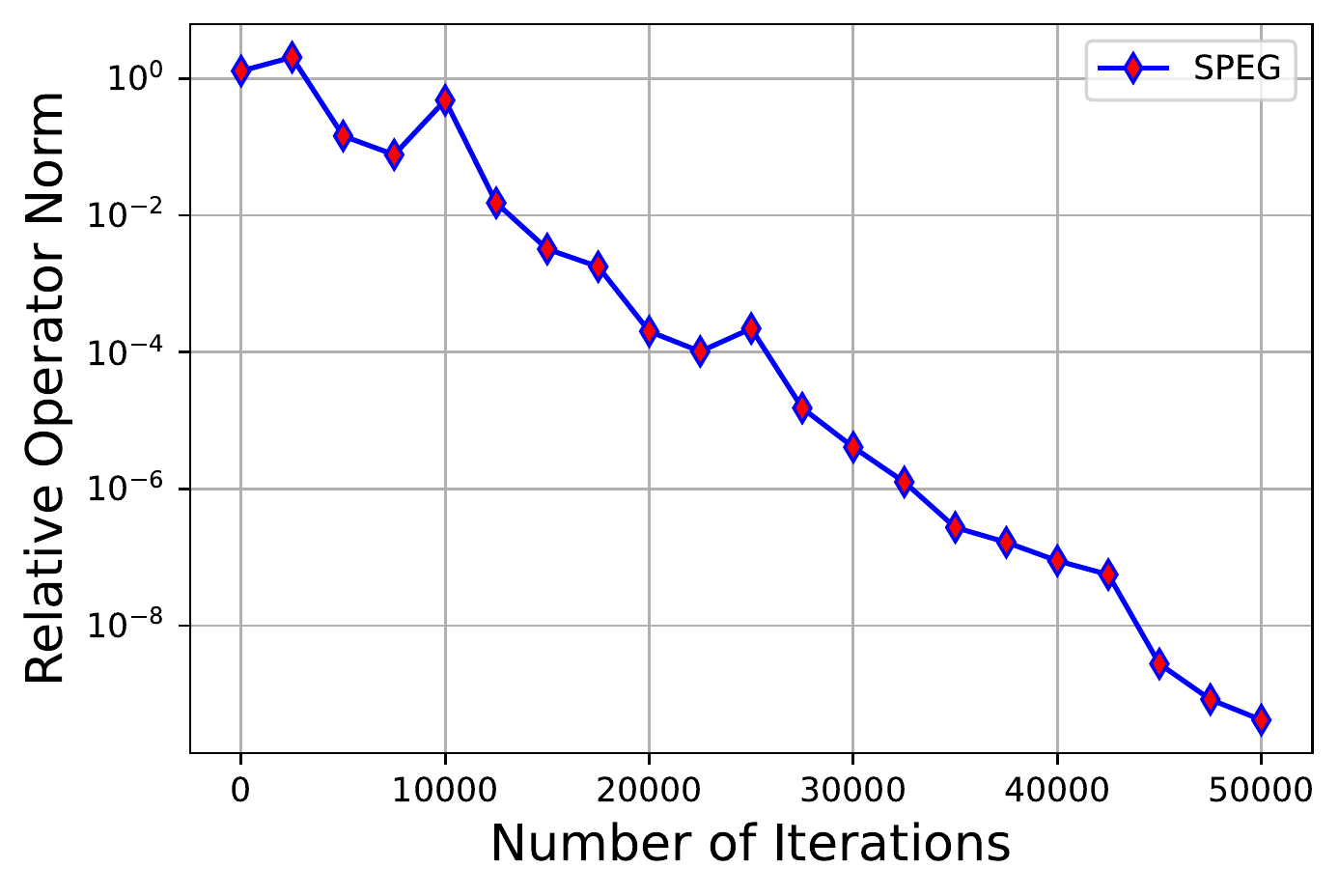}
    \caption{Batchsize = $0.15 \times n.$}\label{fig:batchsize=0.15}
\end{subfigure}
    \caption{\emph{Performance of \algname{SPEG} for solving weak MVI with different batchsizes. In plot (a) we use a batchsize of $10$ while in plot (b) we use $15$.}}
\label{fig:SPEG_weak_MVI_operator_norm}
\end{figure}
\end{document}

%% file: macros.tex
\usepackage{amsmath}
\usepackage{amssymb}
\usepackage{mathtools}
\usepackage{amsthm}
\usepackage{nicefrac}
\usepackage{multirow}
\usepackage{tikz-cd} 
\usepackage{subcaption}
\usepackage{multicol,caption}

\usepackage{wrapfig}

\usepackage[flushleft]{threeparttable} 
\usepackage{colortbl}
\definecolor{bgcolor}{rgb}{0.8,1,1}
\definecolor{bgcolor2}{rgb}{0.8,1,0.8}
\definecolor{niceblue}{rgb}{0.0,0.19,0.56}

\usepackage{hyperref}
\hypersetup{colorlinks,linkcolor={blue},citecolor={niceblue},urlcolor={blue}}

\usepackage{pifont}
\definecolor{PineGreen}{RGB}{0,110,51}
\definecolor{BrickRed}{RGB}{143,20,2}
\newcommand{\cmark}{{\color{PineGreen}\ding{51}}}%
\newcommand{\xmark}{{\color{BrickRed}\ding{55}}}%

\usepackage{thmtools}

\newcommand{\cD}{{\cal D}}
\newcommand{\Exp}{\mathbb{E}}
\newcommand{\Prob}[1]{\mathbb{P} \left[ #1\right]}

 \usepackage[disable,textsize=tiny]{todonotes}

\providecommand{\mycomment}[3]{\todo[inline,caption={},color=#3!20]{\textbf{#1: }#2}}%
\providecommand{\inlinecomment}[3]{{\color{#1}#2: #3}}
\newcommand\commenter[2]%
{%
  \expandafter\newcommand\csname i#1\endcsname[1]{\inlinecomment{#2}{#1}{##1}}
  \expandafter\newcommand\csname #1\endcsname[1]{\mycomment{#1}{##1}{#2}}
}

\commenter{cqmo}{red}
\commenter{jhjd}{blue}
\commenter{hnb}{orange}
\commenter{Nicolas}{yellow}
\commenter{Eduard}{green}
\commenter{Sayantan}{white}

\definecolor{shadecolor}{gray}{0.9}
\declaretheoremstyle[
headfont=\normalfont\bfseries,
notefont=\mdseries, notebraces={(}{)},
bodyfont=\normalfont,
postheadspace=0.5em,
spaceabove=1pt,
mdframed={
  skipabove=8pt,
  skipbelow=8pt,
  hidealllines=true,
  backgroundcolor={shadecolor},
  innerleftmargin=4pt,
  innerrightmargin=4pt}
]{shaded}

\declaretheorem[style=shaded,within=section]{definition}
\declaretheorem[style=shaded,sibling=definition]{theorem}
\declaretheorem[style=shaded,sibling=definition]{proposition}
\declaretheorem[style=shaded,sibling=definition]{assumption}
\declaretheorem[style=shaded,sibling=definition]{corollary}

\declaretheorem[style=shaded,sibling=definition]{lemma}

\usepackage[capitalize,noabbrev]{cleveref}

\usepackage{xspace}

\newcommand{\algname}[1]{{\sf  #1}\xspace}

\newcommand{\R}{\mathbb{R}}

\newcommand{\E}{\mathbb{E}}
\newcommand{\cO}{{\cal O}}

\newcommand{\hx}{\hat{x}}
\newcommand{\la}{\left\langle}
\newcommand{\g}{\gamma}
\newcommand{\ra}{\right\rangle}

\newcommand{\om}{\omega}
\newcommand{\eqdef}{:=}

\DeclarePairedDelimiter\ceil{\lceil}{\rceil}